\documentclass{article}

\usepackage[utf8]{inputenc}
\usepackage[english]{babel}

\usepackage{geometry}
\usepackage{fancyhdr}
\usepackage{titlesec}
\usepackage{indentfirst}
\usepackage{float}

\usepackage{amsmath,comment}
\usepackage{amssymb}
\usepackage{mathrsfs}
\usepackage{amsthm}
\usepackage{amsfonts}

\usepackage[linktoc=all]{hyperref}
\usepackage{cleveref}
\usepackage{ifthen}
\usepackage{tikz}

\usepackage{array}
\newcolumntype{C}[1]{>{\centering\arraybackslash}p{#1}}

\theoremstyle{plain}
\newtheorem{thm}{Theorem}[section]
\newtheorem{prop}[thm]{Proposition}
\newtheorem{lem}[thm]{Lemma}
\newtheorem{cor}[thm]{Corollary}
\newtheorem{defn}[thm]{Definition}
\newtheorem{question}[thm]{Question}

\newtheorem{introthm}{Theorem}

\theoremstyle{remark}
\newtheorem{ex}[thm]{Example}
\newtheorem{rmk}[thm]{Remark}
\newtheorem{remark}[thm]{Remark}

\newtheorem*{remark*}{Remark}

\DeclareMathOperator{\rank}{rk}

\newcommand{\bA}{\mathbf{A}}

\newcommand{\bR}{\mathbf{R}}

\newcommand{\bZ}{\mathbf{Z}}

\newcommand{\ba}{\mathbf{a}}
\newcommand{\bb}{\mathbf{b}}
\newcommand{\bc}{\mathbf{c}}
\newcommand{\bd}{\mathbf{d}}
\newcommand{\be}{\mathbf{e}}

\newcommand{\bm}{\mathbf{m}}

\newcommand{\bq}{\mathbf{q}}
\newcommand{\br}{\mathbf{r}}
\newcommand{\bs}{\mathbf{s}}

\newcommand{\bu}{\mathbf{u}}
\newcommand{\bv}{\mathbf{v}}
\newcommand{\bw}{\mathbf{w}}
\newcommand{\bx}{\mathbf{x}}
\newcommand{\by}{\mathbf{y}}
\newcommand{\bz}{\mathbf{z}}
\newcommand{\bepsilon}{\mathbf{\epsilon}}

\newcommand{\bbC}{\mathbb{C}}
\newcommand{\bbN}{\mathbb{N}}
\newcommand{\bbQ}{\mathbb{Q}}

\newcommand{\bbZ}{\mathbb{Z}}

\newcommand{\cF}{\mathcal{F}}
\newcommand{\cG}{\mathcal{G}}

\newcommand{\cL}{\mathcal{L}}
\newcommand{\cM}{\mathcal{M}}
\newcommand{\cP}{\mathcal{P}}

\newcommand{\sgr}{\le}
\newcommand{\rar}{\rightarrow}
\newcommand{\Rar}{\Rightarrow}
\newcommand{\xrar}[1]{\xrightarrow{#1}}
\newcommand{\ol}[1]{\overline{#1}}
\newcommand{\ot}[1]{\widetilde{#1}}
\newcommand{\abs}[1]{\left|#1\right|}
\newcommand{\divides}{\bigm|}
\newcommand{\gen}[1]{\langle #1 \rangle}

\newcommand{\intsup}[1]{\lceil #1 \rceil}

\newcommand{\pA}{\mathbf{A}^+}
\newcommand{\supp}[1]{\mathrm{supp}(#1)}

\newcommand{\cnj}{\sim_{\mathrm{c}}}
\newcommand{\lecnj}{\preceq_{\mathrm{c}}}
\newcommand{\qcnj}{\sim_{\mathrm{qc}}}
\newcommand{\qcsupp}[1]{\mathrm{supp}_\mathrm{qc}(#1)}
\newcommand{\qcmin}[1]{\mathrm{min}_\mathrm{qc}(#1)}
\newcommand{\cla}[1]{\mathrm{lin}_{\mathrm{qc}}({#1})}

\newcommand{\ee}{\sim_{\mathrm{ee}}}
\newcommand{\leee}{\preceq_{\mathrm{ee}}}
\newcommand{\qee}{\sim_{\mathrm{qee}}}
\newcommand{\qeesupp}[1]{\mathrm{supp}_\mathrm{qee}(#1)}
\newcommand{\qeemin}[1]{\mathrm{min}_\mathrm{qee}(#1)}
\newcommand{\ela}[1]{\mathrm{lin}_{\mathrm{qee}}({#1})}

\newcommand{\cnjedges}[2]{\mathrm{E}_{\mathrm{c}}({#1},{#2})}
\newcommand{\qcnjedges}[2]{\mathrm{E}_{\mathrm{qc}}({#1},{#2})}

\newcommand{\xdash}[1][3em]{\rule[0.5ex]{#1}{0.55pt}}
\newcommand{\edgedash}[1]{\ \xdash[#1]\ }
\newcommand{\edge}{\edgedash{0.6cm}}

\newcommand{\mcd}[1]{\mathrm{gcd}(#1)}

\newcommand{\gencol}[1]{\text{span}(#1)}
\newcommand{\FG}[1]{\mathrm{FG}(#1)}

\newcommand{\floating}{\text{floating}}

\newcommand{\showcommentsbox}{yes}

\newsavebox{\commentbox}
{\ifthenelse{\equal{\showcommentsbox}{yes}}%
{\footnotemark
        \begin{lrbox}{\commentbox}
        \begin{minipage}[t]{1.25in}\raggedright\sffamily\tiny
        \footnotemark[\arabic{footnote}]}
{\begin{lrbox}{\commentbox}}}%
{\ifthenelse{\equal{\showcommentsbox}{yes}}%
{\end{minipage}\end{lrbox}\marginpar{\usebox{\commentbox}}}
{\end{lrbox}}}

\title{On the isomorphism problem for generalized Baumslag-Solitar groups: invariants and flexible configurations}

\author{
Dario Ascari\\
{\small \textit{Department of Mathematics, University of the Basque Country,}}\\
{\small \textit{Barrio Sarriena, Leioa, 48940, Spain}}\\
{\small e-mail: \texttt{ascari.maths@gmail.com}}\\
\and
Montserrat Casals-Ruiz\\
{\small \textit{Ikerbasque - Basque Foundation for Science and Matematika Saila,}}\\
{\small \textit{UPV/EHU,  Sarriena s/n, 48940, Leioa - Bizkaia, Spain}}\\
{\small e-mail: \texttt{montsecasals@gmail.com}}\\
\and
Ilya Kazachkov\\
{\small \textit{Ikerbasque - Basque Foundation for Science and Matematika Saila,}}\\
{\small \textit{UPV/EHU,  Sarriena s/n, 48940, Leioa - Bizkaia, Spain}}\\
{\small e-mail: \texttt{ilya.kazachkov@gmail.com}}
}

\date{}

\begin{document}

\maketitle

\begin{abstract}
We prove that the isomorphism problem is decidable for generalized Baumslag-Solitar (GBS) groups  with one quasi-conjugacy class and full support gaps. In order to do so we introduce a family of invariants that fully characterize the isomorphism within this class of GBSs.
\end{abstract}



\section{Introduction}

The isomorphism problem, formulated by Dehn around the turn of the 20th century, is a fundamental decision problem in group theory. It asks whether there exists an algorithm that, given two groups defined by finite presentations, decides whether or not the groups are isomorphic. Although this problem is undecidable in general, it remains of fundamental interest to solve it for specific, relevant classes of groups.

One of the most important decidability results of the isomorphism problem was established in the context of 3-manifolds and relies on the fact that (fundamental groups of) closed orientable 3-manifolds have a unique decomposition as fundamental groups of graphs of groups with abelian edge groups. These ideas led to a rich and fruitful theory, known as the theory of JSJ decompositions for finitely presented groups \cite{RS97,DS99,FP06,GL17}. As in the case of 3-manifolds, one-ended hyperbolic groups have a unique JSJ-decomposition (even though the construction is more delicate, based on the tree of cylinders, see \cite{GL11}), and this fact is again central for solving the isomorphism problem for hyperbolic groups \cite{Sel95,DG11,DT19}.

In contrast, for other families of groups, the JSJ decomposition is far from being unique and can exhibit significantly greater flexibility: not necessarily in the vertex groups of the decomposition, but at least in the way in which they are glued together. In \cite{For02}, Forester explained how one can go from one decomposition of a group to any other using a certain set of moves, called \textit{elementary deformations}. In this way, given a splitting $T$ of a group $G$ (i.e. an action of $G$ on a tree $T$), the \textit{deformation space} is the family of all splittings of $G$ that are related to $T$ by elementary deformations. Understanding the structure of the infinite deformation spaces represents one of the main obstructions to solving the isomorphism problem (and to describing the automorphism group) for large families of groups. In this regard, the most basic and important example is the class of generalized Baumslag-Solitar groups.

A \textbf{generalized Baumslag-Solitar group} (GBS) is the fundamental group of a graph of groups where all vertex groups and edge groups are infinite cyclic. They can be characterized as the only groups of cohomological dimension $2$ containing an infinite cyclic subgroup that intersects all of its conjugates, see \cite{Kro90}. As the name suggests, GBS groups generalize the classical Baumslag–Solitar groups, which are a rich source of (counter)examples in geometric group theory. They also play a key role in the study of one-relator groups: Gersten conjectured that the presence of a GBS as a subgroup is the only obstruction to hyperbolicity in this class.  Although GBS groups were classified up to quasi-isometry nearly 25 years ago (see \cite{FM98, Why01}), their classification up to isomorphism remains a longstanding open problem.

\vspace{0.5cm}

In \cite{ACK-iso1} the authors initiated a systematic study of the isomorphism problem for GBS groups, see references therein for previously known results. They proved that for any two isomorphic GBS groups there exists an explicit (computable) bound on the number of vertices and edges in the graph of groups decompositions appearing along a sequence of moves that realizes the isomorphism. More precisely, they showed that for any GBS group, one can algorithmically compute a fully reduced representative, i.e. an isomorphic GBS group with the minimal number of vertices and edges in its isomorphism class. Moreover, for any pair of fully reduced isomorphic GBS groups, there exists a sequence of new moves that realizes the isomorphism while preserving the number of vertices and edges. As a consequence, the minimal number of vertices and edges becomes a computable invariant of the isomorphism class.

In the current paper, we continue this study, introduce several new \textbf{isomorphism invariants} for GBS groups, and prove that these invariants suffice to fully determine the isomorphism class within a large and natural subclass of GBS groups: those with one quasi-conjugacy class and full-support gaps, see Definitions \ref{def:conjugacy} and \ref{def:full-support-gaps}. The main result of this paper is summarized in the following

\begin{introthm}[\Cref{cor:isomorphism-oneqcc-fsgaps}]\label{introthm:isomorphism-oneqcc-fsgaps}
    There is an algorithm that, given two GBS graphs $(\Gamma,\psi),(\Delta,\phi)$ with one qc-class and full-support gaps, decides whether or not the corresponding GBS groups are isomorphic.
\end{introthm}

The new isomorphism invariants - the \textbf{quasi-conjugacy classes and their linear invariants}, the set of \textbf{rigid vectors}, the \textbf{assignment map}, and the \textbf{limit angles} - play a crucial role in understanding the algebraic structure of GBS groups and in addressing the isomorphism problem. We next discuss these invariants.

\paragraph{Linear invariants and independence of quasi-conjugacy classes.}

Each GBS group is stratified by the set of quasi-conjugacy classes of its elliptic elements, see \Cref{def:conjugacy}, which, as its name indicates, generalizes the notion of conjugacy classes. The relevant set of quasi-conjugacy classes (those containing at least one edge in the affine representation) forms a finite partially ordered set and provides a natural notion of complexity for a GBS group.

Each quasi-conjugacy class can be described by a finite, algorithmically computable set of data: the set of \textit{minimal regions}, the \textit{quasi-conjugacy support}, and the \textit{linear algebra} (\Cref{def:qcsupp}). It follows that the quasi-conjugacy classes, as well as the number of edges that belong to each of them, are computable isomorphism invariants. Moreover, each quasi-conjugacy class (along with its set of edges) admits a partition into a collection of conjugacy classes, each associated with a subset of edges. We show that this partition, the set of conjugacy classes and their corresponding set of edges, is also a computable isomorphism invariant. We refer to these invariants collectively as the \textbf{linear invariants} of the quasi-conjugacy class.

\begin{introthm}[\Cref{cor:basic-invariants}]
    Let $(\Gamma,\psi),(\Gamma',\psi')$ be two GBS graphs and suppose that there is a sequence of slides, swaps, and connections going from one to the other {\rm(}preserving the number of vertices{\rm)}. Then $(\Gamma,\psi), (\Gamma',\psi')$ have the same quasi-conjugacy classes and linear invariants.
\end{introthm}

We next examine the interaction between different quasi-conjugacy classes in the context of the isomorphism problem, and show that, to a significant extent, these classes can be analyzed independently. More precisely, we categorize slide moves into two types: \textit{internal slides}, which involve edges within the same quasi-conjugacy class, and \textit{external slides}, which involve edges from different classes. We prove that performing external slides requires only to know the linear invariants of the corresponding quasi-conjugacy class. As a consequence, we deduce that edges in distinct quasi-conjugacy classes interact (at most) through external slides.

\begin{introthm}[\Cref{thm:independence-qc-classes}]\label{introthm:independence-qc-classes}
    Let $(\Gamma,\psi),(\Gamma',\psi')$ be two GBS graphs with a bijection $V(\Gamma)=V(\Gamma')$. Then the following are equivalent:
    \begin{enumerate}
        \item There is a sequence of slides, swaps, and connections going from $(\Gamma,\psi)$ to $(\Gamma',\psi')$ inducing the given bijection.
        \item For every quasi-conjugacy class $Q$, there is a sequence of swaps, connections, internal slides, and external slides {\rm(\Cref{lem:external-slide})} --- each of them involving only edges in $Q$ --- going from the configuration in $(\Gamma,\psi)$ to the configuration in $(\Gamma',\psi')$.
    \end{enumerate}
\end{introthm}

As an application of \Cref{introthm:independence-qc-classes}, we recover the main results of \cite{For06} and \cite{Dud17}, and resolve the isomorphism problem for a concrete example that remained open in \cite{Wan25}.

\paragraph{Invariants for GBSs with one quasi-conjugacy class and full-support gaps.} 

As discussed above, the poset of quasi-conjugacy classes provides a measure of complexity of a GBS group. However, \Cref{introthm:independence-qc-classes} shows that the general case can be reduced to the case of a single quasi-conjugacy class considered relative to the linear invariants of the other classes involved in external slides. We address this central case by focusing on GBS groups with \textit{one quasi-conjugacy class}, that is, GBS groups in which all edges belong to the same quasi-conjugacy class.

We will impose an additional condition on the GBS groups under consideration, namely, that they have \textit{full-support gaps} (see \Cref{def:full-support-gaps}). This can be viewed as a  ``genericity" condition: a ``random" element of $\bbN^n$ is expected to have all its components strictly positive. Informally, the benefit of this assumption is that the conditions for applying a connection move reduce to membership in some positive orthant of $\bbN^n$; in the general case, by contrast, these conditions are constrained to a positive orthant of a potentially proper subvariety of $\bbN^n$.

For GBS groups with one quasi-conjugacy class and full-support gaps for which the number of edges is strictly greater than the number of minimal points, we prove that the linear invariants completely determine the isomorphism class. Within this subclass, we define a (coarse) normal form and prove that two such groups are isomorphic if and only if their normal forms are isomorphic (see Proposition \ref{prop:uniqueness-normal-form-floating}). The normal form is constructed by maximizing the interaction among edges, and we show that any GBS graph satisfying these conditions can be algorithmically transformed into this ``maximal interaction" normal form.

When the number of edges equals the number of minimal points, we need to introduce two new invariants: the multi-set of \textit{rigid edges} and the \textit{assignment map}.

Some edges in GBS graphs are \textit{rigid}, meaning they remain fixed under all allowed moves and thus constitute an isomorphism invariant. When a quasi-conjugacy class contains more than one edge, however, the definition of rigidity has to be slightly relaxed: rigid edges permit a very limited form of modification and can, in some cases, be concealed within cycles. This motivates the introduction of rigid vectors (see \Cref{def:rigid-vector}), which formalize the naive notion of rigidity.

The non-rigid edges define a tuple of vectors in an abelian group, and performing moves on the GBS graph, changes the tuple of vectors by Nielsen transformations. Thus we obtain the second isomorphism invariant, the \textit{Nielsen equivalence class} of the tuple of vectors obtained from the non-rigid edges. A key step in our argument is proving the converse: if two configurations yield Nielsen equivalent tuples of vectors, then there is a sequence of moves going from one to another, i.e. the associated GBS groups are isomorphic (see \Cref{sec:Nielsen-equiv-big}).

Under the assumption that the number of edges and minimal points is the same, there is a natural bijection between these two sets. While even permutations of the tuple of edges result in isomorphic GBS groups, odd permutations do not. To account for this, we introduce the assignment map, which records the positions of the rigid and non-rigid edges relative to the minimal points and allows us to track the sign of the permutation.

Finally, there is an exceptional case, when the number of edges interacting with each other is exactly two. In this setting, the results from \Cref{sec:Nielsen-equiv-big} do not apply, as they require at least three interacting vectors. Instead, a surprising and previously unexplored invariant emerges: the \textit{limit angle}. The analysis of limit angles reveals a rich and intricate dynamical structure, which we examine in detail in a separate paper \cite{ACK-iso3}.

\begin{remark*}
Nielsen equivalence in abelian groups is discussed in \Cref{sec:Nielsen-equiv-abelian}. Suppose that we are given an abelian group $A$ with minimal number of generators $k$. Then every two $k'$-tuples of generators for $A$ are Nielsen equivalent if $k'\ge k+1$. If we are given two $k$-tuples of generators for $A$, then we can compute the determinant of the change of basis from one $k$-tuple to the other: this must be considered modulo some number $d=d(A)$. The two $k$-tuples are Nielsen equivalent if and only if the determinant is $1$ modulo $d$. This determinant is thus an isomorphism invariant for GBSs in some cases, and one can produce examples of GBSs which are non-isomorphic exactly because they have different determinants (see Examples \ref{ex10}, \ref{ex5} and \ref{ex8}).
\end{remark*}

\subsection*{Acknowledgements}

This work was supported by the Basque Government grant IT1483-22. The second author was supported by the Spanish Government grant PID2020-117281GB-I00, partly by the European Regional Development Fund (ERDF), the MICIU /AEI /10.13039/501100011033 / UE grant PCI2024-155053-2.

\section{GBS graphs and quasi-conjugacy classes}

In this section, we establish the notation used throughout the paper. We review the concepts of conjugacy and quasi-conjugacy classes, and show how they can be described by a finite amount of data. We also provide explicit algorithms to compute them, and examples to illustrate them.

\subsection{Graphs of groups}

We consider graphs as combinatorial objects, following the notation of \cite{Ser77}. A \textbf{graph} is a quadruple $\Gamma=(V,E,\ol{\cdot},\iota)$ consisting of a set $V=V(\Gamma)$ of \textit{vertices}, a set $E=E(\Gamma)$ of \textit{edges}, a map $\ol{\cdot}:E\rar E$ called \textit{reverse} and a map $\iota:E\rar V$ called \textit{initial vertex}; we require that, for every edge $e\in E$, we have $\ol{e}\not=e$ and $\ol{\ol{e}}=e$. For an edge $e\in E$, we denote with $\tau(e)=\iota(\ol{e})$ the \textit{terminal vertex} of $e$. A \textbf{path} in a graph $\Gamma$, with \textit{initial vertex} $v\in V(\Gamma)$ and \textit{terminal vertex} $v'\in V(\Gamma)$, is a sequence $\sigma=(e_1,\dots,e_\ell)$ of edges $e_1,\dots,e_\ell\in E(\Gamma)$ for some integer $\ell\ge0$, with the conditions $\iota(e_1)=v$ and $\tau(e_\ell)=v'$ and $\tau(e_i)=\iota(e_{i+1})$ for $i=1,\dots,\ell-1$.  The terminal and initial vertices will often be referred to as \textit{endpoints}. A graph is \textbf{connected} if for every pair of vertices, there is a path going from one to the other. For a connected graph $\Gamma$, we define its \textbf{rank} $\rank{\Gamma}\in\bbN\cup\{+\infty\}$ as the rank of its fundamental group (which is a free group).

\begin{defn}
A \textbf{graph of groups} is a quadruple
$$\cG=(\Gamma,\{G_v\}_{v\in V(\Gamma)},\{G_e\}_{e\in E(\Gamma)},\{\psi_e\}_{e\in E(\Gamma)})$$
consisting of a connected graph $\Gamma$, a group $G_v$ for each vertex $v\in V(\Gamma)$, a group $G_e$ for every edge $e\in E(\Gamma)$ with the condition $G_e=G_{\ol{e}}$, and an injective homomorphism $\psi_e:G_e\rar G_{\tau(e)}$ for every edge $e\in E(\Gamma)$.
\end{defn}

Let $\cG=(\Gamma,\{G_v\}_{v\in V(\Gamma)},\{G_e\}_{e\in E(\Gamma)},\{\psi_e\}_{e\in E(\Gamma)})$ be a graph of groups. Define the \textbf{universal group} $\FG{\cG}$ as the quotient of the free product $(*_{v\in V(\Gamma)}G_v)*F(E(\Gamma))$ by the relations
\begin{equation*}\label{FGrelations}
\ol{e}=e^{-1} \qquad\qquad \psi_{\ol{e}}(g)\cdot e=e\cdot\psi_e(g)
\end{equation*}
for $e\in E(\Gamma)$ and $g\in G_e$.

Define the \textbf{fundamental group} $\pi_1(\cG,\ot{v})$ of a graph of group $\cG$ with basepoint $\ot{v}\in V(\Gamma)$ to be the subgroup of $\FG{\cG}$ of the elements that can be represented by words such that, when going along the word, we read a path in the graph $\Gamma$ from $\ot v$ to $\ot v$. The fundamental group $\pi_1(\cG,\ot v)$ does not depend on the chosen basepoint $\ot v$, up to isomorphism. Given an element $g$ inside a vertex group $G_v$, we can take a path $(e_1,\dots,e_\ell)$ from $\ot v$ to $v$: we define the conjugacy class $[g]:=[e_1\dots e_\ell g\ol{e}_\ell\dots \ol{e}_1]\in\pi_1(\cG,\ot v)$, and notice that this does not depend on the chosen path.

\subsection{Generalized Baumslag-Solitar groups}

\begin{defn}
A \textbf{GBS graph of groups} is a finite graph of groups
$$\cG=\left(\Gamma,\{G_v\}_{v\in V(\Gamma)},\{G_e\}_{e\in E(\Gamma)},\{\psi_e\}_{e\in E(\Gamma)}\right)$$
such that each vertex group and each edge group is $\bbZ$.
\end{defn}

A \textbf{Generalized Baumslag-Solitar group} is a group $G$ isomorphic to the fundamental group of some GBS graph of groups.

\begin{defn}
A \textbf{GBS graph} is a pair $(\Gamma,\psi)$, where $\Gamma$ is a finite graph and $\psi:E(\Gamma)\rar\bbZ\setminus\{0\}$ is a function.
\end{defn}

Given a GBS graph of groups $\cG=(\Gamma,\{G_v\}_{v\in V(\Gamma)},\{G_e\}_{e\in E(\Gamma)},\{\psi_e\}_{e\in E(\Gamma)})$, the map $\psi_e:G_e\rar G_{\tau(e)}$ is an injective homomorphism $\psi_e:\bbZ\rar\bbZ$, and thus coincides with multiplication by a unique non-zero integer $\psi(e)\in\bbZ\setminus\{0\}$. We define the GBS graph associated to $\cG$ as $(\Gamma,\psi)$ associating to each edge $e$ the factor $\psi(e)$ characterizing the homomorphism $\psi_e$, see Figure \ref{fig:GBS-graph}. Giving a GBS graph of groups is equivalent to giving the corresponding GBS graph. In fact, the numbers on the edges are sufficient to reconstruct the injective homomorphisms and thus the graph of groups.

\begin{figure}[H]
\centering
\includegraphics[scale=1]{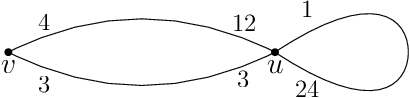}
\caption{In the figure we can see a GBS graph $(\Gamma,\psi)$ with two vertices $v,u$ and three edges $e_1,e_2,e_3$ (and their reverses). The edge $e_1$ goes from $v$ to $u$ and has $\psi(\ol{e}_1)=4$ and $\psi(e_1)=12$. The edge $e_2$ goes from $v$ to $u$ and has $\psi(\ol{e}_2)=\psi(\ol{e}_2)=3$. The edge $e_3$ goes from $u$ to $u$ and has $\psi(\ol{e}_3)=1$ and $\psi(e_3)=24$.}
\label{fig:GBS-graph}
\end{figure}

Let $\cG$ be a GBS graph of groups and let $(\Gamma,\psi)$ be the corresponding GBS graph. The universal group $\FG{\cG}$ has a presentation with generators $V(\Gamma)\cup E(\Gamma)$, the generator $v\in V(\Gamma)$ representing the element $1$ in $\bbZ=G_v$. The relations are given by $\ol{e}=e^{-1}$ and $u^{\psi(\ol{e})}e=ev^{\psi(e)}$ for every edge $e\in E(\Gamma)$ with $\iota(e)=u$ and $\tau(e)=v$.

\subsection{Reduced affine representation of a GBS graph}

\begin{defn}\label{def:set-of-primes}
    For a GBS graph $(\Gamma,\psi)$, define its \textbf{set of primes}
    $$\cP(\Gamma,\psi):=\{r\in\bbN \text{ prime } : r\divides \psi(e) \text{ for some } e\in E(\Gamma)\}.$$
\end{defn}

Given a GBS graph $(\Gamma,\psi)$, consider the finitely generated abelian group
$$\bA:=\bbZ/2\bbZ\oplus\bigoplus\limits_{r\in\cP(\Gamma,\psi)}\bbZ.$$
We denote with $\mathbf{0}\in\bA$ the neutral element. For an element $\ba=(a_0,a_r : r\in\cP(\Gamma,\psi))\in\bA$ (with $a_0\in\bbZ/2\bbZ$ and $a_r\in\bbZ$ for $r\in\cP(\Gamma,\psi)$), we denote $\ba\ge\mathbf{0}$ if $a_r\ge 0$ for all $r\in\cP(\Gamma,\psi)$; notice that we are not requiring any condition on $a_0$. We define the positive cone $\pA:=\{\ba\in\bA : \ba\ge\mathbf{0}\}$.

\begin{defn}
Let $(\Gamma,\psi)$ be a GBS graph. Define its \textbf{{\rm(}reduced{\rm)} affine representation} to be the graph $\Lambda=\Lambda(\Gamma,\psi)$ given by:
\begin{enumerate}
\item $V(\Lambda)=V(\Gamma)\times\pA$ is the disjoint union of copies of $\pA$, one for each vertex of $\Gamma$.
\item $E(\Lambda)=E(\Gamma)$ is the same set of edges as $\Gamma$, and with the same reverse map.
\item For an edge $e\in E(\Lambda)$ we write the unique factorization $\psi(e)=(-1)^{a_0}\prod_{r\in\cP(\Gamma,\psi)}r^{a_r}$ and we define the terminal vertex $\tau_\Lambda(e)=(\tau_\Gamma(e),(a_0,a_r,\dots))$, see {\rm Figure \ref{fig:aff-rep}}.
\end{enumerate}
For a vertex $v\in V(\Gamma)$ we denote $\pA_v:=\{v\}\times\pA$ the corresponding copy of $\pA$.
\end{defn}

If $\Lambda$ contains an edge going from $p$ to $q$, then we write $p\edge q$. If $\Lambda$ contains edges from $p_i$ to $q_i$ for $i=1,\dots,m$, then we write
$$
\begin{cases}
p_1\edge q_1\\
\cdots\\
p_m\edge q_m
\end{cases}
$$
This does not mean that $p_1\edge q_1,\dots,p_m\edge q_m$ are all the edges of $\Lambda$, but only that in a certain situation we are focusing on those edges. If we are focusing on a specific copy $\pA_v$ of $\pA$, for some $v\in V(\Gamma)$, and we have edges $(v,\ba_i)\edge (v,\bb_i)$ for $i=1,\dots,m$, then we say that $\pA_v$ contains edges
$$
\begin{cases}
\ba_1\edge \bb_1\\
\cdots\\
\ba_m\edge \bb_m
\end{cases}
$$
omitting the vertex $v$.

\begin{figure}[H]
\centering
\includegraphics[width=0.7\textwidth]{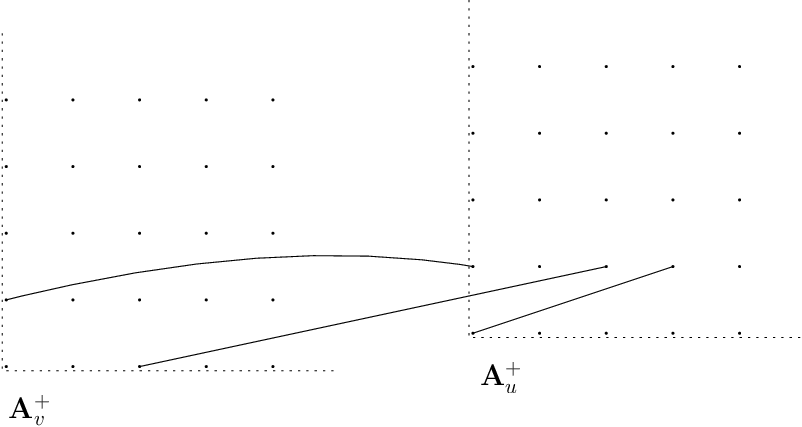}
\caption{The affine representation $\Lambda$ of the GBS graph $(\Gamma,\psi)$ of Figure \ref{fig:GBS-graph}. The set of vertices consists of two copies $\pA_v$ and $\pA_u$ of the positive affine cone $\pA$, associated to the two vertices $v$ and $u$ respectively. The edge $e_1$ going from $v$ to $u$ was labeled with $(\psi(\ol{e}_1),\psi(e_1))=(4,12)=(2^2 3^0,2^2 3^1)$, and thus now it goes from the point $(2,0)$ in $\pA_v$ to the point $(2,1)$ in $\pA_u$. Similarly for $e_2$ and $e_3$.}
\label{fig:aff-rep}
\end{figure}

\subsection{Support and control of vectors}

The following notions for elements of $\bA$ will be widely used throughout the paper.

\begin{defn}\label{def:support}
For $\bx\in\bA$ define its \textbf{support} as the set
$$\supp{\bx}:=\{r\in\cP(\Gamma,\psi) : x_r\not=0\}.$$
\end{defn}
\begin{rmk}
Note that we omit the $\bbZ/2\bbZ$ component from the definition of support.
\end{rmk}

\begin{defn}\label{def:control-a}
Let $\ba,\bb,\bw\in\pA$. We say that $\ba,\bw$ \textbf{controls} $\bb$ if any of the following equivalent conditions holds:
\begin{enumerate}
\item We have $\ba\le \bb\le \ba+k\bw$ for some $k\in\bbN$.
\item We have $\bb-\ba\ge\mathbf{0}$ and $\supp{\bb-\ba}\subseteq\supp{\bw}$.
\end{enumerate}
\end{defn}

\subsection{Affine paths and conjugacy classes}\label{sec:conjugacy-classes}

Let $(\Gamma,\psi)$ be a GBS graph and let $\Lambda$ be its affine representation. Given a vertex $p=(v,\ba)\in V(\Lambda)$ and an element $\bw\in\pA$, we define the vertex $p+\bw:=(v,\ba+\bw)\in V(\Lambda)$. For two vertices $p,p'\in V(\Lambda)$ we denote $p'\ge p$ if $p'=p+\bw$ for some $\bw\in\pA$; in particular this implies that both $p,p'$ belong to the same $\pA_v$ for some $v\in V(\Gamma)$.

\begin{defn}
An \textbf{affine path} in $\Lambda$ with \textit{initial vertex} $p\in V(\Lambda)$ and \textit{terminal vertex} $p'\in V(\Lambda)$ is a sequence $(e_1,\dots,e_\ell)$ of edges $e_1,\dots,e_\ell\in E(\Lambda)$ for some $\ell\ge0$, such that there exist $\bw_1,\dots,\bw_\ell\in\pA$ satisfying the conditions $\iota(e_1)+\bw_1=p$ and $\tau(e_\ell)+\bw_\ell=p'$ and $\tau(e_i)+\bw_i=\iota(e_{i+1})+\bw_{i+1}$ for $i=1,\dots,\ell-1$.
\end{defn}

The elements $\bw_1,\dots,\bw_\ell$ are called \textit{translation coefficients} of the path; if they exist, then they are uniquely determined by the path and by the endpoints, and they can be computed algorithmically. They mean that an edge $e\in E(\Lambda)$ connecting $p$ to $q$ allows us also to travel from $p+\bw$ to $q+\bw$ for every $\bw\in\pA$.

\begin{defn}\label{def:conjugacy}
    Let $p,q\in V(\Lambda)$.
    \begin{enumerate}
        \item We denote $p\cnj q$, and we say that $p,q$ are \textbf{conjugate}, if there is an affine path going from $p$ to $q$.
        \item We denote $p\lecnj q$ if $p\le q'$ for some $q'\cnj q$.
        \item We denote $p\qcnj q$, and we say that $p,q$ are \textbf{quasi-conjugate}, if $p\lecnj q$ and $q\lecnj p$.
    \end{enumerate}
\end{defn}

The relation $\cnj$ is an equivalence relations on the set $V(\Lambda)$. The relation $\lecnj$ is a pre-order on $V(\Lambda)$, and $\qcnj$ is the equivalence relation induced by the pre-order. Note that if $p\cnj p'$, then $p+\bw\cnj p'+\bw$ for all $\bw\in\pA$. Similarly, if $p\qcnj p'$, then $p+\bw\qcnj p'+\bw$ for all $\bw\in\pA$.

\subsection{Minimal points, quasi-conjugacy support, linear algebra}\label{sec:qc-support}

We now show how the conjugacy and quasi-conjugacy classes can be completely described by three pieces of data: the \textit{minimal points}, the \textit{quasi-conjugacy support} and the \textit{linear algebra}.

\begin{defn}\label{def:qcsupp}
    Let $(\Gamma,\psi)$ be a GBS graph and let $p$ be a vertex of its affine representation $\Lambda$.
    \begin{enumerate}
        \item Define the \textbf{minimal points}
        $$\qcmin{p}=\{q\in V(\Lambda) : q\qcnj p \text{ and for all } q'\qcnj p \text{ with } q'\le q \text{ we have } q\le q'\}.$$
        \item Define the \textbf{quasi-conjugacy support}
        $$\qcsupp{p}=\bigcup\{\supp{\bu} : \bu\in\pA \text{ and } p+\bu\qcnj p\}\subseteq\cP(\Gamma,\psi).$$
        \item Define the \textbf{linear algebra}
        \[
            \text{$\cla{p}=\{\br\in\bA : \br=\bw-\bw'$ for some $\bw,\bw'\in\pA$ such that $p+\bw\cnj p+\bw' \cnj p\}$.}
        \]
    \end{enumerate}
\end{defn}

\begin{lem}\label{lem:qcmin-finite}
    The set $\qcmin{p}$ is finite.
\end{lem}
\begin{proof}
    Fix a vertex $u\in V(\Gamma)$: since $\Gamma$ has finitely many vertices, it suffices to prove that $\qcmin{p}\cap\pA_u$ is finite. Consider a set of variables $\{x_i\}_{i\in\cP(\Gamma,\psi)}$ and consider the ring of polynomials $\bbC[x_i : i\in\cP(\Gamma,\psi)]$. For each $(u,\bq)\qcnj p$ consider the monomial $x^\bq=\prod x_i^{q_i}\in\bbC[x_i,\dots]$. Consider the ideal $I=(x^\bq : (u,\bq)\qcnj p)\subseteq\bbC[x_i,\dots]$ generated by these monomials. We have that $I$ is a monomial ideal, i.e. it contains a polynomial $P$ if and only if it contains all the monomials with non-zero coefficients of $P$. By Hilbert's basis theorem, $I$ must be finitely generated, and thus we can find finitely many $\bq_1,\dots, \bq_\ell\in\pA$ such that $I=(x^{\bq_1},\dots,x^{\bq_\ell})$. It follows that $\qcmin{p}\cap\pA_u\subseteq\{(u,\bq_1),\dots,(u,\bq_\ell)\}$ must be finite, as desired.
\end{proof}

\begin{lem}
    If $p\qcnj q$, then $\qcsupp{p}=\qcsupp{q}$.
\end{lem}
\begin{proof}
    If $p+\bu\qcnj p$ for some $\bu\in\pA$, then $q+\bu\qcnj p+\bu\qcnj p\qcnj q$.
\end{proof}

\begin{lem}\label{lem:qc-supp-single-vector}
    For all $p\in V(\Lambda)$ there is $\bw\in\pA$ such that $p+\bw\cnj p$ and $\supp{\bw}=\qcsupp{p}$.
\end{lem}
\begin{proof}
    Realize the quasi-conjugacy support as a finite union $\qcsupp{p}=\supp{\bu_1}\cup\ldots\cup\supp{\bu_\ell}$ for $\bu_i\in\pA$ with $p+\bu_i\qcnj p$. By the definition of quasi-conjugacy, we can find $\bv_i\in\pA$ such that $p+\bu_i+\bv_i\cnj p$. Now set $\bw=(\bu_1+\bv_1)+\dots+(\bu_\ell+\bv_\ell)$ and the statement follows.
\end{proof}

\begin{lem}\label{lem:cla-subgroup}
    In the above notation, $\cla{p}$ is a subgroup of $\bA$.
\end{lem}
\begin{proof}
    Suppose that we are given $\br,\bs\in\cla{p}$. Thus, we have $\br=\bw-\bw'$ and $\bs=\bu-\bu'$ for $\bw,\bw',\bu,\bu'\in\pA$ with $p\cnj p+\bw\cnj p+\bw'\cnj p+\bu\cnj p+\bu'$. In particular we have that $p+\bw+\bu\cnj p+\bw'+\bu\cnj p+\bw'+\bu'$ and $(\bw+\bu)-(\bw'+\bu')=\br+\bs$. Thus, $\br+\bs\in\cla{p}$. The statement follows.
\end{proof}

\begin{lem}[Equivalent characterization of linear algebra]\label{lem:linear-algebra}
    For $p\in V(\Lambda)$ and $\br\in\bA$, we have that $\br\in\cla{p}$ if and only it satisfies the following two conditions:
    \begin{enumerate}
        \item\label{itm:la-1} $\Lambda$ contains vertices $(u,\ba)\cnj(u,\ba+\br)\qcnj p$ for some $u\in V(\Gamma)$ and $\ba\in\bA$ with $\ba,\ba+\br\ge\mathbf{0}$.
        \item\label{itm:la-2} $\supp{\br}\subseteq\qcsupp{p}$.
    \end{enumerate}
\end{lem}
\begin{proof}
    The definition of $\cla{p}$ obviously implies the two conditions. 
    
    Suppose now that $\supp{\br}\subseteq\qcsupp{p}$ and there are vertices $(u,\ba)\cnj(u,\ba+\br)\qcnj p$. Since $\supp{\br}\subseteq\qcsupp{p}$, we can find $\bw\in\pA$ such that $p+\bw\cnj p$ and $\bw+\br\ge\mathbf{0}$. Since $p\qcnj (u,\ba)$ we can find $p'\cnj p$ with $(u,\ba)\le p'$; in particular we also have $p+\bw\cnj p'+\bw$ and $p+(\bw-\br)\cnj p'+(\bw-\br)$. But since $(u,\ba)\le p'+\bw$, the condition $(u,\ba)\cnj(u,\ba+\br)$ implies $p'+\bw\cnj p+\bw+\br$. Therefore we have $p+\bw\cnj p'+\bw\cnj p'+(\bw-\br)\cnj p+(\bw-\br)$. The statement follows.
\end{proof}

\begin{rmk}
    Condition \ref{itm:la-2} of Lemma \ref{lem:linear-algebra} can not be omitted. In fact, one might try to define the linear algebra of a point $p$ as the set of elements $\br\in\bA$ satisfying Condition \ref{itm:la-1} of Lemma \ref{lem:linear-algebra}. However, in general this definition gives a set which is different from the set $\cla{p}$ of Definition \ref{def:qcsupp}, and which might not be a subgroup.
\end{rmk}

It follows that if $p\qcnj q$, then $\qcmin{p}=\qcmin{q}$ and $\qcsupp{p}=\qcsupp{q}$ and $\cla{p}=\cla{q}$; thus, for a quasi-conjugacy class $Q$, the objects $\qcmin{Q},\qcsupp{Q},\cla{Q}$ are well-defined. The following Proposition \ref{prop:description-qc-class} explains how the minimal points and the quasi-conjugacy support give a complete description of a quasi-conjugacy class. The subsequent Proposition \ref{prop:description-c-class} explains how the linear algebra gives a complete description of the conjugacy classes contained in the quasi-conjugacy class.

\begin{prop}\label{prop:description-qc-class}
    Let $(\Gamma,\psi)$ be a GBS graph and let $p$ be a vertex of its affine representation $\Lambda$. Then the following are equivalent:
    \begin{enumerate}
        \item $q\qcnj p$
        \item $q=m+\bw$ for some $m\in\qcmin{p}$ and for some $\bw\in\pA$ with $\supp{\bw}\subseteq\qcsupp{p}$.
    \end{enumerate}
\end{prop}
\begin{proof}
    Immediate from the definitions.
\end{proof}

\begin{prop}\label{prop:description-c-class}
    Let $(\Gamma,\psi)$ be a GBS graph and let $p$ be a vertex of its affine representation $\Lambda$. Then we have the following:
    \begin{enumerate}
        \item\label{itm:c-class-above-p} For all $\bw,\bw'\in\pA$ with $\supp{\bw},\supp{\bw'}\subseteq\qcsupp{p}$, we have that
        $$p+\bw\cnj p+\bw' \Leftrightarrow \bw-\bw'\in\cla{p}.$$
        \item\label{itm:lin-alg-natural-inclusion} There is a natural inclusion of subgroups $\cla{p} \sgr \bbZ/2\bbZ \oplus \bbZ^{\qcsupp{p}} \sgr \bA$.
        \item\label{itm:c-classes-bijection} There is a bijection between the following two sets:
        
        (i) The set of conjugacy classes contained in the quasi-conjugacy class of $p$.
        
        (ii) The quotient $(\bbZ/2\bbZ\oplus\bbZ^{\qcsupp{p}})/\cla{p}$.
    \end{enumerate}
\end{prop}
\begin{rmk}
    The bijection in Item \ref{itm:c-classes-bijection} of Proposition \ref{prop:description-c-class} is not canonical. It is canonical only up to translations in $(\bbZ/2\bbZ\oplus\bbZ^{\qcsupp{p}})/\cla{p}$.
\end{rmk}
\begin{proof}
    Item \ref{itm:c-class-above-p} follows from Lemma \ref{lem:linear-algebra} and Item \ref{itm:lin-alg-natural-inclusion} is trivial.

    Take a conjugacy class $Q$ contained in the quasi-conjugacy class of $p$. This means that we can find $q\in Q$ such that $q=p+\bw$ for some $\bw\in\pA$ with $\supp{\bw}\subseteq\qcsupp{p}$. By Lemma \ref{lem:linear-algebra}, a different choice of $q\in Q$ changes $\bw$ by adding an element of $\cla{p}$. Thus from $Q$ we obtain a well-defined coset $\bw+\cla{p}$. This defines a map from the set of conjugacy classes contained in the quasi-conjugacy class of $p$ to $(\bbZ/2\bbZ\oplus\bbZ^{\qcsupp{p}})/\cla{p}$.

    To prove surjectivity, take $\bx\in\bA$ with $\supp{\bx}\subseteq\qcsupp{p}$ and consider the coset $\bx+\cla{p}$. Since $\supp{\bx}\subseteq\qcsupp{p}$, we can find $\bw\in\pA$ such that $p+\bw\cnj p$ and $\bx+\bw\ge\mathbf{0}$. In particular $\bx+\cla{p}=(\bx+\bw)+\cla{p}$ and this coset is realized by the conjugacy class of $p+\bx+\bw$.

    To prove injectivity, suppose that we are given $\bw,\bw'\in\pA$ with $\supp{\bw},\supp{\bw'}\subseteq\qcsupp{p}$, and such that $p+\bw$ and $p+\bw'$ define the same coset $\bw+\cla{p}=\bw'+\cla{p}$. But then $\bw-\bw'\in\cla{p}=\cla{p}$ and we find $\bu,\bu'\in\pA$ such that $\bw-\bw'=\bu-\bu'$ and $p+\bu\cnj p+\bu'\cnj p$. This implies that $p+\bw\cnj p+\bw+\bu'=p+\bw'+\bu\cnj p+\bw'$, and injectivity follows.
\end{proof}

\subsection{The graph of families and the modular map}\label{sec:graph-of-families}

Let $(\Gamma,\psi)$ be a GBS graph and let $\Lambda$ be its affine representation.

\begin{defn}
    We say that $p,q\in V(\Lambda)$ are \textbf{in the same family} if $p\qcnj q$ and they are of the form $p=(v,\ba)$ and $q=(v,\bb)$ with $\supp{\bb-\ba}\subseteq\qcsupp{p}$, for some $v\in V(\Gamma)$ and $\ba,\bb\in\pA$.
\end{defn}

Being in the same family is an equivalence relation on $V(\Lambda)$; an equivalence class is called a \textbf{family} of vertices. Every family of vertices is contained in a quasi-conjugacy class (by definition), and it always contains at least one minimal point of the quasi-conjugacy class (by \Cref{prop:description-c-class}). In particular, each quasi-conjugacy class is partitioned into finitely many families.

\begin{defn}\label{def:graph-of-families}
    Define the \textbf{graph of families} $\cF(Q)$ of a quasi-conjugacy class $Q$ as follows:
    \begin{enumerate}
        \item $V(\cF(Q))$ is the finite set of families contained in the quasi-conjugacy class $Q$.
        \item $E(\cF(Q))$ is the set of triples $(F,e,F')\in V(\cF(Q))\times E(\Lambda)\times V(\cF(Q))$ such that there is a {\rm(}possibly trivial{\rm)} vector $\bw\in\pA$ with $\iota(e)+\bw\in F$ and $\tau(e)+\bw\in F'$.
        \item We set $\ol{(F,e,F')}=(F',\ol{e},F)$ and $\iota((F,e,F'))=F$ and $\tau((F,e,F'))=F'$.
    \end{enumerate}
\end{defn}

\begin{lem}\label{lem:family-unique-edge}
    Suppose that the vertices of $e\in E(\Lambda)$ lie in the quasi-conjugacy class $Q$. Then there are unique families $F,F'\in V(\cF(Q))$ such that $(F,e,F')\in E(\cF(Q))$, and they are the families containing $\iota(e)$ and $\tau(e)$ respectively.
\end{lem}
\begin{proof}
    If for some $\bw\in\pA$ the vertex $\iota(e)+\bw$ lies in $Q$, then by the definition of quasi-conjugacy support, we have $\supp{\bw}\subseteq\qcsupp{Q}$, and by the definition of family, we have that $\iota(e)+\bw$ lies in the same family as $\iota(e)$. Similarly for $\tau(e)$.
\end{proof}

If the vertices of $e\in E(\Lambda)$ do not lie in $Q$, then the graph of families can contain several edges of the form $(F,e,F')$ corresponding to different translates of $e$. However, it will always contain finitely many edges, as there are finitely many possible triples in $V(\cF(Q))\times E(\Lambda)\times V(\cF(Q))$. Note also that the graph of families is connected: for every two families in the same quasi-conjugacy class, there must be an affine path connecting two points in the two families.

\begin{defn}\label{defn:modular_map}
    Define the \textbf{modular map} $q:E(\Gamma)\rightarrow\bA$ given by $q(e)=\bb-\ba$, where $\iota_\Lambda(e)=(v,\ba)$ and $\tau_\Lambda(e)=(u,\bb)$ for some $v,u\in V(\Gamma)$ and $\ba,\bb\in\pA$.
\end{defn}

Note that $q(\ol{e})=-q(e)$. In particular, we obtain a homomorphism $q:\pi_1(\Gamma)\rightarrow\bA$, which we call the \textbf{modular homomorphism}.

Let $Q$ be a quasi-conjugacy class and let $\cF(Q)$ be the corresponding graph of families. Then we have a modular map $q_Q:E(\cF(Q))\rightarrow\bA$ defined by $q_Q((F,e,F'))=q(e)$. Thus we can define a homomorphism $q_Q:\pi_1(\cF(Q))\rightarrow \bA$, which we call the \textbf{modular homomorphism of $Q$} .

\begin{prop}\label{prop:cla-and-modular-homomorphism}
    Let $Q$ be a quasi-conjugacy class and let $\cF(Q)$ be the corresponding graph of families. Then the image of $q_Q:\pi_1(\cF(Q))\rightarrow\bA$ is exactly $\cla{Q}$.
\end{prop}
\begin{proof}
    Fix $p\in Q$ and take $\br\in\cla{p}$. Then there is an affine path $(e_1,\dots,e_\ell)$, with translation coefficients $\bw_1,\dots,\bw_\ell$, going from $\iota(e_1)+\bw_1=p+\bw$ to $\tau(e_\ell)+\bw_\ell=p+\bw'$ with $\bw,\bw'\in\pA$ and $\bw'-\bw=\br$. For all $i=1,\dots,\ell-1$ we have that $\tau(e_i)+\bw_i=\iota(e_{i+1})+\bw_{i+1}$ must lie in some family $F_i\subseteq Q$. Thus we obtain a path $\sigma=((F_1,e_1,F_2),\dots,(F_{\ell-1},e_\ell,F_1))$ in $\cF(Q)$. It is easy to check that $q_Q(\sigma)=\bw'-\bw=\br$.

    Conversely, suppose that we are given a closed path $\sigma=((F_1,e_1,F_2),\dots,(F_{\ell-1},e_\ell,F_1))$ in $\cF(Q)$, for some families $F_i\subseteq Q$ with $i=1,\dots,\ell$. Then we can find $\bw_i\in\pA$ such that $\iota(e_i)+\bw_i\in F_i$ and $\tau(e_i)+\bw_i\in F_{i+1}$. It is easy to see that we can adjust $\bw_i$ in such a way that $\tau(e_i)+\bw_i=\iota(e_{i+1})+\bw_{i+1}$, obtaining an affine path $(e_1,\dots,e_\ell)$ from $\iota(e_1)+\bw_1$ to $\tau(e_\ell)+\bw_\ell$. In particular, we can check that $q_Q(\sigma)=\bw_\ell-\bw_1\in\cla{p}$ as $\iota(e_1)+\bw_1$ and $\tau(e_\ell)+\bw_\ell$ lie in the same family $F_1$.

    This proves that the image of the homomorphism $q_Q:\pi_1(\cF(Q))\rar\bA$ is exactly $\cla{p}$, as desired.
\end{proof}

\begin{rmk}
    Applying the modular map $q$ to a single edge is not usually very meaningful. If the two endpoints of $e$ are different vertices of $\Gamma$, then the difference $\bb-\ba$ does not make much sense, as we are comparing powers of generators of different vertices. It only makes sense when applied to paths that begin and terminate at the same vertex.

    However, if we are interested in studying a certain quasi-conjugacy class $Q$, then we need to refine the modular homomorphism even more. For example, if the two endpoints of $e$ are in the same vertex of $\Gamma$, but lie in different families of $Q$, then the difference $\bb-\ba$ still needs to be ignored. This is because it will give some linear algebra which is ``external'' to the quasi-conjugacy class $Q$. This is why we consider the modular homomorphism $q_Q:\pi_1(\cF(Q))\rightarrow\bA$ instead.

    For example, in \cite{ACK-iso1} we describe a procedure to make a GBS graph into a one-vertex GBS graph. The procedure will add a lot of loops to $\pi_1(\Gamma)$, adding noise to the modular homomorphism $q:\pi_1(\Gamma)\rightarrow\bA$. However, the procedure will preserve the graphs of families and the modular homomorphisms associated with the quasi-conjugacy classes. Considering the graph of families makes the noise invisible to us.
\end{rmk}

\subsection{Algorithmic computation of conjugacy classes}

All the concepts discussed in the previous Section \ref{sec:qc-support} can be computed algorithmically, as we now explain.

\begin{prop}\label{prop:qc-algorithm}
There is an algorithm that, given a GBS graph $(\Gamma,\psi)$ and a vertex $p$ of its affine representation $\Lambda$, computes the following:
\begin{enumerate}
    \item $\qcmin{p}$.
    \item $\qcsupp{p}$.
    \item For every $m\in\qcmin{p}$, an affine path from $m$ to a point $\ge p$.
    \item An affine path from $p$ to $p+\bw$ for some $\bw\in\pA$ with $\supp{\bw}=\qcsupp{p}$.
\end{enumerate}
\end{prop}

\begin{proof}[Description of the algorithm]
To run the algorithm, we can ignore the component $\bbZ/2\bbZ$, as long as we remember to add it to all minimal points at the end of the procedure. If the GBS graph $(\Gamma,\psi)$ has negative labels, we replace them by their absolute values. Thus, below we work with $\bA=\bbZ^{\cP(\Gamma,\psi)}$. At each step, the algorithm keeps in memory
\begin{enumerate}
    \item A finite set of vertices $M\subseteq V(\Lambda)$.
    \item A finite set of components $S\subseteq\cP(\Gamma,\psi)$.
    \item\label{itm:sigmaq} For every $m\in M$, an affine path $\sigma_m$ from $m$ to $p+\bx_m$ for some $\bx_m\in\pA$ with $\supp{\bx_m}\subseteq S$.
    \item\label{itm:sigma} An affine path $\sigma$ from $p$ to $p+\bz$ for some $\bz\in\pA$ with $\supp{\bz}=S$.
\end{enumerate}

CLEAN UP ROUTINE. For every  distinct $m,m'\in M$  we check whether or not $m'\ge m$. If so, we write $m'=m+\br$ for some $\br\in\pA$ and we perform the following operations:
\begin{enumerate}
    \item We change $S$ into $S\cup\supp{\br}$
    \item We change the path $\sigma$ into $\sigma^d\ol{\sigma}_{m'}\sigma_m$ for some sufficiently large natural number $d$.
    \item We remove $m'$ from $M$, and we forget the path $\sigma_{m'}$.
\end{enumerate}
We reiterate until there is no $m,m'\in M$ so that $m'\ge m$.

MAIN ALGORITHM. We initialize $M=\{p\}$ with $\sigma_p$ trivial. We initialize $S=\emptyset$ with $\sigma$ trivial. Then we repeat the following iteration until either of the sets $M,S$ changes. 

\textbf{Iteration:} We write a list of all the pairs $(m,e)\in M\times E(\Lambda)$. For each pair $(m,e)$, we look for the minimum $\bw\in\pA$ such that $\iota(e)+\bw=m+\bs$ for some $\bs\in\pA$ with $\supp{\bs}\subseteq S$. If it does not exist, then we just ignore that pair, and go on examining the other pairs $(m,e)$. If it exists, then it is unique, and we proceed as follows:
\begin{enumerate}
    \item Suppose that there is no $m'\in M$ such that $\tau(e)+\bw\ge m'$. In this case we change $M$ into $M\cup\{\tau(e)+\bw\}$ and we set $\sigma_{\tau(e)+\bw}=\ol{e}\sigma_m$. Then we run the clean up routine, and we start a new iteration (i.e. we write again the list of all pairs $(m,e)\in M\times E(\Lambda)$ and so on).
    \item Suppose that there is $m'\in M$ such that $\tau(e)+\bw=m'+\br$ with $\br\in\pA$ and $\supp{\br}\not\subseteq S$. Then we change $S$ into $S\cup\supp{\bs'}$, and we change the path $\sigma$ into $\sigma^d\ol{\sigma}_me\sigma_{m'}$ for some big enough natural number $d$. Then we run the clean up routine, and we start a new iteration (i.e. we write again the list of all pairs $(m,e)\in M\times E(\Lambda)$ and so on).
    \item Suppose that there is $m'\in M$ such that $\tau(e)+\bw\ge m'$, and for each such $m'$ we have that $\tau(e)+\bw=m'+\br$ with $\supp{\br}\subseteq S$. Then we just ignore that pair, and proceed examining the other pairs $(m,e)$.
\end{enumerate}
When we run a full iteration without changing the sets $M,S$, then we exit the cycle. The algorithm terminates and outputs $\qcmin{p}=M$, $\qcsupp{p}=S$, the affine paths $\sigma_m$ for $m\in M$, the affine path $\sigma$.
\end{proof}

\begin{proof}[Proof that the algorithm works]
Observe that we can add elements to the set $S$ only finitely many times (since $S$ is a subset of a finite set). As in the proof of Lemma \ref{lem:qcmin-finite}, we consider the ring of polynomials $\bbC[x_i : i\in\cP(\Gamma,\psi)]$ and the ideal $(x^\bm : (u,\bm)\in M)$. This ideal strictly increases every time that the algorithm adds a new element to $M$; since rings of polynomial are Noetherian, this can happen only finitely many times. This proves that we can change the set $M$ only finitely many times, and thus the algorithm terminates.

It is easy to prove by induction that at each step of the algorithm $M$ is contained in the quasi-conjugacy class of $p$, $S\subseteq\qcsupp{p}$, the paths $\sigma_m$ for $m\in M$ has the property of \Cref{itm:sigmaq}, the path $\sigma$ has the property of \Cref{itm:sigma}. In particular, in the algorithm, when we substitute $\sigma$ by $\sigma^d\ol{\sigma}_me\sigma_{m'}$, we are always able to (algorithmically) find an integer $d$ such that the substitution gives another affine path starting at $p$.

When the algorithm terminates, consider the set of vertices $M^{+S}=\{m+\bw : m\in M$ and $\bw\in\pA$ with $\supp{\bw}\subseteq S\}$ and observe that $M^{+S}$ is contained in the quasi-conjugacy class of $p$. But since no pair $(m,e)$ changes the sets $M,S$, we have that no translate of any edge $e$ can be used to move from a point in $M^{+S}$ to a point outside $M^{+S}$. This proves that $M^{+S}$ is closed under quasi-conjugacy, and thus $M^{+S}$ is the quasi-conjugacy class of $p$. Since every time we add a vertex to $M$ we also run the clean up routine, it follows that $M=\qcmin{p}$. This also implies that $S=\qcsupp{p}$. We conclude that the output of the algorithm is correct.
\end{proof}

\begin{prop}\label{prop:c-algorithm}
    There is an algorithm that, given a GBS graph $(\Gamma,\psi)$ and a vertex $p$ of its affine representation $\Lambda$, computes the following:
    \begin{enumerate}
        \item A finite set of generators for $\cla{p}$.
        \item For each generator $\br$, an affine path $\sigma_\br$ satisfying condition {\rm\ref{itm:la-1}} of {\rm Lemma \ref{lem:linear-algebra}}.
    \end{enumerate}
\end{prop}
\begin{proof}
    Let $Q$ be the quasi-conjugacy class of $p$. By Proposition \ref{prop:qc-algorithm}, we can algorithmically compute $\qcsupp{Q}$ and the finite set $M=\qcmin{Q}$. We can also algorithmically compute the graph of families $\cF(Q)$ (see \Cref{def:graph-of-families}), as $V(\cF(Q))$ is a computable quotient of the finite set $M$ and the set of edges is a computable subset of the finite set $V(\cF(Q))\times E(\Gamma)\times V(\cF(Q))$. Moreover, we can give an algorithmic description of the modular homomorphism $q_Q:\pi_1(\cF(Q))\rightarrow\bA$: given a loop $\theta$ in $\cF(Q)$, we can algorithmically compute $q_Q(\theta)\in\bA$. In particular we can compute a set of generators for $\pi_1(\cF(Q))$ and their images through the homomorphism $q_Q$. According to \Cref{prop:cla-and-modular-homomorphism}, the images are the desired set of generators for $\cla{Q}$ and the loops in $\pi_1(\cF(Q))$ are the corresponding affine paths (for some suitable translation coefficients), as desired.
\end{proof}

\begin{cor}\label{cor:algorithmic-paths}
    There is an algorithm that, given a GBS graph $(\Gamma,\psi)$ and two vertices $p,q$ of its affine representation $\Lambda$ 
    \begin{itemize}
        \item Decides whether $p\qcnj q$, and in case they are, computes two affine paths, one from each of them to a vertex above the other.
        \item Decides whether $p\cnj q$, and in case they are, computes an affine path from one to the other.
    \end{itemize}
\end{cor}

\begin{remark}
Notice that two points are conjugate if and only if the corresponding elliptic elements in the GBS group are conjugate. Therefore, Corollary \ref{cor:algorithmic-paths} provides a self-contained algorithm to solve the conjugacy problem for elliptic elements in a GBS group. For hyperbolic elements the conjugacy problem follows essentially from Britton's lemma. Note that the conjugacy problem for GBS groups was proven to be decidable in \cite{L92, B15}, see also \cite{W16} for complexity bounds.
\end{remark}

\subsection{Examples}

We give two examples aimed to help the reader understand how conjugacy classes and quasi-conjugacy classes work. In the examples along the paper, for simplicity of notation, we will only deal with GBS graphs $(\Gamma,\psi)$ with positive labels. For this reason we will use $\bA=\bbZ^{\cP(\Gamma,\psi)}$ omitting the $\bbZ/2\bbZ$ summand.

\begin{ex}\label{ex1part1}
    Later in the paper (see Example \ref{ex1part2}), we will describe the isomorphism problem for this example. Consider the GBS graph $(\Gamma,\psi)$ with a single vertex $v$ and nine edges, with labels as in Figure \ref{fig:example1}. We call $g$ the generator of $G_v\cong\bbZ$. Since the labels only use the prime factors $2$ and $3$, we have that $\cP(\Gamma,\psi)=\{2,3\}$ and we use $\bA=\bbZ^2$. The number $2^a3^b$ will be associated to the element $(a,b)\in\pA$ for $a,b\in\bbN$.
    
    Consider the point $p=(v,(1,3))$ corresponding to the element $g^{2^13^3}=g^{54}$. The quasi-conjugacy class of $p$ consists of only two points $(v,(1,3))$, $(v,(3,2))$, corresponding to the elements $g^{54},g^{72}$ respectively, see Figure \ref{fig:example1-qcclasses}. In this case $\qcsupp{p}=\emptyset$ and $\qcmin{p}$ is the set of the two points of the quasi-conjugacy class, which is also a conjugacy class. This means that $g^n$ for $n\ge1$ is conjugate to $g^{54}$ if and only if $n=54,72$. In this case $\cla{p}=0$.

    Consider the point $p=(v,(4,1))$ corresponding to the element $g^{2^43^1}=g^{48}$. The quasi-conjugacy class of $p$ consists of the points $(v,(k,1))$ for $k\ge 4$, corresponding to the elements $g^{2^k\cdot 3}$ for $k\ge 4$. In this case $\qcsupp{p}=\{2\}$ and $\qcmin{p}=\{(v,(4,1))\}$. This quasi-conjugacy class is partitioned into two conjugacy classes, depending on the parity of $k$. This means that $g^n$ for $n\ge1$ is conjugate to $g^{48}$ if and only if $n=2^k\cdot 3$ for $k\ge 4$ even. In this case $\cla{p}=\gen{(2,0)}$.

    Consider the point $p=(v,(2,3))$ corresponding to the element $g^{108}$. In this case $\qcsupp{p}=\{2,3\}$ and $\qcmin{p}=\{(2,3),(4,2)\}$ and $\cla{p}=\gen{(2,0),(0,1)}$. The quasi-conjugacy class is partitioned into two conjugacy classes, as $(\bbZ^{\qcsupp{p}})/\cla{p}\cong\bbZ/2\bbZ$ which has cardinality $2$.
\end{ex}

\begin{figure}[H]
\includegraphics[width=\textwidth]{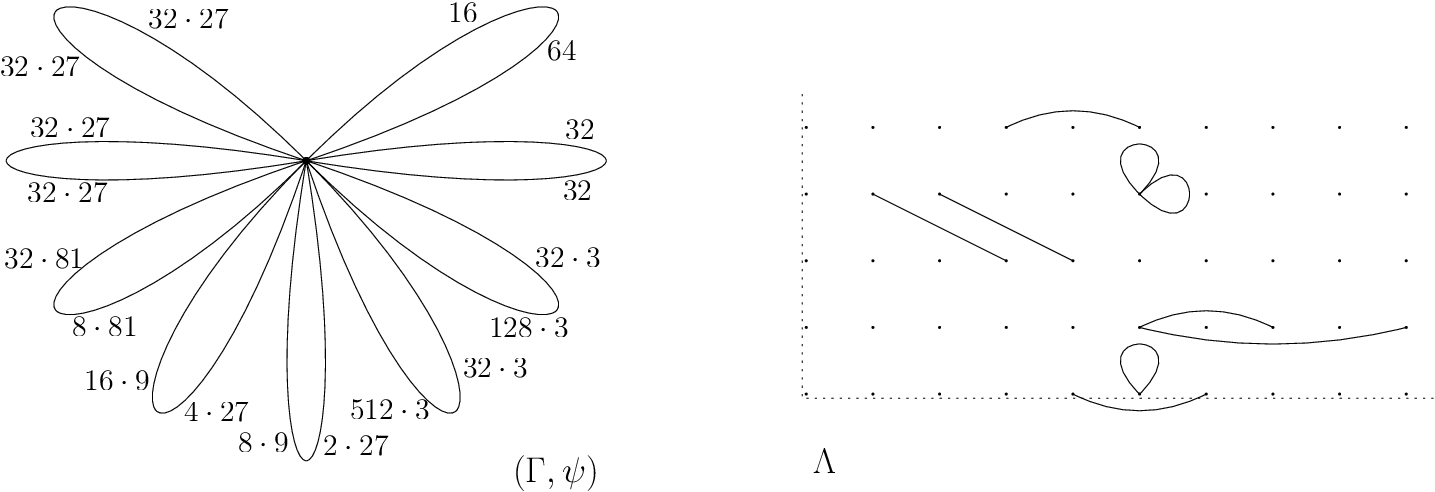}
\centering
\caption{On the left a GBS graph $(\Gamma,\psi)$ with one vertex and nine edges. On the right the corresponding affine representation $\Lambda$. On the horizontal axis the number of factors $2$ and on the vertical axis the number of factors $3$.}
\label{fig:example1}
\end{figure}

\begin{figure}[H]
\includegraphics[width=\textwidth]{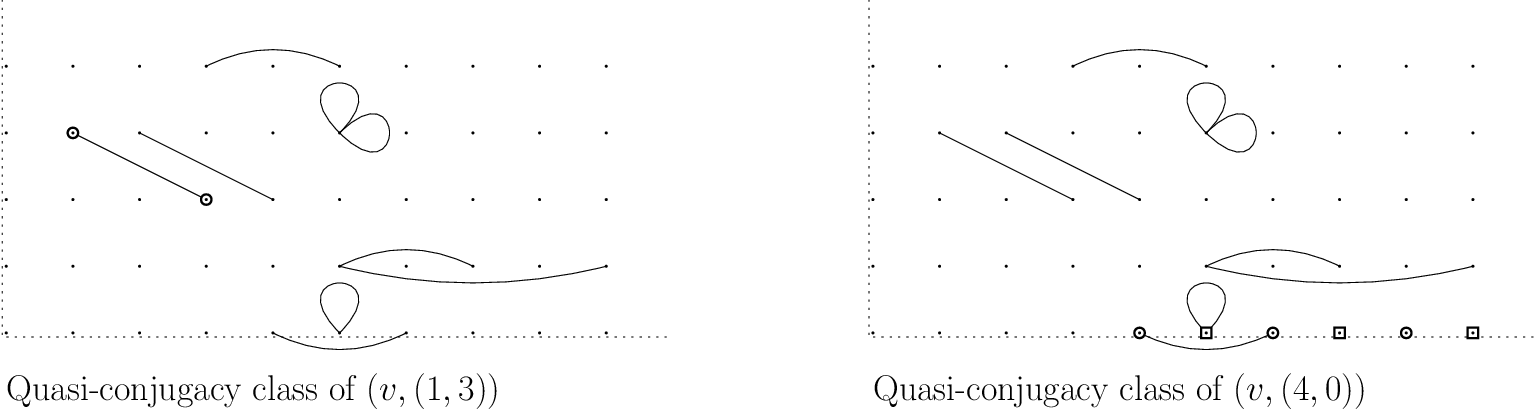}

\ 

\includegraphics[width=\textwidth]{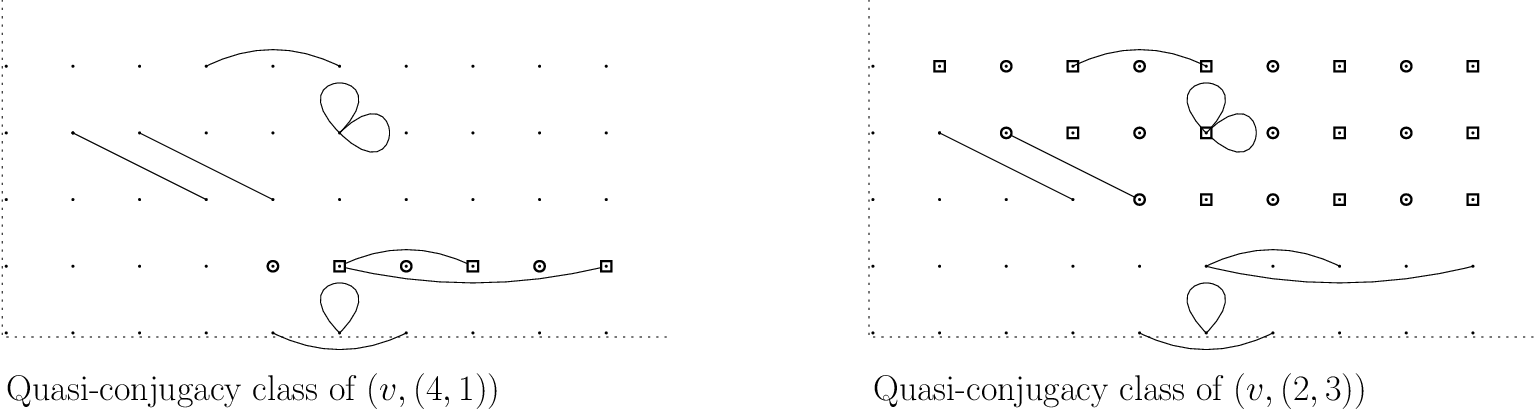}
\centering
\caption{In the figure we see four copies of the same affine representation $\Lambda$ associated with Example \ref{ex1part1}. In each copy, we put in evidence a different quasi-conjugacy class. We can see the quasi-conjugacy class of $(1,3)$ (top left), of $(4,0)$ (top right), of $(4,1)$ (bottom left), of $(2,3)$ (bottom right). To distinguish different conjugacy classes inside a common quasi-conjugacy class, we mark the vertices with different symbols (circles, squares).}
\label{fig:example1-qcclasses}
\end{figure}

\begin{ex}\label{ex2part1}
    This example is taken from \cite{Wan25}. Later in the paper (see Example \ref{ex2part2}), we will describe the isomorphism problem for this example. Consider the GBS graph with a single vertex $v$ and four edges, with labels as in Figure \ref{fig:example2}. The prime numbers of this GBS graph are $\cP(\Gamma,\psi)=\{2,3,5,7\}$. The number $2^a 3^b 5^c 7^d$ will be associated to the element $(a,b,c,d)\in\pA$ for $a,b,c,d\in\bbN$. We call $\Lambda$ the affine representation of $(\Gamma,\psi)$: this consists of a single copy $\pA_v$ of $\pA$ containing four edges
    $$\begin{cases}
    (3,0,0,0)\edge (1,0,1,0)\\
    (1,1,0,0)\edge (0,1,1,0)\\
    (0,1,0,1)\edge (1,1,1,0)\\
    (1,0,0,1)\edge (1,1,1,0)
    \end{cases}$$

    Consider the point $p=(v,(3,0,0,0))$. Its quasi-conjugacy class is finite, with $\qcsupp{p}=\emptyset$ and minimal points $(3,0,0,0),(1,0,1,0)$ inside $\pA_v$. It coincides with a conjugacy class and $\cla{p}=0$.

    Consider the point $p=(v,(1,1,0,0))$. Its quasi-conjugacy class is finite, with $\qcsupp{p}=\emptyset$ and minimal points $(1,1,0,0),(0,1,1,0)$ inside $\pA_v$. It coincides with a conjugacy class and $\cla{p}=0$.

    Consider the point $p=(v,(0,1,0,1))$. Its quasi-conjugacy class has $\qcsupp{p}=\{2,3,5,7\}$ and minimal points $(2,1,0,0),(1,1,1,0),(0,1,2,0),(1,0,0,1),(0,1,0,1)$ inside $\pA_v$. It coincides with a conjugacy class since $\cla{p}=\bbZ^{\qcsupp{p}}$.
\end{ex}

\begin{figure}[H]
\centering
\includegraphics[width=0.4\textwidth]{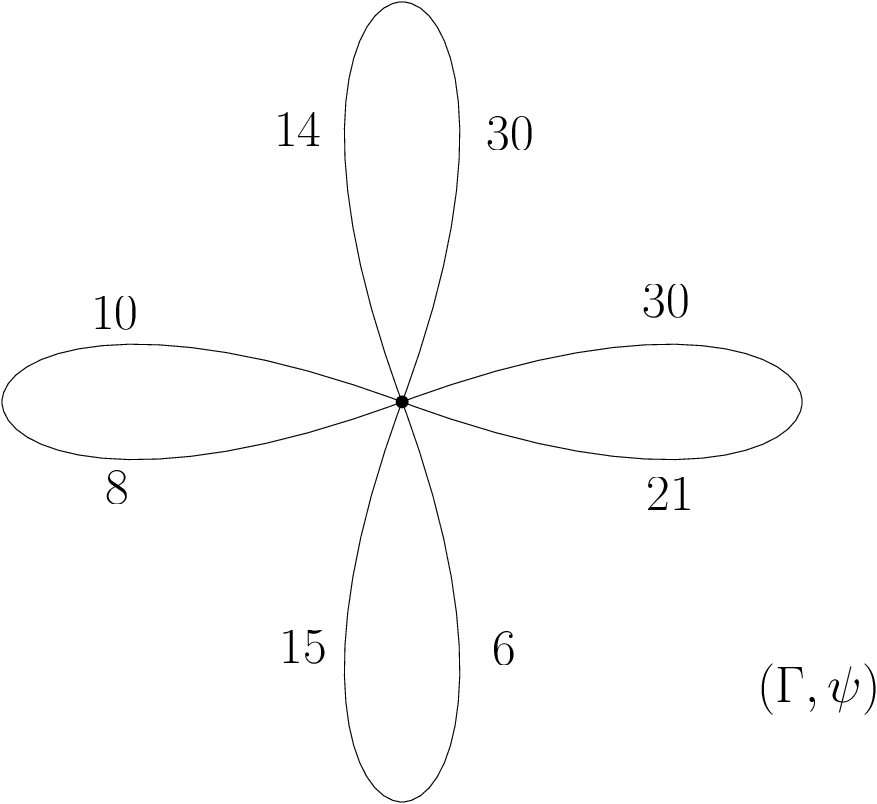}
\centering
\caption{In the picture a GBS graph $(\Gamma,\psi)$ with one vertex and four edges.}
\label{fig:example2}
\end{figure}

\section{Moves on GBS graphs}

In this section, we recall from \cite{For06} and \cite{ACK-iso1} the definition of several moves that, given a GBS graph, produce another GBS graph with isomorphic fundamental group, and review some literature on the topic.

\subsection{List of moves}

Let $(\Gamma,\psi)$ be a GBS graph.

\paragraph{Vertex sign-change.}

Let $v\in V(\Gamma)$ be a vertex. Define the map $\psi':E(\Gamma)\rar\bZ\setminus\{0\}$ such that $\psi'(e)=-\psi(e)$ if $\tau(e)=v$ and $\psi'(e)=\psi(e)$ otherwise. We say that the GBS graph $(\Gamma,\psi')$ is obtained from $(\Gamma,\psi)$ by a \textbf{vertex sign-change}. If $\cG$ is the GBS graph of groups associated to $(\Gamma,\psi)$, then the vertex sign change move corresponds to changing the chosen generator for the vertex group $G_v$.

\paragraph{Edge sign-change.}

Let $(\Gamma,\psi)$ be a GBS graph. Let $d\in E(\Gamma)$ be an edge. Define the map $\psi':E(\Gamma)\rar\bZ\setminus\{0\}$ such that $\psi'(e)=-\psi(e)$ if $e=d,\ol{d}$ and $\psi'(e)=\psi(e)$ otherwise. We say that the GBS graph $(\Gamma,\psi')$ is obtained from $(\Gamma,\psi)$ by an \textbf{edge sign-change}. If $\cG$ is the GBS graph of groups associated to $(\Gamma,\psi)$, then the vertex sign change move corresponds to changing the chosen generator for the edge group $G_d$.

\paragraph{Slide.}

Let $d,e$ be distinct edges with $\tau(d)=\iota(e)=u$ and $\tau(e)=v$; suppose that $\psi(\ol{e})=n$ and $\psi(e)=m$ and $\psi(d)=\ell n$ for some $n,m,\ell\in\bbZ\setminus\{0\}$ (see Figure \ref{fig:slide}). Define the graph $\Gamma'$ by replacing the edge $d$ with an edge $d'$; we set $\iota(d')=\iota(d)$ and $\tau(d')=v$; we set $\psi(\ol{d'})=\psi(\ol{d})$ and $\psi(d')=\ell m$. We say that the GBS graph $(\Gamma',\psi)$ is obtained from $(\Gamma,\psi)$ by a \textbf{slide}. At the level of the affine representation, we have an edge $p\edge q$ and we have another edge with an endpoint at $p+\ba$ for some $\ba\in\pA$. The slide has the effect of moving the endpoint from $p+\ba$ to $q+\ba$ (see Figure \ref{fig:slide}),
$$\begin{cases}
p\edge  q\\
r\edge  p+\ba
\end{cases}
\xrightarrow{\text{slide}}\quad
\begin{cases}
p\edge  q\\
r\edge  q+\ba
\end{cases}$$

\begin{figure}[H]
\centering

\includegraphics[width=0.8\textwidth]{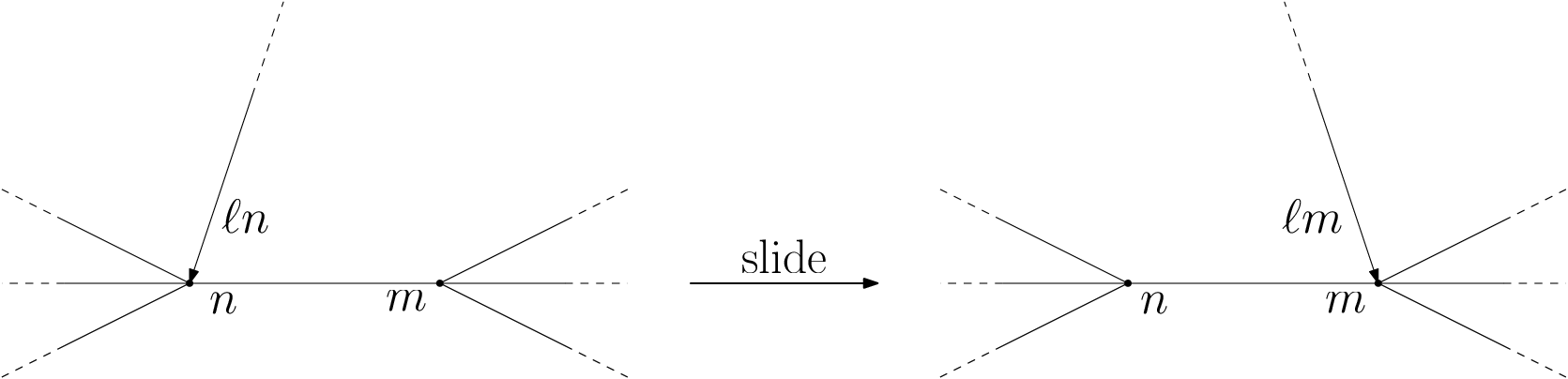}

\vspace{1cm}

\includegraphics[width=0.8 \textwidth]{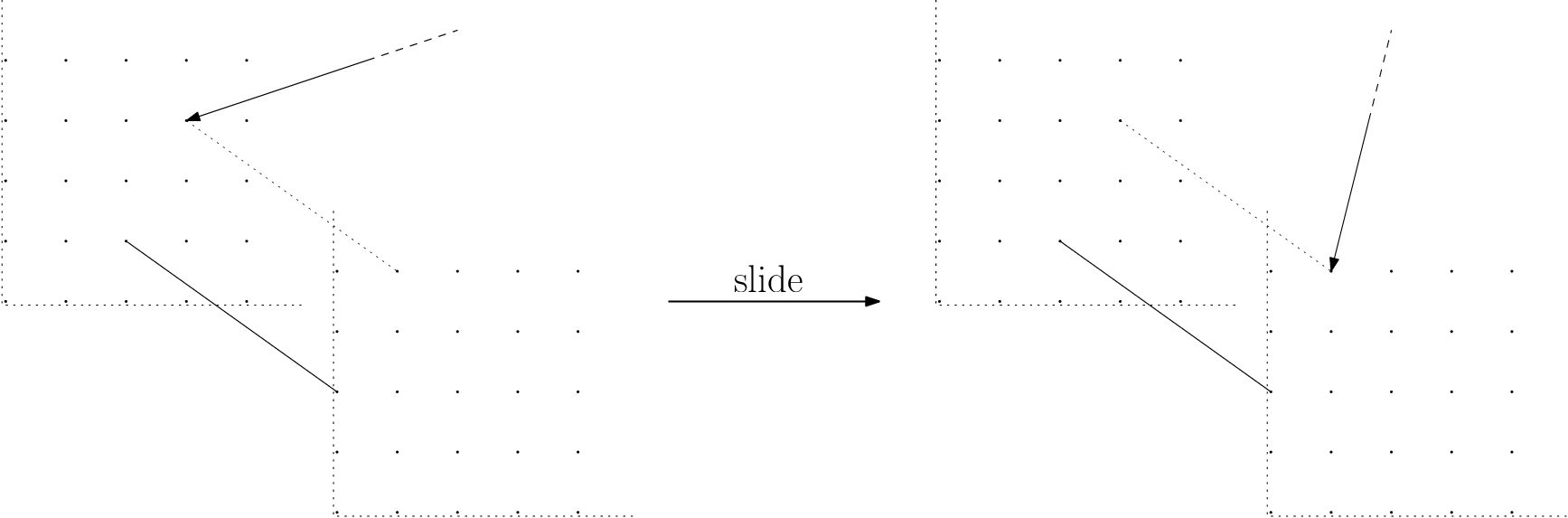}

\caption{An example of a slide move. Above you can see the GBS graphs. Below you can see the corresponding affine representations.}
\label{fig:slide}
\end{figure}

\paragraph{Induction.}

Let $(\Gamma,\psi)$ be a GBS graph. Let $e$ be an edge with $\iota(e)=\tau(e)=v$; suppose that $\psi(\ol{e})=1$ and $\psi(e)=n$ for some $n\in\bbZ\setminus\{0\}$, and choose $\ell\in\bbZ\setminus\{0\}$ and $k\in\bbN$ such that $\ell\divides n^k$. Define the map $\psi'$ equal to $\psi$ except on the edges $d\not=e,\ol{e}$ with $\tau(e)=v$, where we set $\psi'(d)=\ell\cdot\psi(d)$. We say that the GBS graph $(\Gamma,\psi')$ is obtained from $(\Gamma,\psi)$ by an \textbf{induction}. At the level of the affine representation, we have an edge $(v,\mathbf{0})\edge (v,\bw)$. We choose $\bw_1\in\pA$ such that $\bw_1\le k\bw$ for some $k\in\bbN$; we take all the vertices of other edges lying in $\pA_v$, and we translate them up by adding $\bw_1$ (see Figure \ref{fig:ind}).

\begin{figure}[H]
\centering
\vspace{1cm}

\includegraphics[width=0.65\textwidth]{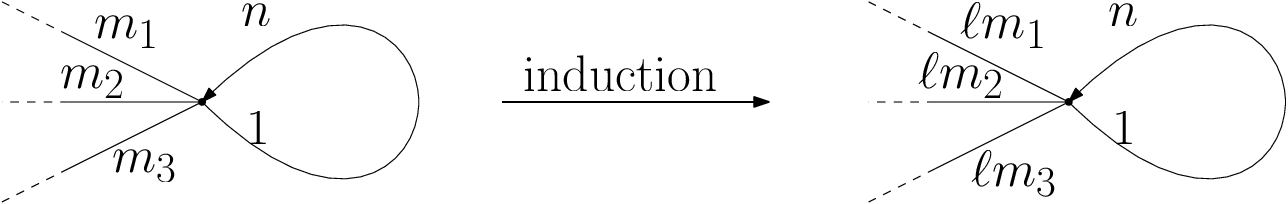}

\vspace{1cm}

\includegraphics[width=0.8\textwidth]{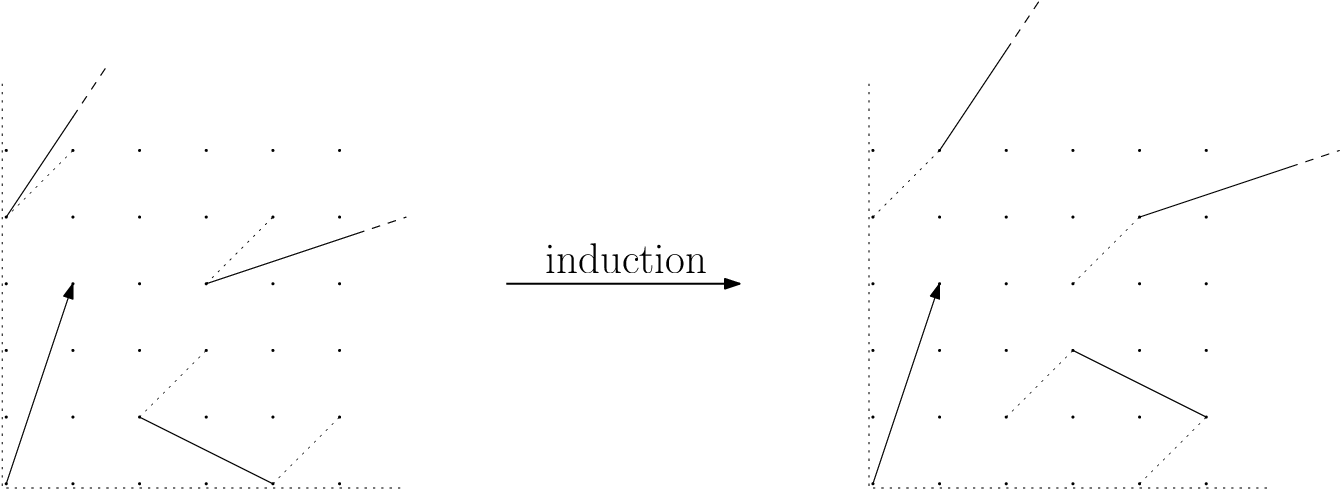}

\caption{An example of an induction move. Above you can see the GBS graphs; here $\ell\divides n^k$ for some integer $k\ge0$. Below you can see the corresponding affine representations.}
\label{fig:ind}
\end{figure}

\paragraph{Swap.}

Let $(\Gamma,\psi)$ be a GBS graph. Let $e_1,e_2$ be distinct edges with $\iota(e_1)=\tau(e_1)=\iota(e_2)=\tau(e_2)=v$; suppose that $\psi(\ol{e}_1)=n$ and $\psi(e_1)=\ell_1 n$ and $\psi(\ol{e}_2)=m$ and $\psi(e_2)=\ell_2 m$ and $n\divides m$ and $m\divides \ell_1^{k_1} n$ and $m\divides \ell_2^{k_2} n$ for some $n,m,\ell_1,\ell_2\in\bbZ\setminus\{0\}$ and $k_1,k_2\in\bbN$ (see Figure \ref{fig:swap}). Define the graph $\Gamma'$ by substituting the edges $e_1,e_2$ with two edges $e_1',e_2'$; we set $\iota(e_1')=\tau(e_1')=\iota(e_2')=\tau(e_2')=v$; we set $\psi(\ol{e_1'})=m$ and $\psi(e_1')=\ell_1 m$ and $\psi(\ol{e_2'})=n$ and  $\psi(e_2')=\ell_2 n$. We say that the GBS graph $(\Gamma',\psi)$ is obtained from $(\Gamma,\psi)$ by a \textbf{swap move}. Let $\bw_1,\bw_2\in\pA$ and $p,q\in V(\Lambda)$ be such that $p\le q\le p+k_1\bw_1$ and $p\le q\le p+k_2\bw_2$ for some $k_1,k_2\in\bbN$. At the level of the affine representation, we have an edge $e_1$ going from $p$ to $p+\bw_1$ and an edge $e_2$ going from $q$ to $q+\bw_2$. The swap has the effect of substituting them with $e_1'$ from $q$ to $q+\bw_1$ and with $e_2'$ from $p$ to $p+\bw_2$ (see Figure \ref{fig:swap}),
$$\begin{cases}
p\edge  p+\bw_1\\
q\edge  q+\bw_2
\end{cases}
\xrightarrow{\text{swap}}\quad
\begin{cases}
q\edge  q+\bw_1\\
p\edge  p+\bw_2
\end{cases}$$

\begin{figure}[H]
\centering
\vspace{1cm}

\includegraphics[width=0.6\textwidth]{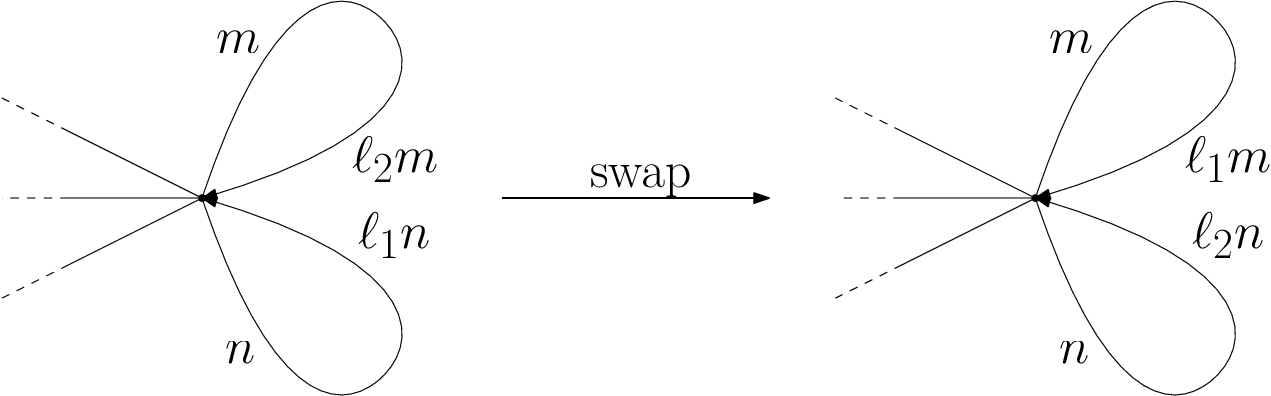}

\vspace{1cm}

\includegraphics[width=0.8 \textwidth]{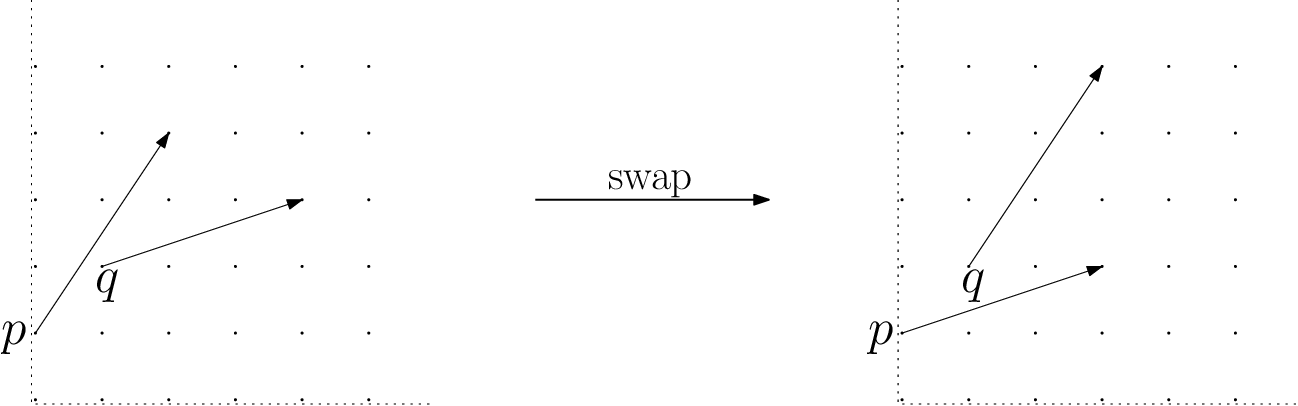}

\caption{An example of a swap move. Above you can see the GBS graphs; here $n\divides m\divides\ell_1^{k_1}n$ and $n\divides m\divides\ell_2^{k_2}n$ for some integers $k_1,k_2\ge0$. Below you can see the corresponding affine representations.}
\label{fig:swap}
\end{figure}

\paragraph{Connection.}\label{sec:connection}

Let $(\Gamma,\psi)$ be a GBS graph. Let $d,e$ be distinct edges with $\iota(d)=u$ and $\tau(d)=\iota(e)=\tau(e)=v$; suppose that  $\psi(\ol{d})=m$ and $\psi(d)=\ell_1 n$ and $\psi(\ol{e})=n$ and $\psi(e)=\ell n$ and $\ell_1\ell_2=\ell^k$ for some $m,n,\ell_1,\ell_2,\ell\in\bbZ\setminus\{0\}$ and $k\in\bbN$ (see Figure \ref{fig:connection}). Define the graph $\Gamma'$ by substituting the edges $d,e$ with two edges $d',e'$; we set $\iota(d')=v$ and $\tau(d')=\iota(e')=\tau(e')=u$; we set  $\psi(\ol{d'})=n$ and $\psi(d')=\ell_2 m$ and $\psi(\ol{e'})=m$ and $\psi(e')=\ell m$. We say that the GBS graph $(\Gamma',\psi)$ is obtained from $(\Gamma,\psi)$ by a \textbf{connection move}. Let $\bw,\bw_1,\bw_2\in\pA$ and $k\in\bbN$ be such that $\bw_1+\bw_2=k\cdot\bw$. At the level of the affine representation, we have a two edges $q\edge  p+\bw_1$ and $p\edge  p+\bw$. The connection move has the effect of replacing them with two edges $p\edge  p+\bw_2$ and $q\edge  q+\bw$ (see Figure \ref{fig:connection}),
$$\begin{cases}
q\edge  p+\bw_1\\
p\edge  p+\bw
\end{cases}
\xrightarrow{\text{connection}}\quad
\begin{cases}
p\edge  q+\bw_2\\
q\edge  q+\bw
\end{cases}$$

\begin{rmk}
In the definition of connection move, we also allow the two vertices $u,v$ to coincide.
\end{rmk}

\begin{figure}[H]
\centering

\includegraphics[width=0.8\textwidth]{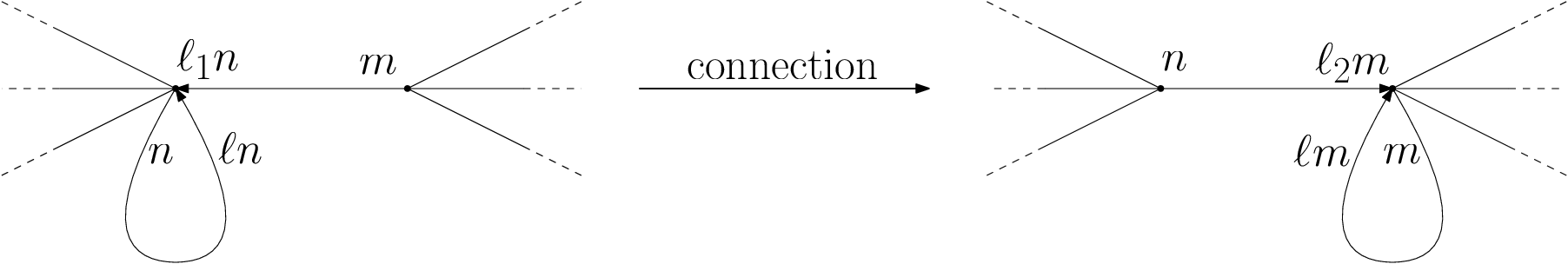}

\vspace{1cm}

\includegraphics[width=0.9 \textwidth]{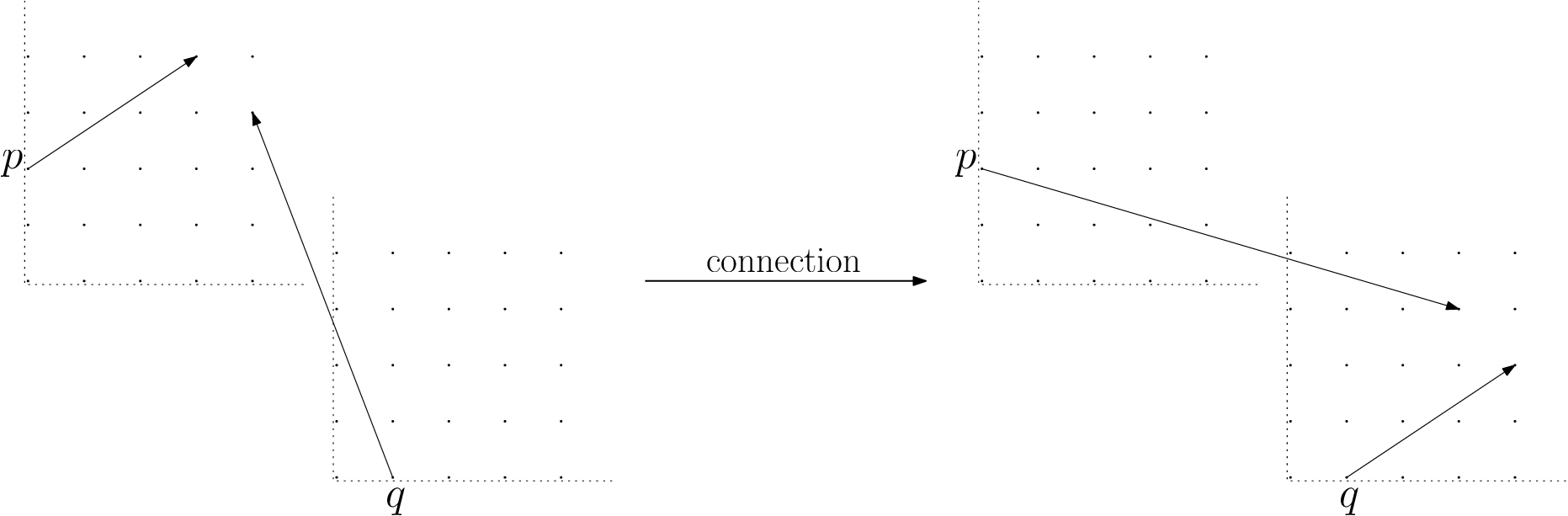}

\caption{An example of a connection move. Above you can see the GBS graphs; here $\ell_1\ell_2=\ell^k$ for some integer $k\ge0$. Below you can see the corresponding affine representations.}
\label{fig:connection}
\end{figure}

\subsection{Further moves}

The following Lemmas \ref{lem:self-slide} and \ref{reverse-slide} introduce two additional moves that will be used in the next sections.

\begin{lem}[Self-slide]\label{lem:self-slide}
Let $\ba,\bb,\bw\in\pA$ and $\bx\in\bA$. Suppose $\ba,\bw$ controls $\bb$ and $\bb+2\bx$. Then we can change
$$
\begin{cases}
\ba\edge  \ba+\bw\\
\bb\edge  \bb+\bx
\end{cases}
\qquad\text{into}\qquad
\begin{cases}
\ba\edge  \ba+\bw\\
\bb+\bx\edge  \bb+2\bx
\end{cases}
$$
by a sequence of slides and swaps.
\end{lem}
\begin{proof}
    See \cite{ACK-iso1}.
\end{proof}

\begin{lem}[Reverse slide]\label{reverse-slide}
Let $\ba,\bb,\bw\in\pA$ and $\bx\in\bA$ with $\bw+\bx\ge\mathbf{0}$. Suppose $\ba,\bw$ controls $\bb,\bb+\bx$ and suppose $\ba,\bw+\bx$ controls $\bb,\bb+\bx$. If $\pA_v$ contains edges $\ba\edge  \ba+\bw$ and $\bb\edge  \bb+\bx$, then we can change
$$
\begin{cases}
\ba\edge  \ba+\bw\\
\bb\edge  \bb+\bx
\end{cases}
\qquad\text{into}\qquad
\begin{cases}
\ba\edge  \ba+\bw+\bx\\
\bb\edge  \bb+\bx
\end{cases}
$$
by a sequence of slides and swaps.
\end{lem}
\begin{proof}
    See \cite{ACK-iso1}.
\end{proof}

\subsection{Sequences of moves}

In this section we recall how the isomorphism problem for generalized Baumslag-Solitar groups is reduced to checking whether two GBS graphs can be related by  a sequence of moves. The following \Cref{thm:sequence-new-moves} is one of the main results from \cite{ACK-iso1}, which will be the starting point for all the results of this paper. For the notion of \textit{totally reduced} GBS graph, we refer the reader to \cite{ACK-iso1}. We point out that the requirement of being totally reduced is not restrictive at all: every GBS graph can be algorithmically changed into a totally reduced one by a sequence of moves.

\begin{thm}\label{thm:sequence-new-moves}
Let $(\Gamma,\psi),(\Delta,\phi)$ be totally reduced GBS graphs and suppose that the corresponding graphs of groups have isomorphic fundamental group. Then $\abs{V(\Gamma)}=\abs{V(\Delta)}$ and there is a sequence of slides, swaps, connections, sign-changes and inductions going from $(\Delta,\phi)$ to $(\Gamma,\psi)$. Moreover, all the sign-changes and inductions can be performed at the beginning of the sequence.
\end{thm}
\begin{proof}
    See \cite{ACK-iso1}.
\end{proof}

Thus it suffices to deal with the following problem: given two GBS graphs $(\Gamma,\psi),(\Gamma',\psi')$, determine whether there is a sequence of edge sign-changes, inductions, slides, swaps, and connections going from one to the other. Note that a sign-change (resp. a slide, an induction, a swap, a connection) induces a natural bijection between the set of vertices of the graph before and after the move. Of course we can ignore the issue of guessing the bijection among the sets of vertices and the sign-changes at the beginning of the sequence, as these choices can be done only in finitely many ways. In what follows, we will also ignore the issue of guessing the inductions at the beginning of the sequence, as this hopefully represents a marginal issue - even though sometimes there are infinitely many possibilities, and thus this issue should be dealt with. In this paper, we will focus on the following question:

\begin{question}\label{quest:main}
    Given two totally reduced GBS graphs $(\Gamma,\psi),(\Gamma',\psi')$ and a bijection $b:V(\Gamma')\rar V(\Gamma)$, is there a sequence of {\rm(}edge sign-changes,{\rm)} slides, swaps, and connections going from $(\Gamma,\psi)$ to $(\Gamma',\psi')$ and inducing the bijection $b$ on the set of vertices?
\end{question}

In \cite{ACK-iso1} we also show that the Question \ref{quest:main} can be reduced to the case of one-vertex graphs, and with all edge-labels positive and $\not=1$. However, the tools we develop in this paper are general and do not make use of these additional assumptions.

\subsection{Examples}

Theorem \ref{thm:sequence-new-moves} can already be used to explicitly describe, in some cases, the list of all the possible configurations which can be reached from a given one. We show through an example how this can help in solving the isomorphism problem for GBSs. As usual, in the examples we use $\bA=\bbZ^{\cP(\Gamma,\psi)}$, omitting the $\bbZ/2\bbZ$ summand.

\begin{ex}\label{ex3}
    Consider the GBS graph $(\Gamma,\psi)$ with one vertex $v$ and two edges, with labels $8,48$ and $27,72$ (see Figure \ref{fig:example3-GBSgraph}). If we apply a connection move, we obtain another GBS graph with one vertex $v$ and two edges, the new labels being $8,108$ and $27,162$. It is easy to see that, using slide moves, we can reach all the configurations of the following list $L$:
\begin{enumerate}
    \item Two edges with labels $8,48$ and $27,72\cdot 6^k$, for some $k\ge0$ integer.
    \item Two edges with labels $8,108\cdot 6^k$ and $27,162$, for some $k\ge0$ integer.
\end{enumerate}
It is also routinge to check that if we start at a configuration in $L$, then we can not apply a swap move, and if we apply a slide or a connection, we fall again in a configuration listed in $L$. Therefore, $L$ is a complete list of all the configurations which can be reached from $(\Gamma,\psi)$ using only slides, swaps, and connections. By Theorem \ref{thm:sequence-new-moves}, we can algorithmically determine whether a GBS graph $(\Delta,\phi)$ encodes a GBS group isomorphic to that of $(\Gamma,\psi)$. In order to do this, we first transform $(\Delta,\phi)$ into a totally reduced GBS graph (this can be done algorithmically). Then we apply sign-changes in all the (finitely many) possible ways, and for each of the resulting GBS graphs, we check whether it appears in the list $L$. Note that induction is not possible on $(\Gamma,\psi)$, so we do not need to worry about it in this example.
\end{ex}

\begin{figure}[H]
\centering
\includegraphics[width=0.5\textwidth]{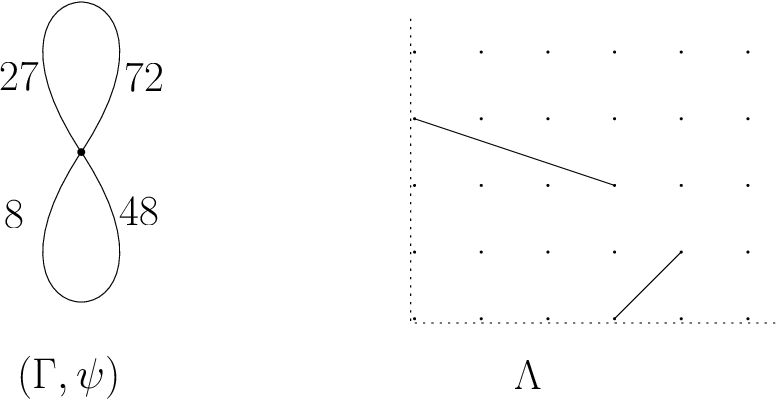}
\centering
\caption{On the left, a GBS graph $(\Gamma,\psi)$ with one vertex and two edges. On the right, the corresponding affine representation $\Lambda$ (on the horizontal axis the number of factors $2$, on the vertical axis the number of factors $3$).}
\label{fig:example3-GBSgraph}
\end{figure}

\begin{figure}[H]
\centering
\includegraphics[width=0.6\textwidth]{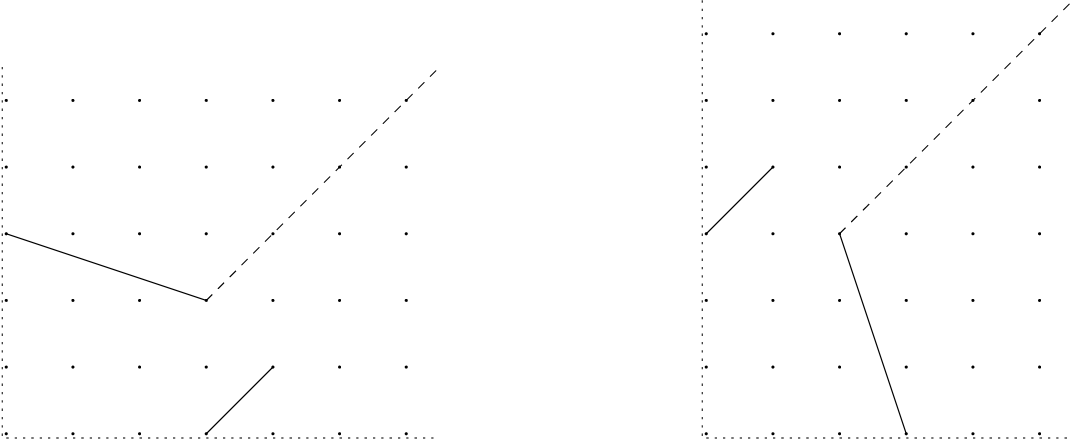}
\centering
\caption{On the left, the configurations that can be reached from $(\Gamma,\psi)$ by applying only slide moves - the endpoint of one of the edges is allowed to move along the dashed line, while the other three endpoints are required to stay fixed. On the right, the configurations that can be reached from $(\Gamma,\psi)$ by applying a connection and then only slide moves - again, the endpoint of one of the edges is allowed to move along the dashed line, while the other endpoints stay fixed.}
\label{fig:example3-deformation-space}
\end{figure}

\section{Independence of quasi-conjugacy classes}

In this section we describe the linear invariants of the moves (which are also essentially isomorphism invariants). These are listed in \Cref{def:linear-invariants}, and are the set of primes, the conjugacy and quasi-conjugacy classes, and the number of edges in each conjugacy class.

We divide slide moves into two types: \textit{internal} (when they involve two edges in the same quasi-conjugacy class) and \textit{external} (otherwise). We introduce a notion of \textit{external equivalence}, we show that it can be computed from the linear invariants, and that it can be used to describe external slide moves (without any need to know the exact position of the edges). We further prove that edges in different quasi-conjugacy classes can only interact with each other using the external equivalence. For the isomorphism problem, this means that distinct quasi-conjugacy classes can be dealt with separately and independently of each other, see \Cref{thm:independence-qc-classes}.

Finally, we introduce the notions of \textit{minimal region}. Every endpoint of every edge must lie above at least one minimal region, making them into excellent ``starting points'' for slide moves. We show that each of these regions is forced to contain at least one endpoint of some edge, see \Cref{lem:escape-minimal-regions}. This reduces considerably the number of possible configurations that one has to consider.

We show with concrete examples how useful these notions are when describing the isomorphism problem.

\subsection{Linear invariants}

Let $(\Gamma,\psi)$ be a GBS graph and let $\Lambda$ be its affine representation. In Section \ref{sec:conjugacy-classes}, we defined a partition of the set of vertices $V(\Lambda)$, given by the equivalence relation of conjugacy. We also defined a pre-order $\lecnj$ on $V(\Lambda)$, by saying that $p\lecnj q$ if $p\le q'\cnj q$ for some other vertex $q'$. This pre-order $\lecnj$ induces another partition of $V(\Lambda)$, given by the equivalence relation of quasi-conjugacy and, at the same time, defines an order relation on the set of quasi-conjugacy classes. These relations can be extended to edges, as follows.

\begin{defn}\label{def:conj-edges}
    For $e,f\in E(\Lambda)$ we denote:
    \begin{enumerate}
        \item $e\cnj f$ if $\iota(e)\cnj\iota(f)$, and we say that $e,f$ are \textbf{conjugate}.
        \item $e\lecnj f$ if $\iota(e)\lecnj\iota(f)$.
        \item $e\qcnj f$ if $\iota(e)\qcnj\iota(f)$, and we say that $e,f,$ are \textbf{quasi-conjugate}.
    \end{enumerate}
\end{defn}

We observe that for $e\in E(\Lambda)$ we have that $\iota(e)\cnj\tau(e)$. In particular, in Definition \ref{def:conj-edges}, we can use $\tau(e),\tau(f)$ instead of $\iota(e),\iota(f)$. As for vertices, conjugacy and quasi-conjugacy are equivalence relations on the sets of edges, and $\lecnj$ induces a partial order on the set of quasi-conjugacy classes of edges. In particular, we obtain a finite poset of quasi-conjugacy classes of edges with the order relation $\lecnj$.

\begin{defn}\label{def:conj-edge-vertex}
    For $e\in E(\Gamma)$ we say that:
    \begin{enumerate}
        \item $e$ belongs to the conjugacy class $C$ if $\iota(e)\in C$. We define the set
        $$\cnjedges{\Lambda}{C}:=\{e\in E(\Lambda) : \iota(e)\in C\}$$
        of the edges in the conjugacy class $C$.
        \item $e$ belongs to the quasi-conjugacy class $Q$ if $\iota(e)\in Q$. We define the set
        $$\qcnjedges{\Lambda}{Q}:=\{e\in E(\Lambda) : \iota(e)\in Q\}$$
        of the edges quasi-conjugate to $Q$.
    \end{enumerate}
\end{defn}

We are now ready to define the linear invariants of a GBS graph.

\begin{defn}[Linear invariants]\label{def:linear-invariants}
    Let $(\Gamma,\psi)$ be a GBS graph with affine representation $\Lambda$. We define the \textbf{linear invariants} of $(\Gamma,\psi)$ as the following finite list of data:
    \begin{enumerate}
        \item\label{itm:inv-1} The set of primes $\cP(\Gamma,\psi)$.
        \item\label{itm:inv-2} The finite list of quasi-conjugacy classes $Q$ containing at least one edge, each of them given through the finite data $\qcmin{Q},\qcsupp{Q},\cla{Q}$. For each such quasi-conjugacy class $Q$, the number $\abs{\qcnjedges{\Lambda}{Q}}$.
        \item\label{itm:inv-3} The finite list of conjugacy classes $C$ containing at least one edge. For each such conjugacy class $C$, the number $\abs{\cnjedges{\Lambda}{C}}$.
    \end{enumerate}
\end{defn}

Direct check using the definitions shows that the linear invariants are indeed invariant under slides, swaps, and connections. Sign-changes and induction preserve the invariants of \Cref{itm:inv-1,itm:inv-2}. An edge sign-change changes the numbers $\abs{\cnjedges{\Lambda}{C}}$, as it moves an edge from a conjugacy class $C$ to the conjugacy class $C+\be$, where $\be\in\bA$ is the $2$-torsion element. Both, a vertex sign-change and an induction, change the structure of the conjugacy classes $C$ and the numbers $\abs{\cnjedges{\Lambda}{C}}$; however, we point out that these changes occur in quite a controlled way, as vertex sign-changes and inductions just move the points inside a copy $\pA_v\subseteq V(\Lambda)$ by translation. In particular, we have the following:

\begin{cor}\label{cor:basic-invariants}
    Let $(\Gamma,\psi),(\Gamma',\psi')$ be two GBS graph and suppose that there is a sequence of slides, swaps, and connections going from one to the other and inducing a bijection $V(\Gamma)=V(\Gamma')$. Then this induces a bijection $V(\Lambda)=V(\Lambda')$, and $(\Gamma,\psi),(\Gamma',\psi')$ have the same linear invariants {\rm(see \Cref{def:linear-invariants})}.
\end{cor}

\subsection{External equivalence}

Let $(\Gamma,\psi)$ be a GBS graph and let $\Lambda$ be its affine representation.

\begin{defn}\label{def:external-equivalence}
    Let $p,q\in V(\Lambda)$ with $p\qcnj q$.
    \begin{enumerate}
        \item We say that $p,q$ are \textbf{externally equivalent}, denoted $p\ee q$, if there is an affine path, going from $p$ to $q$, which does not use edges quasi-conjugate to $p$ and $q$.
        \item  We denote $p\leee q$ if $p\le q'$ for some $q'\ee q$.
        \item We say that $p,q$ are \textbf{quasi-externally equivalent}, denoted $p\qee q$, if $p\leee q$ and $q\leee p$.
    \end{enumerate}
\end{defn}

The (quasi-)external equivalence class of $p$ is its (quasi-)conjugacy class in the GBS graph where we remove all edges quasi-conjugate to $p$. In particular, they can be described in terms of the same data as for (quasi-)conjugacy classes.

\begin{defn}\label{def:eesupp}
    Let $(\Gamma,\psi)$ be a GBS graph and let $p$ be a vertex of its affine representation $\Lambda$.
    \begin{enumerate}
        \item Define the \textbf{external minimal points} as follows
        $$\qeemin{p}=\{q\in V(\Lambda) : q\qee p \text{ and for all } q'\leee p \text{ with } q'\le q \text{ we have } q\le q'\}.$$
        \item Define the \textbf{external support} as follows
        $$\qeesupp{p}=\bigcup\{\supp{\bu} : \bu\in\pA \text{ and } p+\bu\qee p\}\subseteq\cP(\Gamma,\psi).$$
        \item Define the \textbf{external linear algebra} as follows
        \[
            \text{$\ela{p}=\{\br\in\bA : \br=\bw-\bw'$ for some $\bw,\bw'\in\pA$ such that $p+\bw\ee p+\bw' \ee p\}$.}
        \]
    \end{enumerate}
\end{defn}

Lemmas \ref{lem:qcmin-finite} \ref{lem:qc-supp-single-vector} \ref{lem:cla-subgroup} \ref{lem:linear-algebra} hold for external equivalence in the exact same way as they do for conjugacy. Similarly, it is easy to see that $\qeemin{p},\qeesupp{p},\ela{p}$ are invariant if we change $p$ by quasi-external equivalence. Propositions \ref{prop:description-qc-class} and \ref{prop:description-c-class} also hold for external equivalence, meaning that $\qeemin{p},\qeesupp{p},\ela{p}$ can be used to describe external equivalence class using a finite set of data. Finally, the algorithms of Propositions \ref{prop:qc-algorithm} and \ref{prop:c-algorithm}, and Corollary \ref{cor:algorithmic-paths} can be used for external equivalence as well, making everything algorithmically computable.

The following \Cref{lem:basic-invariants-imply-ee-classes} shows that the external equivalence classes are uniquely determined by the linear invariants.

\begin{lem}\label{lem:basic-invariants-imply-ee-classes}
    Let $(\Gamma,\psi),(\Gamma',\psi')$ be GBS graphs with a bijection $V(\Gamma)=V(\Gamma')$. Suppose that they have the same set of primes and the same quasi-conjugacy classes containing at least one edge. Then the GBS graphs $(\Gamma,\psi),(\Gamma',\psi')$ have the same external equivalence and quasi-external equivalence relations.
\end{lem}
\begin{proof}
    Let $p\qcnj q$ be two vertices in a quasi-conjugacy class $Q$ in the affine representation $\Lambda$ of $(\Gamma,\psi)$. Then $p\ee q$ if and only if there is a sequence $p=p_1,p_2,\dots,p_{\ell-1},p_\ell=q$ with the following properties:
    \begin{enumerate}
        \item For all $i=1,\dots,\ell-1$ there is a quasi-conjugacy class $Q_i\not=Q$ containing at least one edge, and points $r_i\cnj s_i\in Q_i$.
        \item For all $i=1,\dots,\ell-1$ we have that $p_i=r_i+\bw_i$ and $p_{i+1}=s_i+\bw_i$ for some $\bw_i\in\pA$.
    \end{enumerate}
    If $p\ee q$ then the existence of such a sequence follows from the definition. Conversely, if there is such a sequence, then we can easily construct an affine path from $p$ to $q$ which does not use any edge in $Q$. But the existence of such a sequence only depends on the quasi-conjugacy classes containing at least one edge. The statement follows.
\end{proof}

\subsection{External slide moves}

The following Lemma \ref{lem:external-slide} introduces a new move, based on external equivalence classes. As we shall see, this is the key to separate the main Question \ref{quest:main} into several independent problems, each of them internal to a single quasi-conjugacy class.

\begin{lem}[External slides]\label{lem:external-slide}
    Let $(\Gamma,\psi)$ be a GBS graph and let $\Lambda$ be its affine representation. Suppose that $\Lambda$ contains an edge $p\edge q$ and that $q\ee q'$ for some $p,q,q'\in V(\Lambda)$. Then we can change $p\edge q$ into $p\edge q'$ by a sequence of slide moves. We call the move that changes the edge $p\edge q$ into $p\edge q'$ an external slide.
\end{lem}
\begin{proof}
    Since $q\ee q'$ there must be an affine path from $q$ to $q'$ that does not use any edge quasi-conjugate to $q$. In particular, the affine paths does not use the edge $p\edge q$. Thus we can slide $p\edge q$ along the affine path in order to change it into $p\edge q'$.
\end{proof}

Let $(\Gamma,\psi)$ be a GBS graph and let $\Lambda$ be its affine representation. Observe that, if we perform a swap or a connection move involving two edges $e,f$, then $f\qcnj e$. Instead, when changing an edge $e$ by a slide over an edge $f$, there are two possibilities:
\begin{enumerate}
    \item $f\qcnj e$, in which case we say that the slide is \textbf{internal}.
    \item $f\not\qcnj e$, in which case we must have $f\lecnj e$, and the slide move is a particular case of the external slide move from Lemma \ref{lem:external-slide}.
\end{enumerate}
In particular, this means that we can deal with distinct quasi-conjugacy classes separately and independently from each other.

\begin{thm}[Independence of quasi-conjugacy classes]\label{thm:independence-qc-classes}
    Let $(\Gamma,\psi),(\Gamma',\psi')$ be two GBS graphs with a bijection $V(\Gamma)=V(\Gamma')$. Then the following are equivalent:
    \begin{enumerate}
        \item\label{fja} There is a sequence of slides, swaps, and connections going from $(\Gamma,\psi)$ to $(\Gamma',\psi')$ inducing the given bijection.
        \item\label{tioa} For every quasi-conjugacy class $Q$, there is a sequence of swaps, connections, internal slides, and external slides {\rm(\Cref{lem:external-slide})} - each of them involving only edges in $Q$ - going from the configuration in $(\Gamma,\psi)$ to the configuration in $(\Gamma',\psi')$.
    \end{enumerate}
\end{thm}
\begin{proof}
    Every slide move which involves two edges from different quasi-conjugacy classes can be replaced by an external slide as in \Cref{lem:external-slide}. On the other hand, external slides do not depend on the position of the edges in the other quasi-conjugacy classes (but only on the external equivalence relation). Thus, different moves involving edges in different quasi-conjugacy classes now commute. Therefore, we can deal with each quasi-conjugacy class separately and independently.
\end{proof}

In some particular cases, the possible configurations of certain quasi-conjugacy classes are extremely easy to describe.

\begin{cor}\label{cor:only-one-edge}
    Let $(\Gamma,\psi)$ be a GBS graph with affine representation $\Lambda$, and let $Q$ be a quasi-conjugacy class. Suppose that $Q$ contains only one edge $e$. Then, all the configurations inside $Q$ obtained by sequences of external slides, internal slides, swaps, and connections, can also be obtained by performing two external-slides on the two endpoints of $e$.
\end{cor}
\begin{proof}
    Note that with only one edge in the quasi-conjugacy class, it is impossible to perform internal slides, swaps, and connections.
\end{proof}

\begin{cor}\label{cor:qcclasses=eeclasses}
    Let $(\Gamma,\psi)$ be a GBS graph with affine representation $\Lambda$, and let $Q$ be a quasi-conjugacy class. Suppose that, for all $p,p'\in Q$, we have that $p\cnj p'$ if and only if $p\ee p'$. Then, all the configurations inside $Q$ obtained by sequences of external slides, internal slides, swaps, and connections, can also be obtained by performing only one external-slide on each endpoint of every edge.
\end{cor}
\begin{proof}
    It is easy to check that a single move that is either an internal slide, swap or connection, can be substituted with external slides. The statement follows.
\end{proof}

\begin{cor}\label{cor:qcclasses-empty-support}
    Let $(\Gamma,\psi)$ be a GBS graph with affine representation $\Lambda$, and let $Q$ be a quasi-conjugacy class. Suppose that $\qcsupp{Q}=\emptyset$. Then only finitely many configurations can be obtained in $Q$ by means of sequences of external slides, internal slides, swaps, and connections.
\end{cor}
\begin{proof}
    If $\qcsupp{Q}=\emptyset$, then $Q$ is finite, and therefore only finitely many configurations are possible.
\end{proof}

\subsection{Mobile edges}

Clay and Forester introduced a notion of \textit{mobile edge} \cite[Definition 3.12]{CF08}, which played an important role in the subsequent literature on GBSs (e.g. \cite{Dud17, Wan25}). An edge $e\in E(\Lambda)$ is mobile if and only if it belongs to a quasi-conjugacy class $Q$ with $\qcsupp{Q}\not=\emptyset$. As a corollary of the above discussion, we recover the following well-known results:

\begin{cor}[\textnormal{\cite[Theorem 8.2]{For06}}]
    Let $(\Gamma,\psi)$ be a GBS graph and suppose that $q(\pi_1(\Gamma))\cap\pA=\{\mathbf{0}\}$, where $q$ is the modular homomorphism, see {\rm Definition \ref{defn:modular_map}}. Then only finitely many configurations can be obtained from $(\Gamma,\psi)$ by means of sequences of sign-changes, inductions, slides, swaps, and connections.
\end{cor}
\begin{proof}
    If $q(\pi_1(\Gamma))\cap\pA=\{\mathbf{0}\}$, then for every quasi-conjugacy class $Q$ we must have $\qcsupp{Q}=\emptyset$, and the result follows from \Cref{cor:qcclasses-empty-support}.
\end{proof}

\begin{cor}[\textnormal{\cite[Theorem 2]{Dud17}}]
    There is an algorithm that, given GBS graphs $(\Gamma,\psi),(\Delta,\phi)$ such that $(\Gamma,\psi)$ has at most one mobile edge, decides whether or not the corresponding GBS groups are isomorphic.
\end{cor}
\begin{proof}
    Let $Q$ be a quasi-conjugacy class in the affine representation $\Lambda$ of $(\Gamma,\psi)$. If $Q$ contains at least one edge, then either $\qcsupp{Q}=\emptyset$, or $Q$ contains only one edge. In the former case, we can solve the isomorphism problem using \Cref{cor:qcclasses-empty-support}, and in the latter, using \Cref{cor:only-one-edge}.
\end{proof}

\subsection{Minimal regions of a quasi-conjugacy class}

Let $(\Gamma,\psi)$ be a GBS graph and let $\Lambda$ be its affine representation. Let $Q\subseteq V(\Lambda)$ be a quasi-conjugacy class.

\begin{defn}\label{def:region}
    A quasi-external equivalence class $R\subseteq Q$ is called a \textbf{region} of $Q$.
\end{defn}

The set of regions of $Q$ is a partition of $Q$. Moreover, the relation $\leee$ makes the set of regions of $Q$ into a partially ordered set. A \textbf{minimal region} of $Q$ is a region $M\subseteq Q$ which is minimal with respect to the partial order $\leee$. Every minimal region must contain at least one point from $\qcmin{Q}$ - possibly more than one. Each quasi-conjugacy class contains finitely many minimal regions, and always at least one.

\begin{rmk}
    Not every point in $\qcmin{Q}$ needs to belong to a minimal region.
\end{rmk}

The reason for which we are interested in the minimal regions of $Q$ is that they often give us important information about the position of the edges in the quasi-conjugacy class (especially when the quasi-conjugacy class contains few edges compared to the number of minimal regions, or when the configuration is particularly ``rigid'').

\begin{lem}\label{lem:escape-minimal-regions}
    Let $Q$ be a quasi-conjugacy class containing at least one edge. Then every minimal region of $Q$ must contain at least one endpoint of some edge. In particular, if $Q$ contains $d\ge1$ minimal regions, then $Q$ contains at least $\intsup{\frac{d}{2}}$ edges.
\end{lem}
\begin{proof}
    Edges from other quasi-conjugacy classes do not allow for a change of region. Therefore, to move away from a minimal region $M$, one has to use an edge belonging to the quasi-conjugacy class $Q$. By minimality, this means that there must be an edge with an endpoint in $M$.
\end{proof}

\begin{rmk}
    The lower bound $\intsup{\frac{d}{2}}$ in the above Lemma \ref{lem:escape-minimal-regions} is optimal.
\end{rmk}

\begin{defn}
    An edge $e$ is called \textbf{floating} if none of its endpoints $\iota(e),\tau(e)$ is in a minimal region.
\end{defn}

Note that the number of floating edges is not invariant when performing moves.

\subsection{Examples}

As usual, in the examples we use $\bA=\bbZ^{\cP(\Gamma,\psi)}$, omitting the $\bbZ/2\bbZ$ summand.

\begin{ex}\label{ex-extra1}
    In \Cref{fig:example-extra1} we can see two GBS graphs. The corresponding GBS groups are non-isomorphic. In fact, the linear invariants (more specifically: the conjugacy classes) are different in the two examples.

    Equivalently, we have the following. The generator $g$ of the vertex group is definable (up to conjugation and up to taking the inverse), for example as the unique elliptic element with no proper root. This means that any isomorphism from one GBS to the other must send the generator of the vertex group to the generator of the vertex group (or to a conjugate, or to a conjugate of the inverse). In the GBS group on the right we have that $g^4$ is not conjugate to $g^3$. In the GBS group on the right we have that $g^4$ is conjugate to $g^3$. This proves that they are not isomorphic.

    Note that the modular homomorphism will not distinguish these two examples. For both groups, the modular homomorphism is a map from a free group of rank three to $\bbQ^*$, sending the three elements of some bases to $1,2,3\in\bbQ^*$.
\end{ex}

\begin{figure}[H]
\centering
\includegraphics[width=0.7\textwidth]{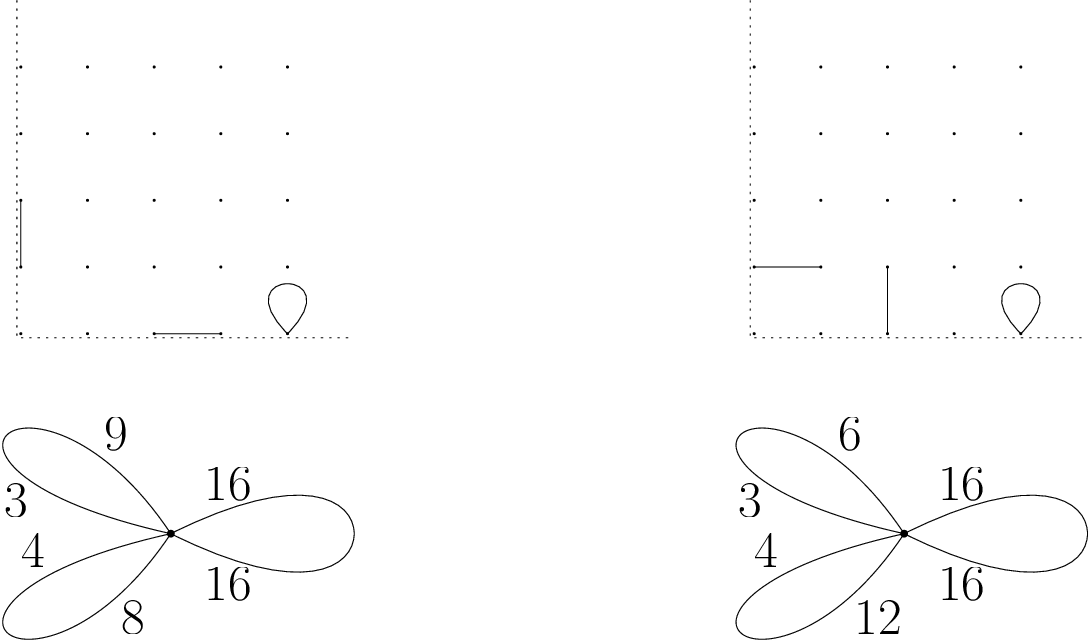}
\centering
\caption{Two non-isomorphic GBS groups, as in \Cref{ex-extra1}.}
\label{fig:example-extra1}
\end{figure}

\begin{ex}\label{ex-extra2}
    In \Cref{fig:example-extra2}, we can see three GBS graphs. The corresponding GBS groups are pairwise non-isomorphic. More generally, we can substitute the labels $8,8$ in the graph on the left with labels $2^k,2^k$ for $k\ge1$, and we obtain an infinite sequence of pairwise non-isomorphic GBS groups. Once again, the linear invariants are different in the three examples. However this time the conjugacy classes are the same; the difference lies in the number of edges in each conjugacy class. 

    Equivalently, we have the following. Consider the generator $g$ of the vertex group (which is definable up to inverse and conjugation). Consider the generators of the stabilizers of the edge groups. In this case, the invariant that distinguishes the three isomorphism classes is the number of generators of edge stabilizers that are conjugate to $g^8$, $g^{16}$ and$g^{32}$. In the GBS on the left, there is one edge whose generator is conjugate to $g^8$, and no edge whose generator is conjugate to $g^{16}$ or to $g^{32}$. For the GBS in the middle, there are $0,1,0$ edges conjugate to $g^8, g^{16}$, and $g^{32}$, respectively. For the GBS on the right, there are $0,0,1$. Since these numbers are isomorphism invariants, they prove that the three GBSs are pairwise non-isomorphic.

    Note that these numbers are only invariant among fully reduced GBS graphs representing the same GBS group. If one allows for expansions, then it is possible to add a new edge with stabilizer in any prescribed conjugacy class.

    We are also using the fact that in these examples induction is not possible. If induction was possible, then the generator $g$ of the vertex group would not be definable. In this case, variants of the same argument will work, but one has to be more careful.
\end{ex}

\begin{figure}[H]
\centering
\includegraphics[width=\textwidth]{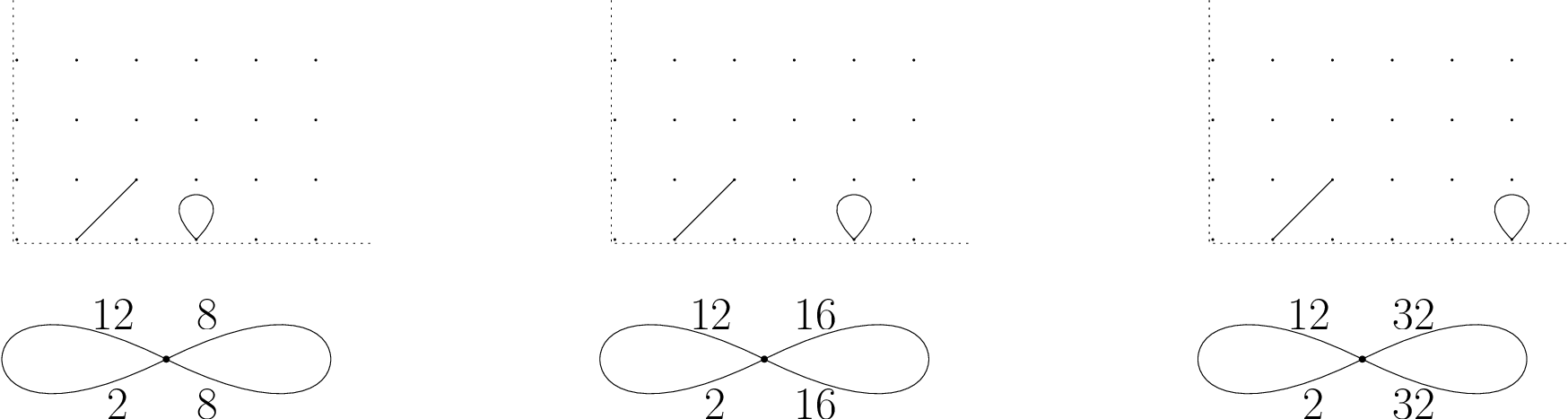}
\centering
\caption{Three pairwise non-isomorphic GBS groups, as in \Cref{ex-extra2}.}
\label{fig:example-extra2}
\end{figure}

\begin{ex}\label{ex1part2}
    We can now solve the isomorphism problem for the GBS graph of Example \ref{ex1part1}. Using the same notation as in Example \ref{ex1part1}, call $(\Gamma,\psi)$ the GBS graph, with a single vertex $v$, and $\Lambda$ its affine representation. The edges of $\Lambda$ are partitioned into four quasi-conjugacy classes $Q_1,Q_2,Q_3,Q_4$, see \Cref{fig:example1-qcclasses-minimals}, which had been described in \Cref{ex1part1}. Suppose that we try and perform a sequence of slides, swaps, and connections starting from this GBS graph.
    
    The quasi-conjugacy class $Q_1$ contains only the edge $(1,3)\edge(3,2)$, and thus by \Cref{cor:only-one-edge} that edge can not be changed. In the quasi-conjugacy classes $Q_3,Q_4$, the notions of conjugacy and external equivalence coincide, and thus by \Cref{cor:qcclasses=eeclasses} we can change the endpoints of the edges at will using external slides, and that is a complete list of all configurations that can be reached with slides, swaps, and connections.

    The quasi-conjugacy class $Q_2$ contains two edges, belonging to two distinct conjugacy classes, and one minimal region $M_2=\{(4,0)\}$; we have $\cla{Q_2}=\gen{(2,0)}$. For simplicity of notation, call $\ba=(4,0)$ and $\bu=(2,0)$. Let $e$ be the edge belonging to the conjugacy class of $(4,0)$: by Lemma \ref{lem:escape-minimal-regions}, the edge $e$ must always be given by $(4,0)\edge(4+2a,0)$ for some integer $a\ge0$. If $a=0$ then $M_2$ would be a quasi-conjugacy class by itself, contradiction, so we also always have $a\not=0$. Let $f$ be the edge belonging to the other conjugacy class: the edge $f$ must always be given by $(5+2b,0)\edge(5+2c,0)$ for some integers $b,c\ge0$. Finally, we must always have that $\gen{(2a,0),(2b-2c,0)}=\gen{(2,0)}$, otherwise the linear algebra of $Q_2$ would not be the correct one. On the other hand, using the results of \cite{ACK-iso1}, every configuration satisfying the above conditions can actually be reached. To summarize, the configurations for $Q_2$, that we can reach with a sequence of slides, swaps, and connections, are exactly the ones such that:
    \begin{enumerate}
        \item $e$ is given by $(4,0)\edge(4+2a,0)$ for some integer $a\ge1$.
        \item $f$ is given by $(5+2b,0)\edge(5+2c,0)$ for some integers $b,c\ge0$.
        \item $(a,b-c)=1$.
    \end{enumerate}

    Finally, we take into account the (finitely many) sign-changes, and we observe that induction moves do not apply to this example. This gives algorithmic solution to the isomorphism problem for an arbitrary GBS graph against this specific one.
\end{ex}

\begin{figure}[H]
\centering
\includegraphics[width=\textwidth]{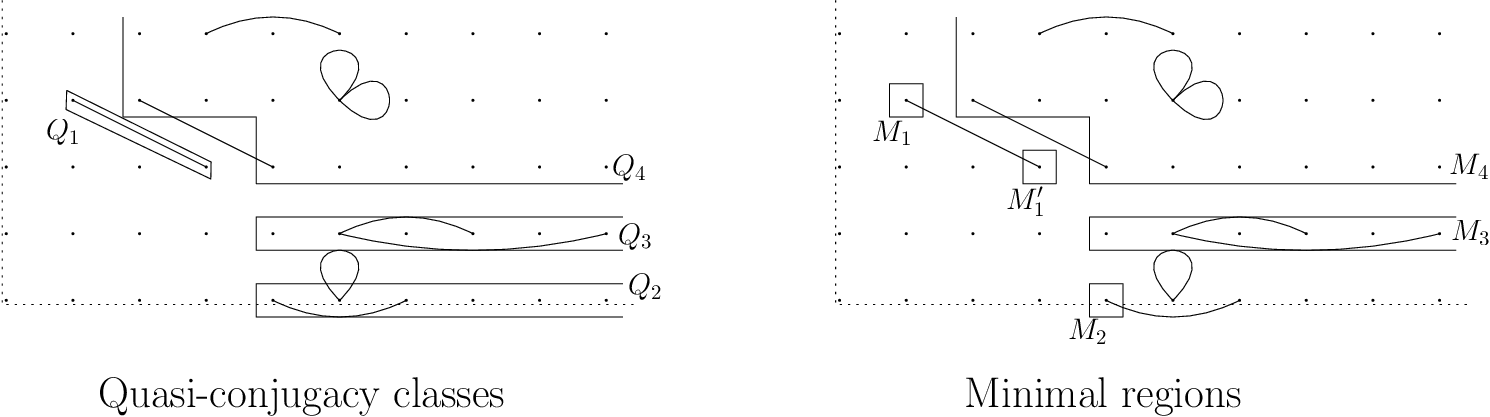}
\centering
\caption{The conjugacy classes and the minimal regions for the GBS graph of \Cref{ex1part2}.}
\label{fig:example1-qcclasses-minimals}
\end{figure}

\begin{ex}\label{ex2part2}
    We can now solve the isomorphism problem for the GBS graph of Example \ref{ex2part1}. Using the notation of Example \ref{ex2part1}, let $(\Gamma,\psi)$ be the GBS graph with a single vertex $v$, and $\Lambda$ be its affine representation.
    
    The edges of $\Lambda$ are partitioned into three quasi-conjugacy classes: one in the conjugacy class $P_1$ of the point $(3,0,0,0)$ in $\pA_v$, one in the conjugacy class $P_2$ of the point $(1,1,0,0)$ in $\pA_v$, and two in the conjugacy class $Q$ of the point $(0,1,0,1)$ in $\pA_v$. For a description of the quasi-conjugacy classes $P_1,P_2,Q$ see Example \ref{ex2part1}. These three quasi-conjugacy classes form a poset given by $P_1\lecnj Q$ and $P_2\lecnj Q$. By \Cref{thm:independence-qc-classes} we can examine different quasi-conjugacy classes independently.
    
    The quasi-conjugacy class $P_1$ contains only  the edge $(3,0,0,0)\edge(1,0,1,0)$, so by \Cref{cor:only-one-edge} that edge can not be changed by any move. A similar analysis on $P_2$ shows that the edge $(1,1,0,0)\edge(0,1,1,0)$ can not be changed either.

    The quasi-conjugacy class $Q$ has three minimal regions $M_1=\{(1,0,0,1)\}$ and $M_2=\{(0,1,0,1)\}$ and $M_3=\{(a,1,c,0) : a+c\ge2\}$. By Lemma \ref{lem:escape-minimal-regions}, along any sequence of slides, swaps, and connections, there is always (at least) one endpoint in $M_1$, one endpoint in $M_2$, one endpoint in $M_3$. We only need to understand where the fourth endpoint is. Consider the following list $L$ of configurations:
    \begin{enumerate}
        \item For $a+c,a'+c'\ge2$ the configurations
        $$\begin{cases}
            (1,0,0,1)\edge(a,1,c,0)\\
            (0,1,0,1)\edge(a',1,c',0)
        \end{cases}$$
        \item For $a+c\ge2$ and $b,d\ge0$ the configurations
        $$\begin{cases}
            (1,0,0,1)\edge(a,1,c,0)\\
            (0,1,0,1)\edge(1+b,0,d,1)
        \end{cases}$$
        \item For $a+c\ge2$ and $b,d\ge0$ the configurations
        $$\begin{cases}
            (1,0,0,1)\edge(b,1,d,1)\\
            (0,1,0,1)\edge(a,1,c,0)
        \end{cases}$$
    \end{enumerate}
    It is easy to see that all the configurations in the list $L$ can be reached from $(\Gamma,\psi)$ by internal slide moves and external slide moves. Moreover, for any configuration in $L$, no swap or connection is possible, and performing an (internal or external) slide yields another configuration in $L$. It follows that $L$ is a complete list of all the configurations that can be obtained from $(\Gamma,\psi)$ by slides, swaps, and connections. Given a GBS graph $(\Delta,\phi)$, in order to decide whether or not it encodes the same group as $(\Gamma,\psi)$, we change $(\Delta,\phi)$ into a totally reduced GBS graph (this can be done algorithmically), we perform sign-changes in all the possible (finitely many) ways, and we check whether or not one of the resulting graphs appears in the list $L$. Note that, since induction can not be performed on $(\Gamma,\psi)$, we do not have to worry about it in this case.
\end{ex}

\begin{ex}\label{ex4}
    Consider the GBS graph $(\Gamma,\psi)$ with one vertex $v$ and four edges, as in Figure \ref{fig:example4-GBSgraph}, and call $\Lambda$ its affine representation. The edges are partitioned into two different quasi-conjugacy classes $P\lecnj Q$, see Figure \ref{fig:example4-qcclasses}. The quasi-conjugacy class $P$ contains one edge and one minimal region $N=\{(1,0)\}$. The quasi-conjugacy class $Q$ contains three edges and two minimal regions $M_1=\{(k,1) : k\ge 1\}$ and $M_2=\{(0,3)\}$. We now prove that the linear invariants are sufficient to determine the isomorphism-type of $(\Gamma,\psi)$.

    Suppose that we are given another GBS graph $(\Gamma',\psi')$ with the same linear invariants. Of course we must have an edge $(1,0)\edge(4,0)$ in $P$. Let $e_1,e_2,e_3$ be the three edges in $Q$, with $e_1\cnj(2,1)$ and $e_2\cnj(3,0)$ and $e_3\cnj(2,2)$. We observe that, by Lemma \ref{lem:escape-minimal-regions}, there must always be at least one edge with an endpoint in $M_2$, and by looking at the conjugacy classes, this edge must be $e_2$. We consider three cases.

    CASE 1: The other endpoint of $e_2$ falls inside $M_1$. In this case $e_2$ must be given by $(0,3)\edge(a,1)$ for some integer $a\ge1$, and up to external slide we can assume that $e_2$ is given by $(0,3)\edge(1,1)$. If we now keep $e_2$ fixed, we can move the endpoints of $e_1,e_3$ at will inside their conjugacy class. In particular, we can arrive at the configuration with $e_1$ given by $(1,3)\edge(2,1)$ and $e_3$ given by $(1,4)\edge(2,2)$, from which we can reach $(\Gamma,\psi)$.

    CASE 2: $e_2$ is given by $(0,3)\edge(a,b)$ for some integers $(a,b)\ge(1,3)$. By \Cref{lem:escape-minimal-regions}, we obtain that $e_1$ must have one endpoint inside $M_1$. It can not have both endpoints inside $M_1$, otherwise $M_1$ would be a quasi-conjugacy class by itself, contradiction; up to external slide, we can assume that $e_1$ is of the form $(2,1)\edge(c,d)$ for some integers $(c,d)\ge(0,3)$. Up to external slide, we can obtain $(a,b)\ge(0,3),(2,1)$, and up to sliding $e_1$ along $e_2$, we can obtain $(c,d)\ge(0,3),(2,1)$. Using slides, we can now make all the edges very long, and then we can use \Cref{Nielsen-equiv-3big} to reach $(\Gamma,\psi)$ - here we are using that the quotient $\cla{Q}/\gen{(3,0)}$ is $2$-generated, and thus every triple of generators is Nielsen equivalent to any other.
    
    CASE 3: $e_2$ is given by $(0,3)\edge(0,a)$ for some integer $a\ge3$. We can not have $a=3$, otherwise $M_2$ would be a quasi-conjugacy class by itself, contradiction; thus $a\ge4$. One of $e_1,e_3$ must be of the form $(0,b)\edge(c,d)$ for integers $b\ge3$ and $(c,d)\ge(1,1)$, otherwise the set $\{(0,k) : k\ge3\}$ would be a (union of) quasi-conjugacy class(es) by itself, contradiction. But then we can perform a connection move and reduce ourselves to case 2.

    To summarize, a GBS graph can be obtained from $(\Gamma,\psi)$ by means of a sequence of slides, swaps, and connections, if and only if it has the same linear invariants as $(\Gamma,\psi)$. This characterizes the isomorphism problem for $(\Gamma,\psi)$. As usual, it remains to check the (finitely many) sign-changes. Note that induction does not apply to this example.
\end{ex}

\begin{figure}[H]
\centering
\includegraphics[width=0.8\textwidth]{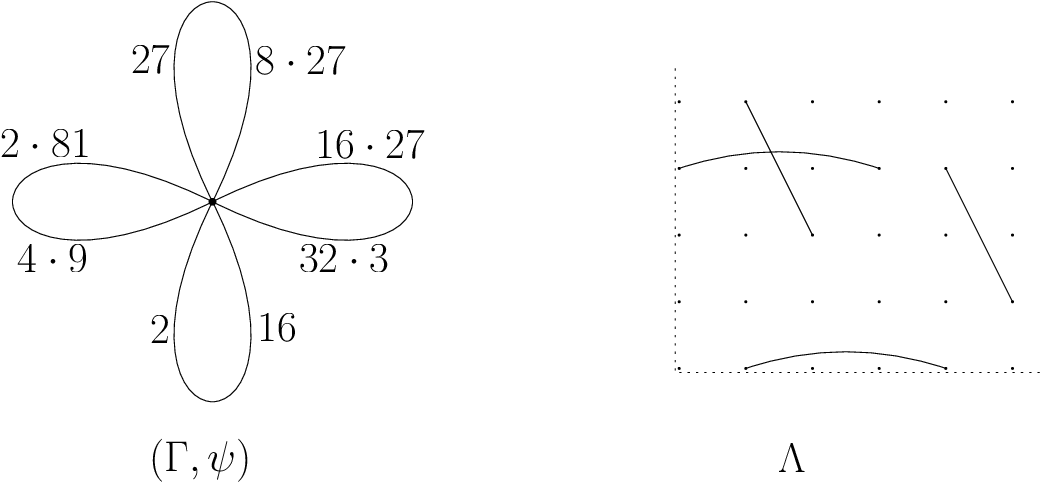}
\centering
\caption{The GBS graph $(\Gamma,\psi)$ and the affine representation $\Lambda$ of \Cref{ex4}.}
\label{fig:example4-GBSgraph}
\end{figure}

\begin{figure}[H]
\centering
\includegraphics[width=0.8\textwidth]{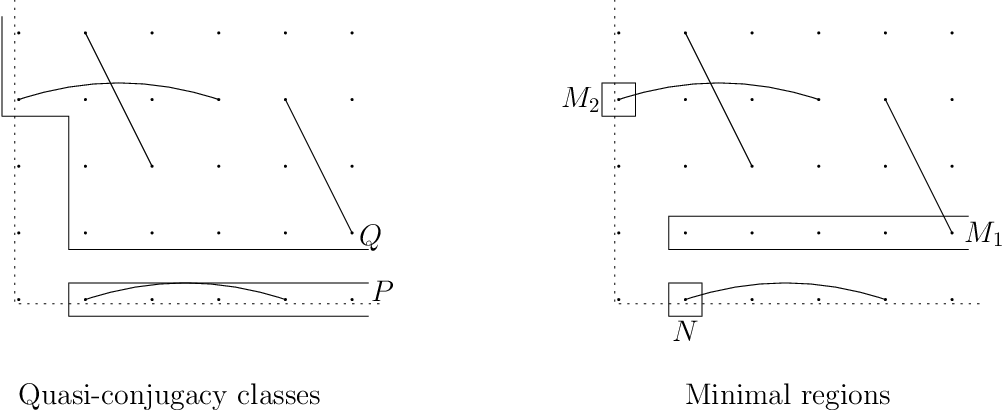}
\centering
\caption{The quasi-conjugacy classes and the minimal regions for the GBS graph of \Cref{ex4}.}
\label{fig:example4-qcclasses}
\end{figure}

\section{GBSs with one qc-class and full-support gaps}\label{sec:one-qc-class-full-supp-gaps}

The aim of this section is to classify GBS graphs with one qc-class and full-support gaps.

\subsection{One qc-class}

\begin{defn}\label{def:one-qc-class}
    Let $(\Gamma,\psi)$ be a GBS graph with affine representation $\Lambda$. We say that $(\Gamma,\psi)$ has \textbf{one qc-class} if all of its edges belong to the same quasi-conjugacy class.
\end{defn}

Note that the property of having one qc-class is an isomorphism invariant. In this case the external equivalence relation is trivial in that quasi-conjugacy class, and in particular the minimal regions are essentially points (to be formal, couples of points, with only the $\bbZ/2\bbZ$ component changing). \textit{Along this section, we will abuse notation and talk about minimal points instead of minimal regions.}

In what follows, we will also assume that the quasi-conjugacy class has non-trivial quasi-conjugacy support: for a quasi-conjugacy class $Q$ with $\qcsupp{Q}=\emptyset$, the isomorphism problem can be solved using \Cref{cor:qcclasses-empty-support}.

\subsection{Graph of families}

Let $(\Gamma,\psi)$ be a totally reduced GBS graph with affine representation $\Lambda$. Suppose that $(\Gamma,\psi)$ has one qc-class $Q$.

In \Cref{sec:one-qc-class-full-supp-gaps}, we consider the graph of families $\cF(Q)$ as defined in \Cref{sec:graph-of-families}. In this setting, it has following properties:
\begin{enumerate}
    \item Each minimal point of $Q$ belongs to exactly one family. Each family contains at least one minimal point. Therefore, the finite set of minimal points of $Q$ is partitioned into $\abs{V(\cF(Q))}$ non-empty subsets.
    \item According to \Cref{lem:family-unique-edge}, we have exactly one edge in $\cF(Q)$ for every edge in $\Gamma$. In other words, we have that $E(\cF(Q))=E(\Gamma)=E(\Lambda)$.
    \item We have the modular homomorphism $q_Q:\pi_1(\cF(Q))\rightarrow\bA$ as defined in \Cref{sec:graph-of-families}.
\end{enumerate}

\begin{defn}\label{def:individual-family}
    A family $F\in V(\cF(Q))$ is called \textbf{individual} if it contains exactly one minimal point, and \textbf{non-individual} otherwise.
\end{defn}

\begin{rmk}
    In the particular case when $\qcsupp{Q}=\cP(\Gamma,\psi)$, the graph of families is exactly $\cF(Q)=\Gamma$. For simplicity, the reader can go through this whole section with this extra hypothesis in mind (this extra assumption does not make any of the arguments easier).
\end{rmk}

\subsection{One endpoint of an edge at each minimal point}\label{sec:projection-clean}

Let $(\Gamma,\psi)$ be a totally reduced GBS graph with affine representation $\Lambda$. Suppose that $(\Gamma,\psi)$ has one qc-class $Q$ with $\qcsupp{Q}\not=\emptyset$.

\begin{defn}
    We say that $(\Gamma,\psi)$ is \textbf{clean} if every minimal point contains exactly one endpoint of exactly one edge, and no edge has both endpoints at minimal points.
\end{defn}

Since we are assuming that $\qcsupp{Q}\not=\emptyset$, we can always arrange for the GBS graph to be clean, as we now explain. Define the graph $\cL$ as follows:
\begin{itemize}
    \item $V(\cL)$ is the set of minimal points of $(\Gamma,\psi)$, together with an extra vertex $v_\floating$.
    \item $E(\cL)=E(\Lambda)$, with the same reverse map. For $e\in E(\Lambda)$, if $\tau_\Lambda(e)$ is a minimal point then we set $\tau_\cL(e)=\tau_\Lambda(e)$, otherwise we set $\tau_\cL(e)=v_\floating$.
\end{itemize}
Note that loops in $\cL$ at $v_\floating$ correspond to floating edges in $(\Gamma,\psi)$. Since $(\Gamma,\psi)$ has one qc-class $Q$, we must have that $\cL$ is connected. Note that $(\Gamma,\psi)$ is clean if and only if in $\cL$ every minimal point $m$ has valence one, and the only edge at $m$ connects $m$ to $v_\floating$ (see \Cref{fig:cleanprojection}).

We now define a \textit{clean projection} from $(\Gamma,\psi)$ to a clean GBS graph. Choose a maximal tree of $\cL$. Define the GBS graph $(\Gamma',\psi')$ as follows. If $e\in E(\cL)$ lies in the maximal tree, we consider its two endpoints $\iota(e),\tau(e)$, we take the one that lies nearer to $v_\floating$, and we slide it along the maximal tree until it ends up at $v_\floating$. We do this for all edges that lie in the chosen maximal tree: after this, the maximal tree is a star around $v_\floating$. Now, if $e\in E(\cL)$ is an edge which does not lie in the maximal tree, then we change $\tau(e)$ with a single slide move, along a single edge in the maximal tree, so that it ends up at $v_\floating$. We do this for all endpoints of all edges that do not lie in the maximal tree. The result of this procedure is a clean GBS graph $(\Gamma',\psi')$.

\begin{figure}[H]
\centering
\includegraphics[scale=0.6]{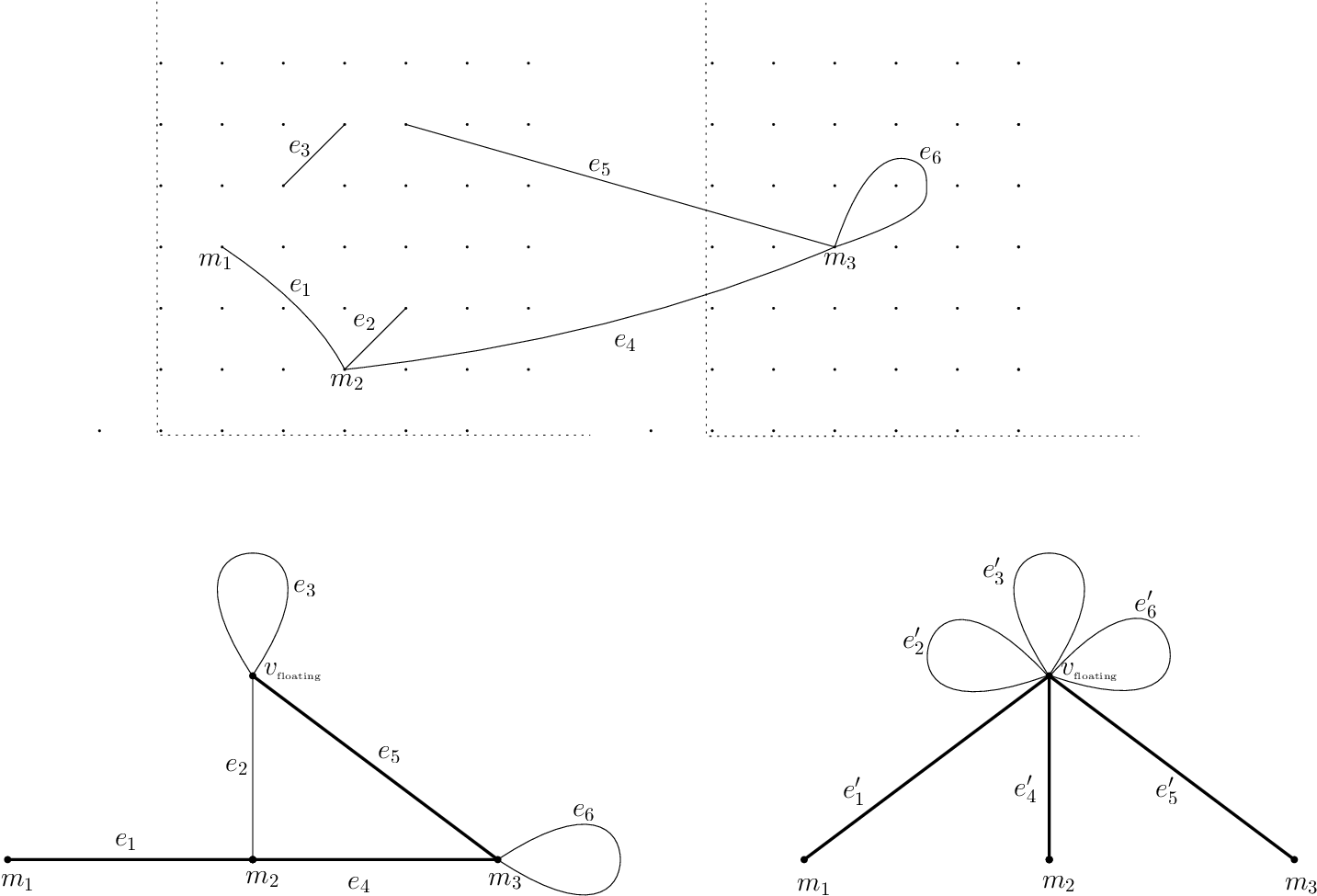}
\caption{On the top, an affine representation of a GBS graph $(\Gamma, \psi) $(with one qc-class and non-empty support); on the bottom left, the associated graph $\mathcal L$, with thick edges describing a maximal tree; on the bottom right, the clean projection obtained from $(\Gamma, \psi)$.}
\label{fig:cleanprojection}
\end{figure}

\begin{lem}\label{lem:clean-projection-1}
    The above clean projection has the following properties:
    \begin{enumerate}
        \item\label{ahfa} For a fixed maximal tree, performing the slide moves in a different order will produce the same clean GBS graph $(\Gamma',\psi')$.
        \item\label{pgoa} For different choices of maximal trees, producing clean GBS graphs $(\Gamma',\psi'),(\Gamma'',\psi'')$, there is a sequence of edge sign-changes and slides going from $(\Gamma',\psi')$ to $(\Gamma'',\psi'')$ and such that every GBS graph along the sequence is clean.
    \end{enumerate}
\end{lem}
\begin{proof}
    Statement \ref{ahfa} is immediate from the definitions.

    For statement \ref{pgoa}, we consider the following particular case first. Suppose that we are given an embedded cycle $e_1e_2\dots e_\ell$ in $\cL$ such that $\iota_\cL(e_1)=\tau_\cL(e_\ell)=v_\floating$. Suppose that a maximal tree $T$ for $\cL$ contains all the cycle except for $e_k$, for some $1\le k\le \ell-1$. Suppose that the maximal tree $T'$ for $\cL$ is obtained from $T$ by removing $e_k$ and adding $e_{k+1}$. We want to prove that $T$ and $T'$ give the same clean projection (up to a sequence of slides and edge sign-changes going only through clean GBS graphs).

    Let $p=\iota_\Lambda(e_1)$ and $q=\tau_\Lambda(e_\ell)$. Write $\tau_\Lambda(e_i)=\iota_\Lambda(e_{i+1})+\bepsilon_i$ for $\bepsilon_i\in\bA$ torsion element, for $i=1,\dots,\ell-1$. The edges $e_i$ for $i\not=k,k+1$ give the same result, when changed using the maximal tree $T$ or using the maximal tree $T'$. The edges $e_k,e_{k+1}$ are changed into
    \[
        \begin{cases}
            \tau_\Lambda(e_k)\edge p+\bepsilon_1+\dots+\bepsilon_{k-1}\\
            p+\bepsilon_1+\dots+\bepsilon_k\edge q+\bepsilon_{k+1}+\dots+\bepsilon_{\ell-1}
        \end{cases}
       \text{or}\hspace{0.25cm}
        \begin{cases}
            p+\bepsilon_1+\dots+\bepsilon_{k-1}\edge q+\bepsilon_k+\dots+\bepsilon_{\ell-1}\\
            \iota_\Lambda(e_{k+1})\edge q+\bepsilon_{k+1}+\dots+\bepsilon_{\ell-1}
        \end{cases}
    \]
    when using the maximal tree $T$ or when using the maximal tree $T'$, respectively. It is immediate to see that we can pass from one configuration to the other by a slide move and possibly an edge sign-change move (while keeping the GBS graph clean). All the edges different from $e_1,\dots,e_\ell$ are changed into different expressions by using $T$ or $T'$, but each endpoint can be adjusted with at most a slide move (while keeping the GBS graph clean).

    Thus we can take an embedded cycle, contained in $T$ except for one edge, and passing through $v_\floating$, and we can change the choice of which edge is outside $T$. A similar computation shows that the same is true for embedded cycles not passing through $v_\floating$. But by means of moves like that, we can pass from any maximal tree to any other maximal tree. The result follows.
\end{proof}

\begin{lem}\label{lem:clean-projection-2}
    Let $(\Gamma,\psi),(\Delta,\phi)$ be related by a edge sign-change, slide, swap or connection; let $(\Gamma',\psi'),(\Delta',\psi')$ be their respective clean projections. Then there is a sequence of edge sign-changes, slides, swaps, and connections passing from $(\Gamma',\psi')$ to $(\Delta',\psi')$ and going only through clean GBS graphs.
\end{lem}
\begin{proof}
    This is an easy check with the definitions, using the above \Cref{lem:clean-projection-1}.
\end{proof}

Note that \Cref{lem:clean-projection-1} and \Cref{lem:clean-projection-2} allow us to assume without loss of generality that the GBS graphs are clean and that all the sequences of moves go only through clean GBS graphs. Accordingly, we shall assume this for the rest of this section.

In particular, if $(\Gamma,\psi)$ is a GBS graph with one qc-class $Q$ and $\qcsupp{Q}\not=\emptyset$, then a strong version of \Cref{lem:escape-minimal-regions} holds. To be precise, up to taking the clean projection of $(\Gamma,\psi)$, for every minimal point $m$ we have exactly one edge with exactly one endpoint at $m$, and all the other edges are floating edges. It follows that a clean GBS graph $(\Gamma,\psi)$ has at least as many edges as minimal points; moreover, the inequality is strict if and only if $(\Gamma,\psi)$ has floating edges. This motivates the following definition.

\begin{defn}
    Let $(\Gamma,\psi)$ be a GBS graph with one qc-class $Q$ and with $\qcsupp{Q}\not=\emptyset$. We say that $(\Gamma,\psi)$ has \textbf{floating pieces} if it has strictly more edges than minimal points.
\end{defn}

\subsection{One endpoint in each individual family} \label{sec:projection-polished}

Let $(\Gamma,\psi)$ be a totally reduced GBS graph with affine representation $\Lambda$. Suppose that $(\Gamma,\psi)$ has one qc-class $Q$ with $\qcsupp{Q}\not=\emptyset$.

\begin{defn}
    If $F\in V(\cF(Q))$ is an individual family {\rm(\Cref{def:individual-family})}, then we say that $(\Gamma,\psi)$ is \textbf{$F$-polished} if $F$ contains exactly one endpoint of one edge.
\end{defn}

Let $F\in V(\cF(Q))$ be an individual family, and let $m$ be the unique minimal point in $F$. A \textbf{set of $F$-polishing data} is given by pairwise distinct edges $e_1,\dots,e_n,e$ satisfying the following properties:
\begin{enumerate}
\item For $j=1,\dots,n$ we have that $e_j$ goes from $m+\ba_j$ to $m+\bb_j$.
\item For $j=1,\dots,n$ we have $\supp{\ba_j}\subseteq(\supp{\bb_{j-1}}\cup\ldots \cup\supp{\bb_1})$.
\item The edge $e$ goes from $m+\bc$ to $p$ for some $p\not\ge m$.
\item We have $\supp{\bc}\subseteq(\supp{\bb_n}\cup\ldots \cup\supp{\bb_1})$.
\end{enumerate}

If $p$ belongs to a family $F'\in V(\cF(Q))$, then we say that the set of polishing data is \textbf{pointing to $F'$}. As in \cite[Section 6.3]{ACK-iso1} 
we can choose sufficiently large natural numbers $k_1,\dots,k_n$, and define
$$\bw_j=(\bb_j-\ba_j)+k_{j-1}(\bb_{j-1}-\ba_{j-1})+k_{j-1}k_{j-2}(\bb_{j-2}-\ba_{j-2})+\dots +k_{j-1}k_{j-2}\dots k_1(\bb_1-\ba_1)$$
for $j=1,\dots,n$, and $\bw=k_n\bw_n-\bc$. This allows us to define an \textbf{$F$-polishing procedure} that changes the edges
$$
\begin{cases}
m+\ba_1\edge  m+\bb_1\\
m+\ba_2\edge m+\bb_2\\
m+\ba_3\edge m+\bb_3\\
\dots \\
m+\ba_n\edge m+\bb_n\\
m+\bc \edge p
\end{cases}
\quad\text{into}\qquad
\begin{cases}
m+\ba_1 \edge p+\bw+\ba_1\\
p+\bw+\ba_2\edge p+\bw+\ba_2+\bw_1\\
p+\bw+\ba_3\edge p+\bw+\ba_3+\bw_2\\
\dots \\
p+\bw+\ba_n\edge p+\bw+\ba_n+\bw_{n-1}\\
p\edge  p+\bw_n\\
\end{cases}
$$
and all other endpoints of the form $m+\bd$ into $p+\bw+\bd$.

In the same way as in \cite[Section 6.3]{ACK-iso1}, 
we can see that for big enough $k_1,\dots,k_n$ the $F$-polishing procedure can be obtained through a sequence of moves, giving an $F$-polished GBS graph as a result (see \Cref{fig:polishing}).

\begin{figure}[H]
\centering
\includegraphics[scale=0.45]{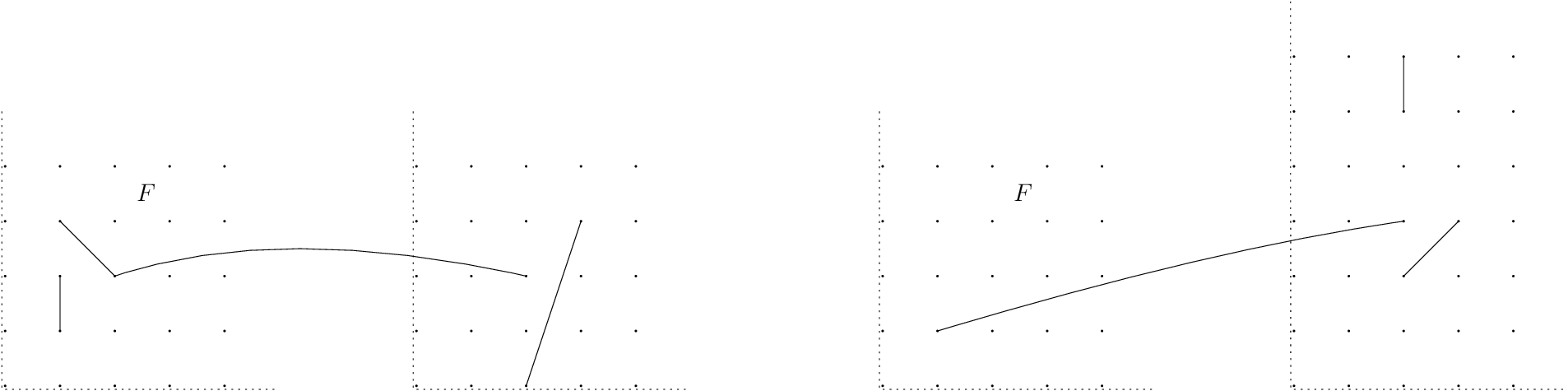}
\caption{On the left, an affine representation of a GBS graph (with one qc-class and non-empty support); on the right, an affine representation of an $F$-polished GBS graph.}
\label{fig:polishing}
\end{figure}

\begin{lem}\label{lem:polished-projection-1}
    For different choices of the set of $F$-polishing data and of sufficiently large constants, producing $F$-polished GBS graphs $(\Gamma',\psi'),(\Gamma'',\psi'')$, there is a sequence of edge sign-changes, slides, swaps, and connections going from $(\Gamma',\psi')$ to $(\Gamma'',\psi'')$ and such that every GBS graph along the sequence is $F$-polished.
\end{lem}
\begin{proof}
    Analogous to \cite[Proposition 6.16]{ACK-iso1}. 
\end{proof}

\begin{lem}\label{lem:polished-projection-2}
    Let $(\Gamma,\psi),(\Delta,\psi)$ be related by an edge sign-change, slide, swap or connection; let $(\Gamma',\psi'),(\Delta',\phi')$ be corresponging $F$-polished GBS graphs. Then there is a sequence of edge sign-changes, slides, swaps, and connections passing from $(\Gamma',\psi')$ to $(\Delta',\psi')$ and going only through $F$-polished GBS graphs.
\end{lem}
\begin{proof}
    Analogous to \cite[Proposition 6.17]{ACK-iso1}. 
\end{proof}

Note that, if $F,G\in V(\cF(Q))$ are two individual families, and if $(\Gamma,\psi)$ is $G$-polished, then there must be a set of $F$-polishing data pointing to some $F'\not=G$ (unless $F,G$ are the only families). If we use such a set of data to perform the above procedure, the result will be $G$-polished. Moreover, if we restrict only to set of data pointing to families $F'\not=G$, then \Cref{lem:polished-projection-1} and \Cref{lem:polished-projection-2} still hold, producing sequences of moves going only through $G$-polished GBS graphs. In particular, in what follows we can assume that all our GBS graphs are $F$-polished with respect to all individual families $F\in V(\cF(Q))$, and we can consider only sequences of moves going through $F$-polished GBS graphs.

\begin{remark}
    The $F$-polishing procedure is also compatible with the notion of clean. All of the above discussion, \Cref{lem:polished-projection-1} and \Cref{lem:polished-projection-2} still hold if we add the additional requirement that all GBS graphs appearing are clean. In particular, we can restrict our attention to GBS graphs which are clean and $F$-polished for all individual families $F\in V(\cF(Q))$, and to sequences of moves going only through GBS graphs of this kind.
\end{remark}

\subsection{Full-support gaps}\label{sec:full-support-gaps}

Let $(\Gamma,\psi)$ be a GBS graph with affine representation $\Lambda$. Suppose that we are given two vertices $p,q\in V(\Lambda)$ with $q\ge p$: then we can write $q=p+\bw$ for a unique $\bw\in\pA$, which we call the \textit{gap} from $p$ to $q$. We say that the gap from $p$ to $q$ is \textit{full-support} if $\supp{\bw}=\qcsupp{p}$ (which in particular implies that $q\qcnj p$).

\begin{defn}\label{def:full-support-gaps}
    Let $(\Gamma,\psi)$ be a GBS graph and let $Q$ be a quasi-conjugacy class. We say that $Q$ has \textbf{full-support gaps} if for every edge $e$ in $Q$ we have that either $\tau(e)$ lies in a minimal region, or $\tau(e)$ lies above a minimal point of $Q$ with full-support gap.
\end{defn}

This means that whenever an endpoint of an edge that falls strictly above a minimal region, the gap can be taken to have the full quasi-conjugacy support. Note that the property of being full-support gap is not an isomorphism invariant, nor invariant under performing moves.

\subsection{GBSs graphs with floating pieces}\label{sec:invariants-floating}

Let $(\Gamma,\psi)$ be a totally reduced GBS graph with affine representation $\Lambda$. Suppose that $(\Gamma,\psi)$ has one qc-class $Q$ with $\qcsupp{Q}\not=\emptyset$, and that $(\Gamma,\psi)$ is clean. Suppose that $(\Gamma,\psi)$ has floating pieces.

Fix a family $F_0\in V(\cF(Q))$. For every family $F\not=F_0$ we fix a minimal point $m_F$ in that family.

\begin{defn}[Normal form for GBSs with floating pieces]
    We say that $(\Gamma,\psi)$ is in \textbf{normal form} {\rm(}with respect to the family $F_0$ and the minimal points $\{m_F\}_{F\not=F_0}${\rm)} if the following conditions hold:
    \begin{enumerate}
        \item For every family $F\not=F_0$ and for the minimal point $m_F$, take the unique edge $e$ with $\iota(e)=m_F$. We require that $\tau(e)$ is in $F_0$, and lies above all minimal points of $F_0$ with full-support gap.
        \item For every family $F$ and for every minimal point $m$ {\rm(}except $m_F${\rm)}, take the unique edge $e$ with $\iota(e)=m$. We require that $\tau(e)$ is in $F$, and lies above all minimal points of $F$ with full-support gap.
        \item For every floating edge $e$, we require that $\iota(e),\tau(e)$ are in $F_0$, and lie above all minimal points of $F_0$ with full-support gap.
    \end{enumerate}
\end{defn}

\begin{figure}[H]
\centering
\includegraphics[scale=0.5]{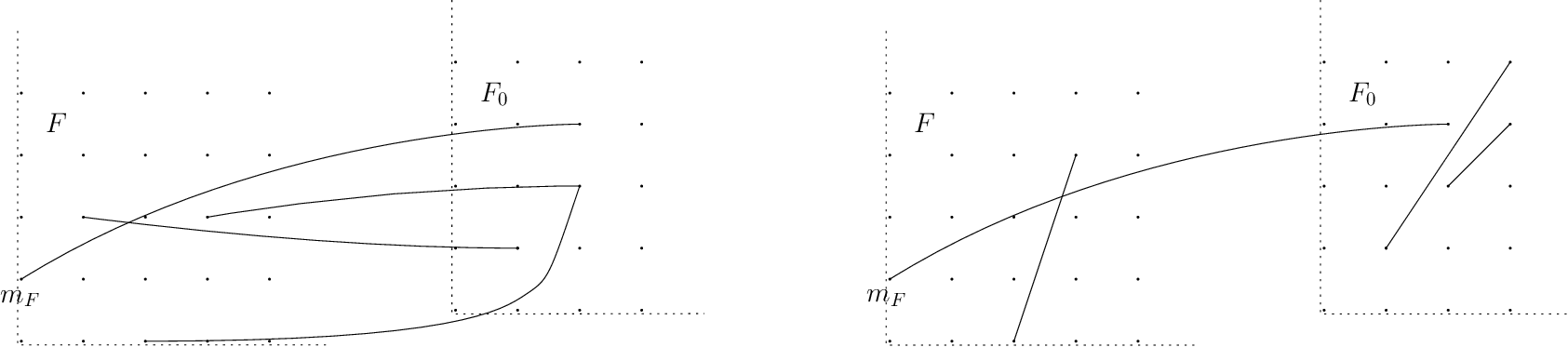}
\caption{On the left, an affine representation of a clean GBS graph (with one qc-class and non-empty support, and floating edges). On the right, the affine representation associated to a normal form.}
\label{fig:normalform}
\end{figure}

\begin{prop}\label{prop:existence-normal-form-floating}
    Suppose that $(\Gamma,\psi)$ has full-support gaps. Then, for every family $F_0$ and minimal points $\{m_F\}_{F\not=F_0}$, we can bring $(\Gamma,\psi)$ to normal form by a sequence of slides, swaps, and connections.
\end{prop}
\begin{proof}
    For every minimal point $m$, let $e_m$ be the unique edge such that $\iota(e_m)=m$. We say that $e_m$ is \textit{positive} if $\tau(e_m)$ is above $m$ with full-support gap.
    
    In steps 1 and 2, we maximize the number of positive edges, and we observe that the remaining edges at the minimal points form a tree in the graph of families $\cF(Q)$. In step 3, we change the tree, and we assume that all non-positive edges at the minimal points point at $F_0$. In step 4, we make all positive edges very long (taking advantage of floating edges). In step 5, we arrange for the non-positive edges to be at the prescribed minimal points $m_F$.
    
    STEP 1: If $e_m$ is not positive, then we choose a minimal point $m'\not=m$ such that $\tau(e_m)$ is above $m'$ with full-support gap. If $e_{m'}$ is not positive, then we choose $m''$ such that $\tau(e_{m'})$ is above $m''$ with full-support gap; and so on. If the sequence does not terminate with a positive edge, then we find a cycle, and by sliding one edge of the cycle along the others, we increase the number of positive edges. Note that the hypothesis of having full-support gaps is preserved.

    We reiterate this argument until, for every sequence $m,m',m'',\dots$, the sequence eventually ends at some $m^{(k)}$ with $e_{m^{(k)}}$ positive. At this point, we can slide $\tau(e_m)$ along $e_{m'},e_{m''},\dots$ in order to get that $\tau(e_m)$ is above $m^{(k)}$ with full-support gap. Similarly for all the other non-positive edges.

    Thus we can assume that, for all minimal points $m$, either the edge $e_m$ is positive, or $\tau(e_m)$ lies above some other minimal point $m'$ with full-support gap and with $e_{m'}$ positive.

    STEP 2: Now consider the graph of families $\cF(Q)$, and remove the floating edges to get a graph $\cF'$.
    
    Suppose that $\cF'$ contains an embedded (unoriented) closed cycle $\gamma$ containing a non-positive edge $e_m$. Thus $\tau(e_m)$ lies above some minimal point $m'$ with full-support gap and with $e_{m'}$ positive. By sliding $\tau(e_m)$ along $e_{m'}$ many times, we can assume that all components of $\tau(e_m)$ in $\qcsupp{Q}$ are very big. We can then slide $\tau(e_m)$ along the other edges of $\gamma$ until it becomes positive. This increases the number of positive edges.
    
    By reiterating this procedure, we can assume that every embedded closed cycle in $\cF'$ is a single loop consisting of a positive edge. This means that $\cF'$ is a tree plus some loops at the vertices consisting of positive edges.

    STEP 3: We now want all non-positive edges of $\cF'$ to point towards a prescribed family $F_0$. As we said before, non-positive edges of $\cF'$ form a tree. Suppose that $e_m$ is a non-positive edge, and $\tau(e_m)$ is further from $F_0$ than $\iota(e_m)$ (in the tree). Then we have that $\tau(e_m)$ lies above $m'$ such that $e_{m'}$ is positive, and thus we can perform a connection move. This substitutes $e_m,e_{m'}$ with $f_m,f_{m'}$ respectively, where $f_m$ is positive and $f_{m'}$ is non-positive; what we gain is that now $\tau(f_{m'})$ is nearer to $F_0$ than $\iota(f_{m'})$. If some other non-positive edge $e_{m''}$ has $\tau(e_{m''})$ above $m'$, then we now slide $\tau(e_{m''})$ along $f_{m'}$, in such a way that $\tau(e_{m''})$ ends up above $m$ with positive edge $f_m$.

    Thus we can assume that all non-positive edges point towards $F_0$. For every non-positive edge $e_m$, we have that $\tau(e_m)$ is above a positive edge, and thus up to slide we can make all components of $\tau(e_m)$ in $\qcsupp{Q}$ very big, and then slide $\tau(e_m)$ so that it falls in $F_0$, and above all minimal points of $F_0$ with full-support gap.

    Thus now $\cF'$ is made of non-positive edges (one for each family $F\not=F_0$, all pointing towards $F_0$, and with terminal vertex having all components in $\qcsupp{Q}$ very big), and of positive edges at every other minimal point.

    STEP 4: Up to slides, we can assume that every floating edge has both endpoints in $F_0$.

    Let $e_m$ be a positive edge given by $m\edge m+\bu$ for $\bu\in\pA$ with $\supp{\bu}=\qcsupp{Q}$. Take a floating edge $e$ and, up to slides, assume that $\iota(e)$ lies above $m$ and $\tau(e)=\iota(e)+\bw$ for some $\bw\in\pA$ with $\supp{\bw}=\qcsupp{Q}$ and with all non-zero components of $\bw$ very big. Then we perform a swap move, and we get an edge $m\edge m+\bw$ and a floating edge. In other words, we can assume that all components in $\qcsupp{Q}$ of all positive edges are very big.

    STEP 5: For a family $F\not=F_0$, take a prescribed minimal point $m_F$. If $e_{m_F}$ is positive, then there is another minimal point $m'$ with $e_{m'}$ non-positive. We have that $\tau(e_{m'})$ lies above some minimal point $m''$ in $F_0$, with $e_{m''}$ positive. We now perform a connection to change $e_{m'},e_{m''}$ into $f_{m'},f_{m''}$ with $f_{m'}$ positive. Then we perform a connection to change $e_{m_F},f_{m''}$ into $g_{m_F},g_{m''}$ with $g_{m''}$ positive. In this way, the edges at $m',m''$ are now positive, and the edge at $m_F$ is not.

    Thus the resulting GBS graph is in normal form, as desired.
\end{proof}

\begin{prop}\label{prop:uniqueness-normal-form-floating}
    Suppose that $(\Gamma,\psi),(\Delta,\phi)$ are in normal form. Suppose that they have the same $\qcmin{Q},\qcsupp{Q},\cla{Q}$ and the same number of edges in each conjugacy class {\rm(}up to edge sign-changes{\rm)}. Then there is a sequence of edge sign-changes, slides, swaps, and connections going from one to the other.
\end{prop}
\begin{proof}
    Let $\Lambda$ be the affine representation of $(\Gamma,\psi)$. We have the edges $m_F\edge c_F$ for $F\not=F_0$ and $c_F\in F_0$. We can easily assume that all the other edges are given by $p_i\edge p_i+\bx_i$ with $\bx_i\in\pA$ and $\supp{\bx_i}=\qcsupp{Q}$, for $i=1,\dots,r$.

    It is fairly easy to see that, using swap moves, and up to translating the floating edges with slide moves, we can permute the vectors $\bx_1,\dots,\bx_r$ as we want. Once we are able to permute them as we want, we can also perform Nielsen moves among them (since we are able, for example, to change a floating edge by slide over another edge), provided that we keep all components in $\qcsupp{Q}$ of all $\bx_i$ big enough, so that the graph remains in normal form.

    By \Cref{prop:cla-and-modular-homomorphism} we have that $\bx_1,\dots\bx_r$ generate $\cla{Q}$. But $\cla{Q}$ is isomorphic to either $\bbZ^k$ or $\bbZ/2\bbZ\oplus\bbZ^k$ for some $k\in\bbN$. Thus in this case all $r$-tuples of generators for $\cla{Q}$ are Nielsen equivalent to each other up to permutation. Therefore, using the results of \Cref{sec:Nielsen-equiv-big}, we can perform moves to make $\bx_1,\dots,\bx_r$ to be the same in $(\Gamma,\psi)$ and $(\Delta,\phi)$.

    Now, the $p_i$ which are minimal points are already the same in $(\Gamma,\psi)$ and $(\Delta,\phi)$, by assumption (up to performing some edge sign-change). The $p_i$ which are non-minimal can be made to be the same by slides and self-slides (\Cref{lem:self-slide}) and permutations of $\bx_1,\dots,\bx_r$, since we are assuming that there are the same number of edges in each conjugacy class in $(\Gamma,\psi)$ and $(\Delta,\phi)$.

    Finally, we can change the endpoints $c_F$ by any combination of $\bx_1,\dots,\bx_r$ (since we can slide $c_F$ over at least one $\bx_i$, and we can permute $\bx_1,\dots,\bx_r$ as we want in any moment). Since $(\Gamma,\psi)$ and $(\Delta,\phi)$ have the same conjugacy classes, we can arrange for the $c_F$ to be the same too.

    The conclusion follows.
\end{proof}

\subsection{Rigid vectors}\label{sec:rigid-vectors}

Let $(\Gamma,\psi)$ be a totally reduced GBS graph with affine representation $\Lambda$. Suppose that $(\Gamma,\psi)$ has one qc-class $Q$ with $\qcsupp{Q}\not=\emptyset$, and that $(\Gamma,\psi)$ is clean.  Suppose that $(\Gamma,\psi)$ does not have floating pieces.

\begin{defn}\label{def:rigid-vector}
    A \textbf{rigid cycle} is a sequence of $\ell\ge1$ edges $m_1\edge m_2+\bw_1$, $m_2\edge m_3+\bw_2$, $\dots$, $m_\ell\edge m_1+\bw_\ell$ with the following properties:
    \begin{enumerate}
        \item $m_1,\dots,m_\ell$ are pairwise distinct minimal points.
        \item $\bw_1,\dots,\bw_\ell\in\pA$.
        \item\label{itm:rigidity} There are no distinct minimal points $m,m'$ such that $m+\bw_1+\dots+\bw_\ell\ge m'$.
    \end{enumerate}
    For a rigid cycle we define the associated \textbf{rigid vector} as $\bw=\bw_1+\dots+\bw_\ell\in\pA$.
\end{defn}

In the above hypothesis, we must have that $\supp{\bw}\not=\emptyset$ since $(\Gamma,\psi)$ is clean. In a rigid cycle, every vertex of every edge lies above a unique minimal point. In fact, if $m_{i+1}+\bw_i\ge m'$ for some minimal point $m'$, then in particular $m_{i+1}+\bw\ge m'$, and by \Cref{itm:rigidity} of \Cref{def:rigid-vector} this implies that $m'=m_{i+1}$. In particular, an edge can be part of at most one rigid cycle. Note also that $\bw\in\cla{Q}$, since we can, for example, slide one edge of the rigid cycle along the others to obtain an edge $m\edge m+\bw$.

\begin{defn}\label{def:multi-set-rigid-vectors}
    Define the set of \textbf{rigid vectors} of $(\Gamma,\psi)$ as the multi-set of the vectors $\bw$ associated with the rigid cycles of $(\Gamma,\psi)$.
\end{defn}

\begin{prop}
    The multi-set of rigid vectors of $(\Gamma,\psi)$ is invariant under edge sign-changes, slides, connections.
\end{prop}
\begin{proof}
It is routine to check using the definitions.
\end{proof}

\subsection{GBS graphs with no floating pieces}\label{sec:invariants-no-floating}

Let $(\Gamma,\psi)$ be a totally reduced GBS graph with affine representation $\Lambda$. Suppose that $(\Gamma,\psi)$ has one qc-class $Q$ with $\qcsupp{Q}\not=\emptyset$, and that $(\Gamma,\psi)$ is clean.  Suppose that $(\Gamma,\psi)$ does not have floating pieces.

We denote with $\cM$ the finite set of minimal points of $Q$. This is partitioned into the non-empty sets $\cM_F$ for $F\in V(\cF(Q))$, according to which family a given minimal point belongs. We give an orientation to each edge of $\Lambda$: we say that $e$ is the \textit{oriented edge} of a pair $e,\ol{e}\in E(\Lambda)$ if $\iota(e)$ is a minimal point. Note that exactly one of $e,\ol{e}$ is the oriented edge. Since $(\Gamma,\psi)$ is clean and does not have floating pieces, we have a bijection between the set of oriented edges and $\cM$: the oriented edge $e$ corresponds to the minimal point $\iota(e)\in\cM$. Recall that $E(\cF(Q))=E(\Lambda)$, and in particular $\cF(Q)$ is an oriented graph. We fix a family $F_0\in V(\cF(Q))$, which will play the role of basepoint.

Choose a maximal tree $T$ for $\cF(Q)$, and let $t_1,\dots,t_h$ the oriented edges in $T$. Let $\bR=\{\br_1,\dots,\br_k\}$ be the multi-set of rigid vectors as in \Cref{def:multi-set-rigid-vectors}. For $i=1,\dots,k$, choose an oriented edge $e_i$ in the rigid cycle corresponding to $\br_i$, and such that $e_i$ does not belong to $T$. Note that such an $e_i$ exists, since the rigid cycle is a cycle in $\cF(Q)$; note that distinct rigid cycles give distinct edges, as they have no edges in common. Call $f_1,\dots,f_\ell$ the remaining oriented edges. This means that $e_1,\dots,e_k,f_1,\dots,f_\ell,t_1,\dots,t_h$ is a list of all the oriented edges of $\cF(Q)$, each appearing exactly once. In particular $\cM=\{\iota(e_1),\dots,\iota(e_k),\iota(f_1),\dots,\iota(f_\ell),\iota(t_1),\dots,\iota(t_h)\}$.

For $i=1,\dots,\ell$ we consider the path $\sigma_i$ which goes through $f_i$ and then back to the starting point through the maximal tree $T$. For $i=1,\dots,\ell$ we define $\bx_i=q_Q(\sigma_i)+\gen{\bR}\in\frac{\cla{Q}}{\gen{\bR}}$. We consider the tree $T$ with base-point $F_0$: every edge has an orientation, which is either pointing towards $F_0$, or away from $F_0$. We call $\epsilon\in\bbN$ the number of edges pointing away from $F_0$.

We define the map $\Theta:\cM\rightarrow \bR\sqcup\frac{\cla{Q}}{\gen{\bR}}\sqcup \left(V(\cF(Q))\setminus \{F_0\}\right)$ as follows:
\begin{enumerate}
    \item $\Theta(\iota(e_i))=\br_i\in\bR$ for $i=1,\dots,k$.
    \item $\Theta(\iota(f_i))=\bx_i\in\frac{\cla{Q}}{\gen{\bR}}$ for $j=1,\dots,\ell$.
    \item $\Theta(\iota(t_i))\in V(\cF(Q))\setminus \{F_0\}$ is defined as follows. If $t_i$ is pointing towards $F_0$, we set $\Theta(\iota(t_i))$ as the family of $\iota(t_i)$. If $t_i$ is pointing away from $F_0$, we set $\Theta(\iota(t_i))$ as the family of $\tau(t_i)$.
\end{enumerate}
We define the map $(-1)^\epsilon\Theta$ by pre-composing $\Theta$ with any permutation of $\cM$ of sign $(-1)^\epsilon$.

\begin{defn} \label{defn:assignmentmap}
    We define the \textbf{assignment map} of $(\Gamma,\psi)$ as the map $(-1)^\epsilon\Theta$, considered up to
    \begin{enumerate}
        \item Nielsen equivalence of $\bx_1,\dots,\bx_\ell$ in the group $\frac{\cla{Q}}{\gen{\bR}}$.
        \item Pre-composing $\Theta$ with an even permutation of $\cM$.
    \end{enumerate}
\end{defn}

The above construction can look artificial, but it becomes more natural if one thinks about the normal form introduced below, see \Cref{def:normal-form-no-floating}. We encourage the reader to think like this: for an oriented edge $e$ with $\iota(e)=m$ minimal, the map $\Theta(m)$ is telling us what the role of the edge $e$ in the given configuration is. If $\Theta(m)=\br\in\bR$, then you should imagine that $e$ is hosting the rigid vector $\br$ (as if $e$ would be $m\edge m+\br$). If $\Theta(m)=\bx\in\frac{\cla{Q}}{\gen{\bR}}$, then you should imagine that $e$ is hosting a ``very long'' vector $\bx$ (as if $e$ would be $m\edge m+\bx$); here ``very long'' means that it is able to play with he others through slide moves). If $\Theta(m)=F\in V(\cF(Q))\setminus \{F_0\}$, then you should imagine that $e$ is the edge connecting $F$ to the basepoint $F_0$ (as if $e$ would be connecting a minimal point in $F$ to a non-minimal point in $F_0$). And in fact, when $(\Gamma,\psi)$ is in the normal form of \Cref{def:normal-form-no-floating}, this is exactly what the map $\Theta$ is encoding.

\begin{lem}\label{lem:assignment-indep-rigid-edges}
    A different choice of the edges in the rigid cycles produces the same assignment map.
\end{lem}
\begin{proof}
    Let $\gamma$ be a rigid cycle with rigid vector $\br$, let $T$ be a maximal tree, and let $s_1,s_2,\dots,s_h$ be the oriented edges of $\gamma$ that do not belong to $T$.
    
    For $i=1,\dots,k$, we call $\sigma_i$ the path that crosses $s_i$ and then closes through the maximal tree $T$, and let $\bx_i=q_Q(\sigma_i)+\gen{\bR}\in\cla{Q}/\gen{\bR}$. The key observation is that $\bx_1+\dots+\bx_k=\br$.
    
    It follows that, if we choose $s_i$ to be the edge of the rigid cycle, then we must set $\Theta(\iota(s_i))=\br$ and $\Theta(\iota(s_j))=\bx_j$ for $j\not=i$. The other values of $\Theta$ are independent of $i$. In particular, a different choice of $i$ will produce the same map $\Theta$ up to even permutation and Nielsen equivalence.
\end{proof}

\begin{lem}\label{lem:assignment-indep-max-tree}
    A different choice of maximal tree $T$ produces the same assignment map.
\end{lem}
\begin{proof}
    Take an embedded cycle $\gamma$ in $\cF(Q)$, and call $\bx=q_Q(\gamma)+\gen{\bR}\in\cla{Q}/\gen{\bR}$. Suppose that $\gamma=\dots s_1^{\eta_1}\dots s_2^{\eta_2}\dots$ for some oriented edges $s_1,s_2$ and for some $\eta_1,\eta_2=\pm1$. Suppose that some maximal tree $T_1$ contains all of $\gamma$ except $s_2$, and let $T_2$ be the maximal tree obtained from $T_1$ by adding $s_2$ and removing $s_1$. Let $\Theta_1,\Theta_2$ be the maps associated with $T_1,T_2$ respectively.

    For an edge $e$ not appearing in $\gamma$, we have two cases. If $e\in T_1$ then $e\in T_2$ and $\Theta_1(\iota(e))=\Theta_2(\iota(e))$ and the contribution to $\epsilon$ does not change. If $e\not\in T_1$ then $e\not\in T_2$ and $\Theta_1(\iota(e))-\Theta_2(\iota(e))$ is $\mathbf{0}$ or $\pm\bx$.

    For the edges $e$ appearing in $\gamma$, the values of $\Theta_1,\Theta_2$ will be $\pm\bx$ and $F_1,F_2,\dots,F_r$ (the vertices appearing in $\gamma$, except the one nearest to $F_0$), in some order. When changing from $T_1$ from $T_2$, these values get permuted (the permutation consisting of a single cycle), the contributions to $\epsilon$ might change, and $\pm\bx$ might change sign. It is a routine check that these changes compensate each other, i.e. $\pm\bx$ changes sign if and only if the sign of the permutation of the values, plus the sum of the contributions to $\epsilon$, is negative.

    Thus we can change from $T_1$ to $T_2$ without affecting the assignment map. But with changes of this kind, we can go from any tree to any other tree. The conclusion follows.
\end{proof}

\begin{lem}
    The assignment map of $(\Gamma,\psi)$ is invariant under edge sign-changes, slides, connections.
\end{lem}
\begin{proof}
    Immediate check with the definitions (using \Cref{lem:assignment-indep-rigid-edges} and \Cref{lem:assignment-indep-max-tree}).
\end{proof}

We now define a normal form for GBS graphs with no floating pieces. Fix a non-individual family $F_0\in V(\cF(Q))$. For every family $F\not=F_0$ we fix a minimal point $m_F$ in that family.

\begin{defn}[Normal form for GBSs with no floating pieces]\label{def:normal-form-no-floating}
    Suppose that $(\Gamma,\psi)$ has full-support gaps. We say that $(\Gamma,\psi)$ is in \textbf{normal form} {\rm(}with respect to the non-individual family $F_0$ and the minimal points $\{m_F\}_{F\not=F_0}${\rm)} if the following conditions hold:
    \begin{enumerate}
        \item For every family $F\not=F_0$ and for the minimal point $m_F$, take the unique edge $e$ with $\iota(e)=m_F$. We require that $\tau(e)$ is in $F_0$, and lies above all minimal points of $F_0$ with full-support gap.
        \item For every family $F$ and for every minimal point $m$ in $F$ {\rm(}except possibly $m_F${\rm)}, take the unique edge $e$ with $\iota(e)=m$. We require that $\tau(e)$ is in $F$. If $e$ is not a rigid cycle, we also require that $\tau(e)$ lies above all minimal points of $F$ with full-support gap.
    \end{enumerate}
\end{defn}

\begin{prop}\label{prop:existence-normal-form-no-floating}
    Suppose that $(\Gamma,\psi)$ has full-support gaps. Then, for every non-individual family $F_0$ and minimal points $\{m_F\}_{F\not=F_0}$, we can bring $(\Gamma,\psi)$ to normal form by a sequence of slides, swaps, and connections.
\end{prop}
\begin{proof}
    For every minimal point $m$, let $e_m$ be the unique edge such that $\iota(e_m)=m$. We say that $e_m$ is \textit{positive} if $\tau(e_m)$ is above $m$ with full-support gap.
    
    The first three steps in the proof of \Cref{prop:existence-normal-form-floating} apply here. Therefore, we can assume that every edge is positive, except for the edges $e_{m_F}$ for $F\not=F_0$, for which $\tau(e_{m_F})$ lies in $F_0$ and above all minimal points of $F_0$ with full-support gap.

    Let $F\not=F_0$ be a family containing at least two minimal points $m,m_F$, and let $m_1,m_2$ be two minimal points in $F_0$. Then we have edges $m_1\edge m_1+\bx$ and $m_2\edge m_2+\by$ and $m\edge m+\bz$ and $m_F\edge p$ with $p$ above both $m_1,m_2$. Then we can perform connections to change
    \[
        \begin{cases}
            m_1\edge m_1+\bx\\
            m_2\edge m_2+\by\\
            m\edge m+\bz\\
            m_F\edge p
        \end{cases}
        \quad \text{into} \quad
        \begin{cases}
            m_1\edge q\\
            m_2\edge m_2+\by\\
            m\edge m+\bz\\
            m_F\edge m_F+\bx
        \end{cases}
        \quad \text{into} \quad
        \begin{cases}
            m_1\edge m_1+\bz\\
            m_2\edge m_2+\by\\
            m\edge p'\\
            m_F\edge m_F+\bx
        \end{cases}
    \]
    \[
        \quad \text{into} \quad
        \begin{cases}
            m_1\edge m_1+\bz\\
            m_2\edge q'\\
            m\edge m+\by\\
            m_F\edge m_F+\bx
        \end{cases}
        \quad \text{into} \quad
        \begin{cases}
            m_1\edge m_1+\bz\\
            m_2\edge m_2+\bx\\
            m\edge m+\by\\
            m_F\edge p''
        \end{cases}
    \]
    ultimately getting a cyclic permutation of $\bx,\by,\bz$ (here $p,p',p''$ lie in $F_0$ above $m_1,m_2$, while $q,q'$ lie in $F$ above $m,m_F$).

    Now suppose that some positive edge $e_m$ given by $m\edge m+\bw$ is not a rigid edge, but $m+\bw$ does not lie above all minimal points in its family. Then we can find two minimal points $m',m''$ such that $m'+\bw\ge m''$. Using cyclic permutations as explained above, we can bring $\bw$ to $F_0$, and then to $m'$, so that we end up with an edge $m'\edge m'+\bw$. Using that $m'+\bw\ge m''$, we can perform slides and make all components of $\bw$ in $\qcsupp{Q}$ very big. Finally, using other cyclic permutations, we bring $\bw$ back to $m$.

    By reiterating this procedure, we eventually obtain a GBS graph in normal form, as desired.
\end{proof}

\begin{prop}\label{prop:uniqueness-normal-form-no-floating}
    Suppose that $(\Gamma,\psi),(\Delta,\phi)$ are in normal form. Suppose that the following conditions hold:
    \begin{enumerate}
        \item $(\Gamma,\psi),(\Delta,\phi)$ have the same linear invariants.
        \item $(\Gamma,\psi),(\Delta,\phi)$ have the same multi-set of rigid edges and the same assignment map.
        \item Either $F_0$ contains at least three minimal points, or there is at least another non-individual family besides $F_0$.
    \end{enumerate}
    Then there is a sequence of edge sign-changes, slides, swaps, and connections going from one to the other.
\end{prop}
\begin{proof}
    Up to edge sign-changes, we can assume that the minimal points are exactly the same.
    
    As in the proof of \Cref{prop:existence-normal-form-no-floating}, we can perform cyclic permutations on the vectors at the edges, using connections. In particular, we can assume that all rigid edges are at the same minimal points in $(\Delta,\phi)$ as in $(\Gamma,\psi)$.

    If we have a (possibly rigid) edge $m\edge m+\bx$ and a non-rigid edge $m'\edge m'+\by$, then with a sequence of slides we can change $\by$ into $\by\pm\bx$, provided that all components in $\qcsupp{Q}$ of $\by$ remain large enough. Now we deal with two cases: if there is at least one rigid edge, then we want to use \Cref{positive-vector}; if there is no rigid edge, then by hypothesis there are at least $3$ positive vectors, and we can use \Cref{Nielsen-equiv-3big}. In both cases, since $(\Gamma,\psi),(\Delta,\phi)$ have the same linear invariant, we can obtain that the values of $\bx$ are the same in $(\Delta,\phi)$ as in $(\Gamma,\psi)$.

    Finally, for $F\not=F_0$ and for the edge $m_F\edge c_F$, we can change $c_F$ to any other point in the same conjugacy class, by means of slide moves (and connections, if we would need to add/subtract some vector which lies in $F$). In particular, we can set also these to be the same in $(\Delta,\phi)$ as in $(\Gamma,\psi)$. The conclusion follows.
\end{proof}

\subsection{The algorithm}

\begin{thm}\label{thm:isomorphism-oneqcc-fsgaps}
    There is an algorithm that, given two GBS graph $(\Gamma,\psi),(\Delta,\phi)$ with one qc-class and full-support gaps, decides whether there is a sequence of edge sign-changes, slides, swaps, and connections going from one to the other. In case such a sequence exists, the algorithm also computes one such sequence.
\end{thm}
\begin{proof}
    First of all, we check that the conjugacy class $Q$ has the same $\qcmin{Q},\qcsupp{Q},\cla{Q}$ and the same number of edges in each conjugacy class (up to edge sign-changes), otherwise by \Cref{cor:basic-invariants} we can not change $(\Gamma,\psi)$ into $(\Delta,\phi)$.

    If $(\Gamma,\psi)$ has floating pieces, then $(\Delta,\phi)$ must have floating pieces too. We make both of them clean and we compute the linear invariant: this must be the same, otherwise there is no sequence of moves going from one to the other. If they have the same linear invariant, then we bring both of them in normal form as in \Cref{prop:existence-normal-form-floating}, and then we can go from one to the other by \Cref{prop:uniqueness-normal-form-floating}.

    If $(\Gamma,\psi)$ has no floating pieces, but it has a family with at least three minimal points, or at least two families with at least two minimal points each. Then $(\Delta,\phi)$ must have the same property (as it has the same families and minimal points). We make both of them clean and we compute the multi-set of rigid edges and the assignment map: these must be the same, otherwise there is no sequence of moves going from one to the other. If they have the same multi-set of rigid edges and the same assignment map, then we bring both of them in normal form as in \Cref{prop:existence-normal-form-no-floating}, and then we can go from one to the other by \Cref{prop:uniqueness-normal-form-no-floating}.

    Finally, suppose that $(\Gamma,\psi)$ has no floating pieces, and all families of $(\Gamma,\psi)$ are individual, except at most one with at most two minimal points; we call $F_0$ such family. Note that $(\Delta,\phi)$ must have the same property (as it has the same families and minimal points). In this case, we make both $(\Gamma,\psi)$ and $(\Delta,\phi)$ clean and $F$-polished for each family $F\not=F_0$, as in \Cref{sec:projection-polished}. From now on, we only deal with sequences of moves involving GBS graphs which are $F$-polished for each $F\not=F_0$.

    For every family $F\not=F_0$, we take the unique minimal point $m_F$ in that family, and the unique edge $m_F\edge c_F$ with $c_F\in F_0$, and we observe that $c_F$ can be changed to any other point in the same conjugacy class, using slide moves only; in particular, in any moment, we can arrange them to be the same in $(\Gamma,\psi)$ as in $(\Delta,\phi)$. Let $(\Gamma',\psi'),(\Delta',\phi')$ be the GBS graphs obtained from $(\Gamma,\psi),(\Delta,\phi)$ respectively, by removing all edges with vertices outside $F_0$. Then there is a sequence of moves going from $(\Gamma,\psi)$ to $(\Delta,\phi)$ if and only if there is a sequence of moves going from $(\Gamma',\psi')$ to $(\Delta',\phi')$. But $(\Gamma',\psi'),(\Delta',\phi')$ are GBS graphs with one vertex and at most two edges each. The conclusion thus follows from \cite{ACK-iso3}.
\end{proof}

\begin{cor}\label{cor:isomorphism-oneqcc-fsgaps}
    There is an algorithm that, given two GBS graph $(\Gamma,\psi),(\Delta,\phi)$ with one qc-class and full-support gaps, decides whether the corresponding GBS groups are isomorphic or not. In case they are, the algorithm also computes a sequence of sign-changes, inductions, slides, swaps, and connections going from $(\Gamma,\psi)$ to $(\Delta,\phi)$.
\end{cor}
\begin{proof}
    By \Cref{thm:isomorphism-oneqcc-fsgaps}, and since vertex sign-changes can be guessed in finitely many ways, we only have to deal with induction.

    But if $(\Gamma,\psi)$ allows for induction, is totally reduced, and has one qc-class, then it must have just one vertex, and using a sequence of moves we can make it into a controlled GBS graph. In this case, the statement follows from \cite{ACK-iso1}.
\end{proof}

\subsection{Examples}

As usual, in the examples we use $\bA=\bbZ^{\cP(\Gamma,\psi)}$, omitting the $\bbZ/2\bbZ$ summand.

\begin{ex}\label{ex10}
    In \Cref{fig:example10} we can see two GBS graphs and the corresponding affine representations. Note that they have the same linear invariants (i.e. same conjugacy classes and same number of edges in each conjugacy class).
    
    In this case, the invariant distinguishing them is the Nielsen equivalence class of the pair of vectors. The affine representations of the two GBS graphs are
    $$
        \begin{cases}
            \bm_1\edge\bm_1+(2,1)\\
            \bm_2\edge\bm_2+(1,2)
        \end{cases}
        \qquad\text{and}\qquad
        \begin{cases}
            \bm_1\edge\bm_1+(1,2)\\
            \bm_2\edge\bm_2+(2,1)
        \end{cases}
    $$
    respectively, where $\bm_1=(1,0)$ and $\bm_2=(0,1)$, and we put the two vectors as columns in the two matrices
    $$
        \begin{pmatrix}
            2&1\\
            1&2
        \end{pmatrix}
        \qquad\text{and}\qquad
        \begin{pmatrix}
            1&2\\
            2&1
        \end{pmatrix}
    $$
    respectively. The determinants of the matrix tells us the Nielsen equivalence class of the pair of vectors, and is thus an isomorphism invariant. The first matrix has determinant $+3$, while the second has determinant $-3$. Therefore the two groups are not isomorphic.
\end{ex}

\begin{figure}[H]
\centering
\includegraphics[width=0.8\textwidth]{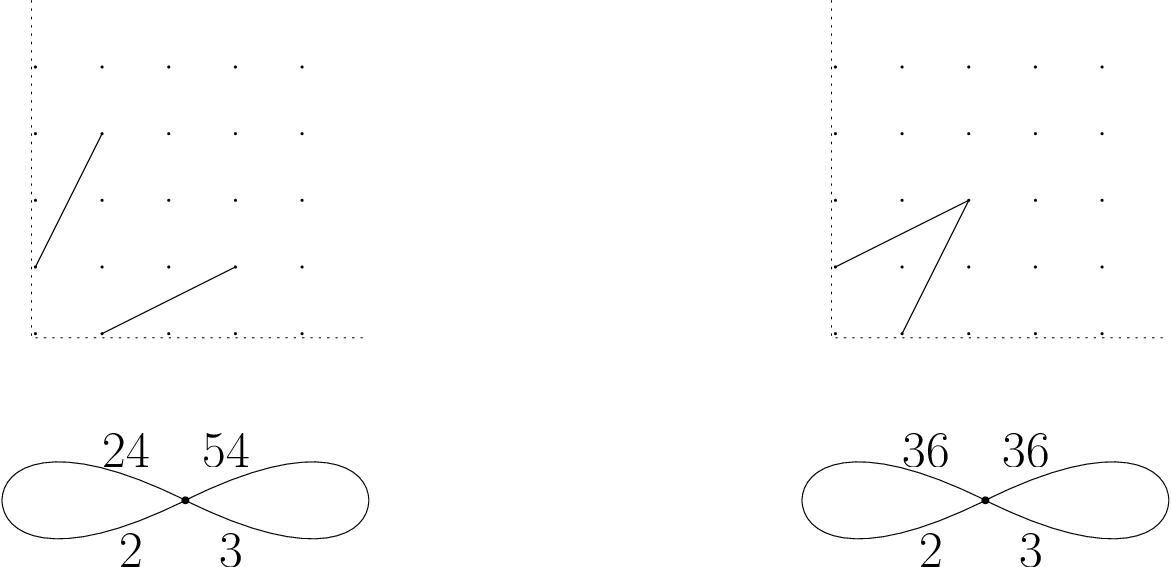}
\centering
\caption{Two non-isomorphic GBS graphs and the corresponding affine representations. The invariant distinguishing them is the Nielsen equivalence class of the couple of vectors: the determinant of the change of basis from one to the other is $-1$.}
\label{fig:example10}
\end{figure}

\begin{ex}\label{ex5}
    Consider the three GBS graphs $(\Gamma_1,\psi_1),(\Gamma_2,\psi_2),(\Gamma_3,\psi_3)$ as in Figure \ref{fig:example5-GBSgraphs}, with affine representations $\Lambda_1,\Lambda_2,\Lambda_3$ respectively. It is easy to see that the three GBS graphs have the same linear invariants.
    
    There is a quasi-conjugacy class $P$ satisfying: $\qcmin{P}=\{(3,0,0)\}$, $\qcsupp{P}=\{2\}$ and $\cla{P}=\gen{(7,0,0)}$; this quasi-conjugacy class contains one edge $(3,0,0)\edge(3,0,0)+\bv$, with $\bv=(7,0,0)$, that can not be changed along any sequence of slides, swaps, and connections (see \Cref{cor:only-one-edge}).

    There is a quasi-conjugacy class $Q$ satisfying: $\qcmin{Q}=\{(2,1,0),(1,2,0),(0,3,0)\}$, $\qcsupp{Q}=\{2,3,5\}$ and $\cla{Q}=\gen{(1,0,0),(0,1,0),(0,0,1)}$. The three minimal points are also the three minimal regions $M_1=\{(2,1,0)\}$ and $M_2=\{(1,2,0)\}$ and $M_3=\{(0,3,0)\}$. This quasi-conjugacy class contains three edges, $e_1$ given by $(2,1,0)\edge(2,1,0)+\bx_1$, and $e_2$ given by $(1,2,0)\edge(1,2,0)+\bx_2$, and $e_3$ given by $(0,3,0)\edge(0,3,0)+\bx_3$, for some $\bx_1,\bx_2,\bx_3\in\bA$ which are different in the three examples.

    The question is: which of these GBS graphs represent groups isomorphic to each other? We can use \Cref{Nielsen-equiv-3big} to rearrange the three elements $\bx_1,\bx_2,\bx_3$ as we want (and relative to $\bv$), provided that we check two conditions. The first condition, is that the components of $\bx_1,\bx_2,\bx_3$ must remain big enough. The second condition is that we must preserve the Nielsen equivalence class of the (ordered) triple $[\bx_1],[\bx_2],[\bx_3]$ in the abelian group
    \[
        \frac{\gen{\bx_1,\bx_2,\bx_3,\bv}}{\gen{\bv}}=\frac{\gen{(1,0,0),(0,1,0),(0,0,1)}}{\gen{(7,0,0)}}\cong\bbZ^2\oplus\bbZ/7\bbZ.
    \]
    By Proposition \ref{Nielsen-equiv}, we need to look at the determinant of the corresponding matrix modulo $7$. The three matrices for $(\Gamma_1,\psi_1),(\Gamma_2,\psi_2),(\Gamma_3,\psi_3)$ are given by
    \[
    \begin{pmatrix}
        3&2&2\\
        7&5&4\\
        6&4&5
    \end{pmatrix}
    \qquad\text{and}\qquad
    \begin{pmatrix}
        6&4&4\\
        7&5&4\\
        6&4&5
    \end{pmatrix}
    \qquad\text{and}\qquad
    \begin{pmatrix}
        9&6&6\\
        7&5&4\\
        6&4&5
    \end{pmatrix}
    \]
    respectively, and have determinants congruent to $1,2,3$ modulo $7$ respectively. So \Cref{Nielsen-equiv-3big} does not allow us to jump from one to the other.

    And in fact the three groups are pairwise non-isomorphic. Suppose that we start from one of $(\Gamma_1,\psi_1),(\Gamma_2,\psi_2),(\Gamma_3,\psi_3)$ and we perform a sequence of slides, swaps, and connections. By Lemma \ref{lem:escape-minimal-regions}, in each moment there must be at least one endpoint of some edge at $(2,1,0)$, one at $(1,2,0)$, and one at $(0,3,0)$. 
    We would like to deduce that the configuration can be written uniquely as
    \[
        \begin{cases}
            (2,1,0)\edge(2,1,0)+\bx_1\\
            (1,2,0)\edge(1,2,0)+\bx_2\\
            (0,3,0)\edge(0,3,0)+\bx_3
        \end{cases}
    \]
    for some $\bx_1,\bx_2,\bx_3\in\bA$. This is true, even though one has to be careful in to the degenerate cases where some edge connects two minimal regions. To be precise, there must be some edge with a vertex in $\{(2,1,0),(1,2,0),(0,3,0)\}$ and the other vertex outside; without loss of generality, assume it has an vertex in $(0,3,0)$. Then we call such edge $(0,3,0)\edge(0,3,0)+\bx_3$, uniquely determining $\bx_3$. Now there must be some edge with one vertex in $\{(2,1,0),(1,2,0)\}$ and the other vertex outside; without loss of generality, assume it has a vertex in $(1,2,0)$. Then we call such edge $(1,2,0)\edge(1,2,0)+\bx_2$, uniquely determining $\bx_2$. Finally, with a similar reasoning we uniquely determine $\bx_1$.

    Now we write the three vectors $\bx_1,\bx_2,\bx_3$ as columns of a $3\times 3$ matrix, and we compute the determinant modulo $7$. This is invariant by external slides (since the determinant is taken modulo $7$), by internal slides (which correspond to column operations on the matrix), by connections (which correspond to exchanging two columns, changing the sign of one column, and then performing a few column operation), and swaps can never be applied. It follows that the determinant modulo $7$ must be the same along any sequence of slides, swaps, and connections. Since in the three graphs $(\Gamma_1,\psi_1),(\Gamma_2,\psi_2),(\Gamma_3,\psi_3)$ the determinant has different values, there is no sequence of slides, swaps, and connections going from one of them to another. Finally, one can easily check that sign-changes can not help finding an isomorphism in this case. Induction does not apply to these examples.

    So these are three GBS graphs which exhibit a very similar behavior - same conjugacy classes, same number of edges in each conjugacy class - but are not isomorphic; and the isomorphism problem in this case is decided by a finer invariant, which is this determinant modulo $7$.
    
    In this specific case, if the determinant modulo $7$ would have been the same, then we would have immediately been able to provide an isomorphism using \Cref{Nielsen-equiv-3big}. In general, one also has to be careful about rigid vectors.
\end{ex}

\begin{figure}[H]
\centering
\includegraphics[width=\textwidth]{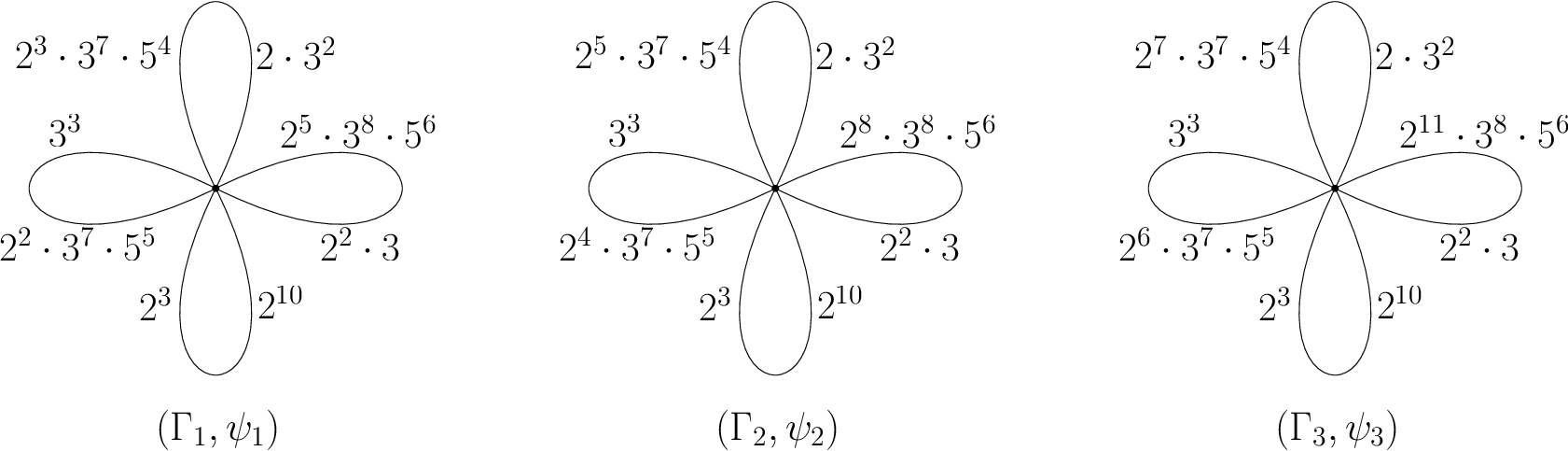}
\centering
\caption{Three pairwise non-isomorphic GBS groups, as in \Cref{ex5}.}
\label{fig:example5-GBSgraphs}
\end{figure}

\begin{ex}\label{ex8}
    Consider a GBS graph $(\Gamma,\psi)$ with affine representation $\Lambda$ as in \Cref{fig:example8-initial-config}. This means that $\Gamma$ has three vertices $v_1,v_2,v_3$ and set of primes $\cP(\Gamma,\psi)$ containing two prime numbers. The affine representation $\Lambda$ consists of three copies $\pA_{v_1},\pA_{v_2},\pA_{v_3}$ of $\bbN^2$, and of edges
    $$\begin{cases}
        (v_1,(11,0))\edge(v_1,(21,8))\\
        (v_1,(0,11))\edge(v_1,(14,9))\\
        (v_2,(10,1))\edge(v_3,(4,10))\\
        (v_2,(0,12))\edge(v_1,(15,15))\\
        (v_3,(2,6))\edge(v_2,(15,4))
    \end{cases}$$
    
    First, we compute the linear invariants of $(\Gamma,\psi)$. It is easy to check that all the edges belong to a common quasi-conjugacy class $Q$, with five minimal points $\bm_1=(v_1,(11,0))$ and $\bm_2=(v_1,(0,11))$ and $\bm_3=(v_2,(10,1))$ and $\bm_4=(v_2,(0,12))$ and $\bm_5=(v_3,(2,6))$. A computation shows that $\qcsupp{Q}=\cP(\Gamma,\psi)$ and $\cla{Q}=\gen{(1,0),(0,1)}$, and in particular $Q$ is also a conjugacy class. With this notation, $\Lambda$ can be rewritten as
    $$\begin{cases}
        \bm_1\edge\bm_1+(10,8)\\
        \bm_2\edge\bm_2+(14,9)\\
        \bm_3\edge\bm_5+(2,5)\\
        \bm_4\edge\bm_2+(15,4)\\
        \bm_5\edge\bm_3+(4,3)
    \end{cases}$$
    
    The graph of families $\cF(Q)$ coincides with $\Gamma$ in this case. We note that $(\Gamma,\psi)$ is already clean, and has no floating pieces, so we can assign an orientation to each edge, in such a way that edges are always oriented going out of minimal points. We can see that $(\Gamma,\psi)$ has a unique rigid cycle, given by the edges $\bm_5\edge\bm_3+(4,3)$ and $\bm_3\edge\bm_5+(2,5)$, with corresponding rigid vector $\br=(7,7)$.

    We now bring $(\Gamma,\psi)$ to normal form. First, we maximize the number of positive edges, by performing a slide and a connection, to get the configuration
    $$\begin{cases}
        \bm_1\edge\bm_1+(10,8)\\
        \bm_2\edge\bm_4+(13,14)\\
        \bm_3\edge\bm_3+(7,7)\\
        \bm_4\edge\bm_4+(14,9)\\
        \bm_5\edge\bm_3+(4,3)
    \end{cases}$$
    as in \Cref{fig:example8-not-normal-form}. This is already near to being a normal form, as all vectors are positive, except the two going from $\pA_{v_1},\pA_{v_3}$ to $\pA_{v_2}$. However, the two positive vectors at $\bm_1,\bm_4$, and the edge at $\bm_5$, have an endpoint which is not high enough. Thus we have to perform further manipulations to reach, for example, the configuration of \Cref{fig:example8-normal-form}, given by
    $$\begin{cases}
        \bm_1\edge\bm_1+(7,7)\\
        \bm_2\edge\bm_4+(13,14)\\
        \bm_3\edge\bm_3+(12,13)\\
        \bm_4\edge\bm_4+(17,15)\\
        \bm_5\edge\bm_4+(20,5)
    \end{cases}$$
    where the rigid vector has been moved at $\bm_1$, and all the other endpoints are sufficiently high above the minimal points.

    For the normal form of \Cref{fig:example8-normal-form}, the linear invariant $(-1)^\epsilon\Theta:\cM\rightarrow\{\br\}\sqcup\frac{\cla{Q}}{\gen{\br}}\sqcup\{v_1,v_3\}$ is given by $\Theta(\bm_1)=\br$ and $\Theta(\bm_2)=v_1$ and $\Theta(\bm_3)=(12,13)$ and $\Theta(\bm_4)=(17,15)$ and $\Theta(\bm_5)=v_3$. This means that the minimal points $\bm_2,\bm_5$ are hosting the edges connecting $\pA_{v_1},\pA_{v_3}$ to $\pA_{v_2}$ respectively, the minimal point $\bm_1$ is hosting the rigid vector $\br$, and the minimal points $\bm_3,\bm_4$ are hosting the vectors $(12,13),(17,15)$ of the abelian group $\frac{\cla{Q}}{\gen{\br}}\cong\bbZ\oplus\bbZ/7\bbZ$. Note that the couple $(12,13),(17,15)$ is Nielsen equivalent to the couple $(1,0),(0,1)$ in the abelian group $\frac{\cla{Q}}{\gen{\br}}$.
    
    If in \Cref{fig:example8-normal-form} we substitute the edge $\bm_4\edge\bm_4+(17,15)$ with $\bm_4\edge\bm_4+(17,14)$ while leaving the others unchanged, then we obtain a non-isomorphic GBS group. In fact, $(12,13),(17,14)$ is Nielsen equivalent to $(3,0),(0,1)$, and the matrix of change of basis from $(1,0),(0,1)$ to $(3,0),(0,1)$ has determinant $\equiv 3 \not\equiv 1$ modulo $7$. In fact, in $\bbZ\oplus\bbZ/7\bbZ$ there are exactly $6$ Nielsen equivalence classes of couples of generators, two couples being Nielsen equivalent if and only if the matrix of change of basis has determinant $\equiv1$ mod $7$ (see \Cref{Nielsen-equiv}). Fixed the linear invariants and the multiset of rigid vectors $\bR=\{\br\}$, there are exactly $6$ isomorphism classes of GBS groups, corresponding to the $6$ possible values of this determinant modulo $7$.
\end{ex}

\begin{figure}[ht!]
\centering
\includegraphics[width=\textwidth]{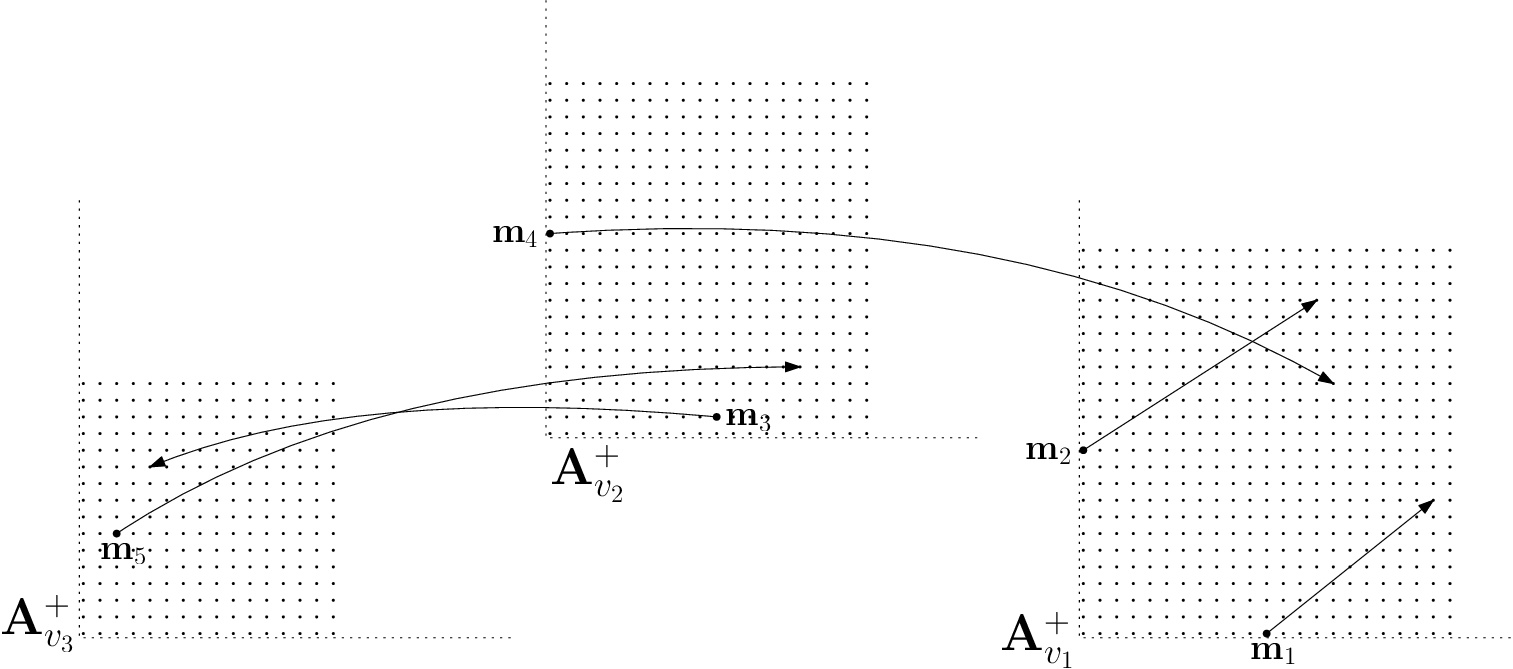}
\centering
\caption{The affine representation for the GBS graph of \Cref{ex8}. The edges have been assigned orientations going out of the minimal points.}
\label{fig:example8-initial-config}
\end{figure}

\begin{figure}[H]
\centering
\includegraphics[width=\textwidth]{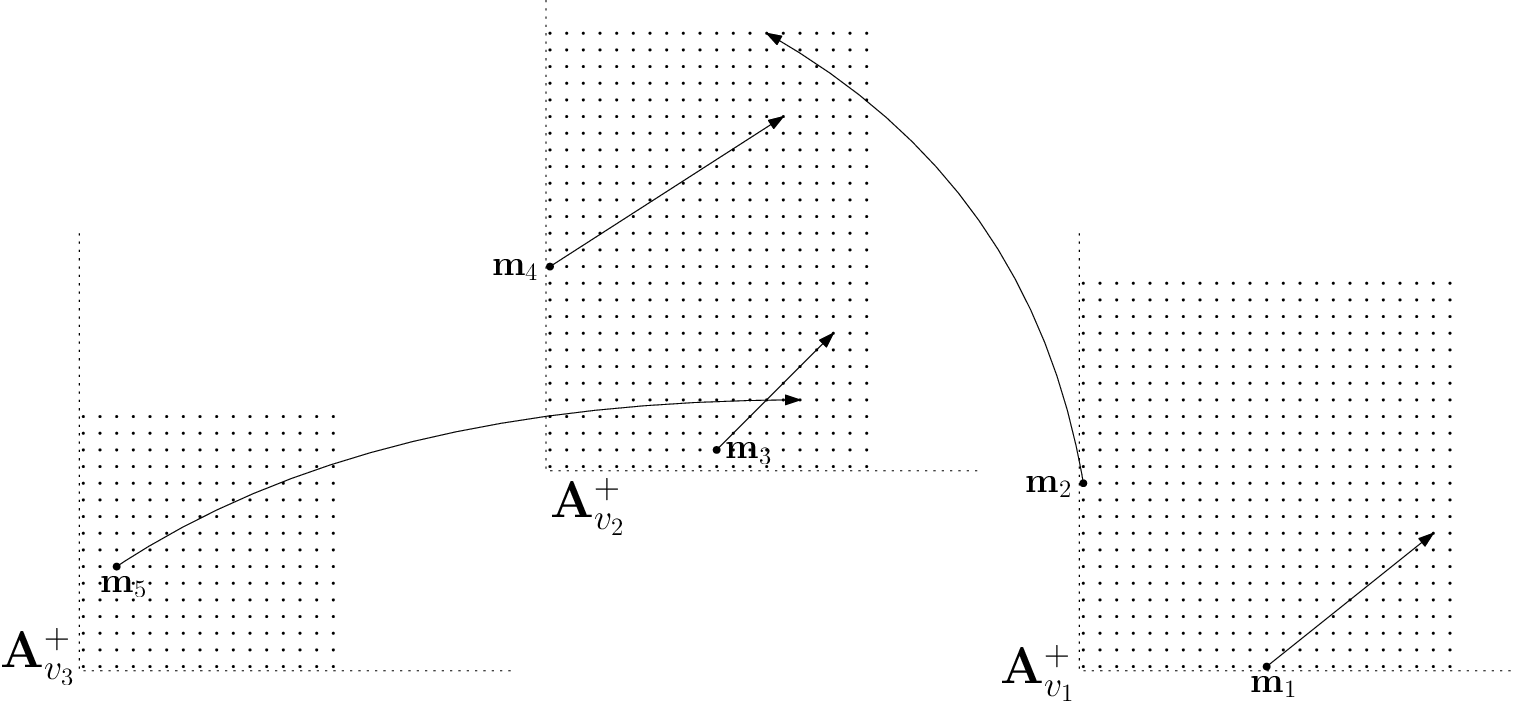}
\centering
\caption{A first manipulation maximizes the number of positive vectors, leaving only two edges (at $\bm_5,\bm_2$) going from a vertex to another.}
\label{fig:example8-not-normal-form}
\end{figure}

\begin{figure}[H]
\centering
\includegraphics[width=\textwidth]{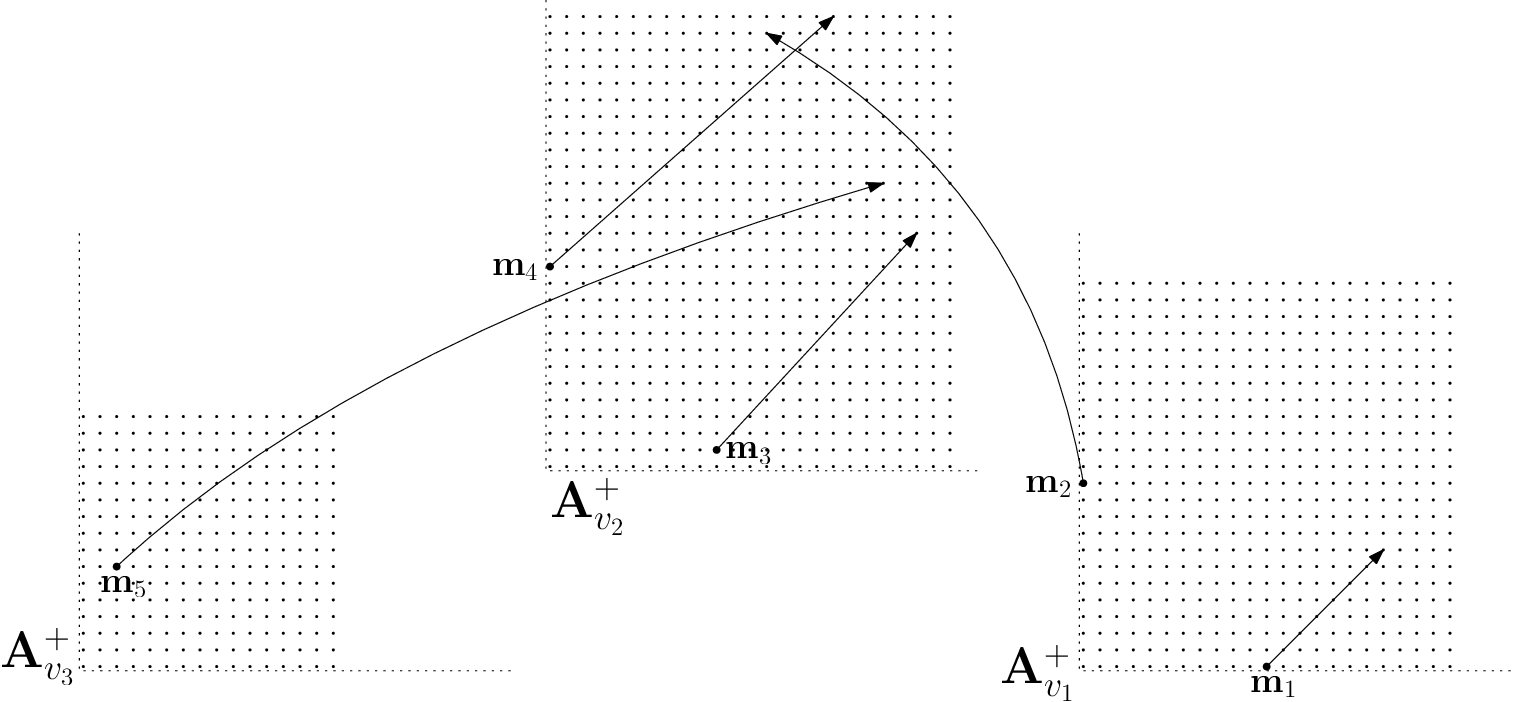}
\centering
\caption{A normal form for the GBS graph of \Cref{ex8}.}
\label{fig:example8-normal-form}
\end{figure}

\begin{ex}\label{ex9}
    Consider a GBS graph $(\Gamma,\psi)$ with affine representation $\Lambda$ as in \Cref{fig:example9}. This means that $\Gamma$ has one vertex and set of primes $\cP(\Gamma,\psi)$ with two prime numbers. The affine representation $\Lambda$ consists of a single copy $\pA\cong\bbN^2$ containing four edges
    $$\begin{cases}
        \bm_1\edge\bm_1+(2,3)\\
        \bm_2\edge\bm_2+(0,2)\\
        \bm_3\edge\bm_3+(6,2)\\
        \bm_4\edge\bm_4+(4,3)
    \end{cases}$$
    where $\bm_1=(6,0)$ and $\bm_2=(3,2)$ and $\bm_3=(2,5)$ and $\bm_4=(0,6)$. These four edges all belong to a common quasi-conjugacy class $Q$, characterized by $\qcmin{Q}=\{\bm_1,\bm_2,\bm_3,\bm_4\}$ and $\qcsupp{Q}=\cP(\Gamma,\psi)$ and $\cla{Q}=\gen{(2,0),(0,1)}$. The conjugacy class of $\bm_1,\bm_3,\bm_4$ contains three edges while the conjugacy class of $\bm_2$ contains one edge.

    This GBS graph has one qc-class and no floating pieces. Note however that it does not have full-support gaps. We can still look for rigid cycles as in \Cref{def:rigid-vector}. The only reasonable candidate to be a rigid vector would be $\br=(0,2)$, with rigid cycle $\bm_2\edge\bm_2+\br$. However, we can find minimal points $m\not=m'$ such that $m+\bw\ge m'$, as for example we have that $\bm_1+\br\ge\bm_2$ and $\bm_3+\br\ge\bm_4$. So the edge $\bm_2\edge\bm_2+\br$ does not satisfy the conditions of \Cref{def:rigid-vector}, and is thus not a rigid cycle.
    
    We now try to bring the GBS in normal form. It is easy to see that the other three edges can be made very long using slide moves. It remains however unclear how to change the edge $\bm_2\edge\bm_2+\br$. In fact, swap moves do not apply here, and changing it by slide or by connection would require some other vertex of some other edge to be of the form $\bm_2+(a,0)$ for some $a\in\bbN$. But this would imply that the conjugacy class of $\bm_2$ would contain more that one edge, and this is impossible, since the conjugacy class of $\bm_2$ contains only one edge. This means that no slide, swap, connection can ever change the edge $\bm_2\edge\bm_2+\br$.

    This shows that the hypothesis of being full-support gaps is really necessary for \Cref{prop:existence-normal-form-no-floating}, as otherwise one might encounter obstructions when trying to bring the GBS graph to normal form.
\end{ex}

\begin{figure}[H]
\centering
\includegraphics[width=0.5\textwidth]{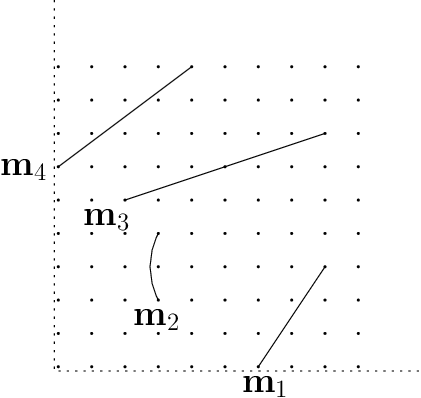}
\centering
\caption{The affine representation for the GBS graph of \Cref{ex9}.}
\label{fig:example9}
\end{figure}

\bibliographystyle{alpha}

\footnotesize
\begin{thebibliography}{ACRK25b}
	
	\bibitem[ACRK25a]{ACK-iso1}
	D.~Ascari, M.~Casals-Ruiz, and I.~Kazachkov.
	\newblock {On the isomorphism problem for cyclic JSJ decompositions: vertex elimination}, 2025.
	\newblock preprint.
	
	\bibitem[ACRK25b]{ACK-iso3}
	D.~Ascari, M.~Casals-Ruiz, and I.~Kazachkov.
	\newblock {On the isomorphism problem for generalized Baumslag-Solitar groups: angles}, 2025.
	\newblock preprint.
	
	\bibitem[Bee15]{B15}
	B.~Beeker.
	\newblock Multiple conjugacy problem in graphs of free abelian groups.
	\newblock {\em Groups Geom. Dyn.}, 9(1):1--27, 2015.
	
	\bibitem[CF08]{CF08}
	M.~Clay and M.~Forester.
	\newblock On the isomorphism problem for generalized {B}aumslag-{S}olitar groups.
	\newblock {\em Algebr. Geom. Topol.}, 8(4):2289--2322, 2008.
	
	\bibitem[DG11]{DG11}
	F.~Dahmani and V.~Guirardel.
	\newblock The isomorphism problem for all hyperbolic groups.
	\newblock {\em Geom. Funct. Anal.}, 21(2):223--300, 2011.
	
	\bibitem[DS99]{DS99}
	M.~J. Dunwoody and M.~E. Sageev.
	\newblock {JSJ}-splittings for finitely presented groups over slender groups.
	\newblock {\em Invent. Math.}, 135(1):25--44, 1999.
	
	\bibitem[DT19]{DT19}
	F.~Dahmani and N.~Touikan.
	\newblock Deciding isomorphy using {D}ehn fillings, the splitting case.
	\newblock {\em Invent. Math.}, 215(1):81--169, 2019.
	
	\bibitem[Dud17]{Dud17}
	F.~A. Dudkin.
	\newblock The isomorphism problem for generalized {B}aumslag-{S}olitar groups with one mobile edge.
	\newblock {\em Algebra Logika}, 56(3):300--316, 2017.
	
	\bibitem[FM98]{FM98}
	B.~Farb and L.~Mosher.
	\newblock {A rigidity theorem for the solvable {B}aumslag-{S}olitar groups}.
	\newblock {\em Inventiones mathematicae}, 131(2):419--451, 1998.
	
	\bibitem[For02]{For02}
	M.~Forester.
	\newblock Deformation and rigidity of simplicial group actions on trees.
	\newblock {\em Geom. Topol.}, 6(1):219--267, 2002.
	
	\bibitem[For06]{For06}
	M.~Forester.
	\newblock Splittings of generalized {B}aumslag-{S}olitar groups.
	\newblock {\em Geom. Dedicata}, 121:43--59, 2006.
	
	\bibitem[FP06]{FP06}
	K.~Fujiwara and P.~Papasoglu.
	\newblock {JSJ}-decompositions of finitely presented groups and complexes of groups.
	\newblock {\em Geom. Funct. Anal.}, 16(1):70--125, 2006.
	
	\bibitem[GL11]{GL11}
	V.~Guirardel and G.~Levitt.
	\newblock Trees of cylinders and canonical splittings.
	\newblock {\em Geom. Topol.}, 15(2):977--1012, 2011.
	
	\bibitem[GL17]{GL17}
	V.~Guirardel and G.~Levitt.
	\newblock {JSJ} decompositions of groups.
	\newblock {\em Ast\'erisque}, 395:vii+165, 2017.
	
	\bibitem[Kro90]{Kro90}
	P.~H. Kropholler.
	\newblock Baumslag-solitar groups and some other groups of cohomological dimension two.
	\newblock {\em Commentarii mathematici Helvetici}, 65(4):547--558, 1990.
	
	\bibitem[Loc92]{L92}
	J.~M. Lockhart.
	\newblock The conjugacy problem for graph products with infinite cyclic edge groups.
	\newblock {\em Proc. Amer. Math. Soc.}, 114(3):603--606, 1992.
	
	\bibitem[RS97]{RS97}
	E.~Rips and Z.~Sela.
	\newblock Cyclic splittings of finitely presented groups and the canonical {JSJ} decomposition.
	\newblock {\em Ann. of Math. (2)}, 146(1):53--109, 1997.
	
	\bibitem[Sel95]{Sel95}
	Z.~Sela.
	\newblock The isomorphism problem for hyperbolic groups. {I}.
	\newblock {\em Ann. of Math. (2)}, 141(2):217--283, 1995.
	
	\bibitem[Ser77]{Ser77}
	J.~P. Serre.
	\newblock {\em Arbres, amalgames, {${\rm SL}\sb{2}$}}, volume No. 46 of {\em Ast\'erisque}.
	\newblock Soci\'et\'e{} Math\'ematique de France, Paris, 1977.
	\newblock Avec un sommaire anglais, R\'edig\'e{} avec la collaboration de Hyman Bass.
	
	\bibitem[Wan25]{Wan25}
	D.~Wang.
	\newblock The isomorphism problem for small-rose generalized {B}aumslag-{S}olitar groups.
	\newblock {\em J. Algebra}, 661:193--217, 2025.
	
	\bibitem[Wei16]{W16}
	A.~Wei\ss.
	\newblock A logspace solution to the word and conjugacy problem of generalized {B}aumslag-{S}olitar groups.
	\newblock In {\em Algebra and computer science}, volume 677 of {\em Contemp. Math.}, pages 185--212. Amer. Math. Soc., Providence, RI, 2016.
	
	\bibitem[Why01]{Why01}
	K.~Whyte.
	\newblock {The large scale geometry of the higher {B}aumslag-{S}olitar groups}.
	\newblock {\em Geometric \& Functional Analysis}, 11(6):1327--1343, 2001.
	
\end{thebibliography}
\normalsize

\appendix

\section{Nielsen equivalence in finitely generated abelian groups}\label{sec:Nielsen-equiv-abelian}

We consider elements of $\bbZ^n$ as column vectors. We denote with $\cM_{n\times k}(\bbZ)$ the set of matrices with $n$ rows and $k$ columns and integer coefficients. For $M\in\cM_{n\times k}(\bbZ)$ we denote with $M_i^j$ the integer number in the $i$-th row and $j$-th column of $M$, for $i=1,\dots,n$ and $j=1,\dots,k$; we denote with $\gencol{M}$ the subgroup of $\bbZ^n$ generated by the columns of $M$. A \textbf{column operation} on a matrix consists of choosing a column and adding or subtracting it from another column. A \textbf{(column) reordering operation} consists of interchanging two columns of the matrix. Notice that column and column reordering operations do not change $\gencol{M}$. A \textbf{row operation} consists of choosing a row and adding or subtracting it from another row. Row operations correspond to changing basis for the free abelian group $\bbZ^n$.

\subsection*{Smith normal form}

\begin{defn}\label{def:Smith}
A matrix $S\in\cM_{n\times k}(\bbZ)$ is in \textbf{Smith normal form} if it satisfies the following conditions:
\begin{enumerate}
\item $S_i^j=0$ whenever $i\not=j$.
\item $S_i^i\ge0$ for $i=1,\dots,m$ where $m=\min\{n,k\}$.
\item We have $S_{i+1}^{i+1}\divides S_i^i$ for all $i=1,\dots,m-1$.
\end{enumerate}
\end{defn}

Every matrix is equivalent, up to row and column operations and reordering operations, to a unique matrix in Smith normal form.
An interesting feature is that the integers $S_i^i$ can be computed directly from the initial matrix, as follows. For $M\in\cM_{n\times k}(\bbZ)$, define $D_\ell(M)\ge0$ as the greatest common divisor of all the determinants of the $\ell\times\ell$ minors of $M$, for $1\le\ell\le m$ where $m=\min\{n,k\}$. Notice that, if $M$ has rank $r$, then $D_\ell(M)=0$ if and only if $\ell\ge r+1$.

\begin{lem}\label{computing-Smith}
Let $M\in\cM_{n\times k}(\bbZ)$ be a matrix of rank $r$. Let $S$ be the unique matrix in Smith normal form obtained from $M$ by row and column operations. Then we have the following:
\begin{enumerate}
\item $D_1(M)\divides D_2(M)\divides \dots \divides D_m(M)$, where $m=\min\{n,k\}$.
\item $S_i^i=D_{m+1-i}(M)=0$ for $i=1,\dots,m-r$.
\item $S_i^i=D_{m+1-i}(M)/D_{m-i}(M)$ for $i=m-r+1,\dots,m-1$.
\item $S_m^m=D_1(M)$.
\end{enumerate}
\end{lem}
\begin{proof}
It is immediate to notice that $D_\ell(M)$ is invariant under row and column operations, as well as column swap operation. But for a matrix in Smith normal the statement holds, so it holds for all matrices.
\end{proof}

\subsection*{Finitely generated abelian groups}

Every finitely generated abelian group is isomorphic to
$$\bbZ/d_1\bbZ\oplus\ldots\oplus\bbZ/d_n\bbZ$$
for unique integers numbers $n,d_1,\dots,d_n\ge0$ satisfying $1\not=d_n\divides d_{n-1}\divides\dots\divides d_1$. Notice that some $d_i$ can be equal to $0$, producing $\bbZ$ summands.

\begin{lem}\label{Smith-to-abelian}
Let $M\in\cM_{n\times k}(\bbZ)$. Let $S\in\cM_{n\times k}(\bbZ)$ be the Smith normal form of $M$. Then $\bbZ^n/\gencol{M}$ is isomorphic to $\bbZ/d_1\bbZ\oplus\ldots\oplus\bbZ/d_{n'}\bbZ$ where
\begin{enumerate}
\item $n'\le\min\{n,k\}$ is the maximum integer such that $S_{n'}^{n'}\not=1$.
\item $d_i=S_i^i$ for $i=1,\dots,n'$.
\end{enumerate}
\end{lem}
\begin{proof}
Column and reordering operations do not change $\gencol{M}$, and in particular they do not change the quotient $\bbZ^n/\gencol{M}$. Row operations correspond to changing basis for $\bbZ^n$, so they do not change the isomorphism type of $\bZ^n/\gencol{M}$. Thus we have that $\bbZ^n/\gencol{M}$ is isomorphic to $\bbZ^n/\gencol{S}$. The conclusion follows.
\end{proof}

\subsection*{Nielsen equivalence in finitely generated abelian groups}

Let $A$ be a finitely generated abelian group. Let $(w_1,\dots,w_k)$ be an ordered $k$-tuple of elements of $A$: a \textbf{Nielsen move} on $(w_1,\dots,w_k)$ is any operation that consists of the substitution $w_i$ with $w_i+w_j$ or with $w_i-w_j$, for some $j\not=i$. Two ordered $k$-tuples are \textbf{Nielsen equivalent} if there is a sequence of Nielsen moves going from one to the other. Our aim in this section is to classify the $k$-tuples of generators for $A$ up to Nielsen equivalence (we ignore $k$-tuples that do not generate the whole group $A$, as they would require us to change the ambient group we are working in).

Fix an isomorphism
$$A=\bbZ/d_1\bbZ\oplus\bbZ/d_2\bbZ\oplus\ldots\oplus\bbZ/d_n\bbZ$$
for integers $n,d_1,\dots,d_n\ge0$ with $1\not=d_n\divides d_{n-1}\divides\dots\divides d_1$. Then we can represent elements of $A$ as column vectors with $n$ entries, where the $i$-th entry is an element of $\bbZ/d_i\bbZ$ for $i=1,\dots,n$. We represent an ordered $k$-tuple $(w_1,\dots,w_k)$ as a matrix $M=M(w_1,\dots,w_k)$ with $n$ rows and $k$ columns. Nielsen moves correspond to column operations on the matrix $M$.

\begin{lem}\label{min-generators}
The group $A=\bbZ/d_1\bbZ\oplus\ldots\oplus\bbZ/d_n\bbZ$ can not be generated by less than $n$ elements.
\end{lem}
\begin{proof}
Take a prime $p\divides d_n$. Consider the surjective projection map from $\bbZ/d_1\bbZ\oplus\ldots\oplus\bbZ/d_n\bbZ$ to $(\bbZ/p\bbZ)^n$. But $(\bbZ/p\bbZ)^n$ can not be generated by less than $n$ elements.
\end{proof}

\begin{prop}\label{Nielsen-equiv}
For every $k\ge n+1$, every two ordered $k$-tuples of generators for $A$ are Nielsen equivalent. Two ordered $n$-tuples of generators for $A$ are Nielsen equivalent if and only if the corresponding matrices have the same determinant modulo $d_n$.
\end{prop}
\begin{proof}
Let $(w_1,\dots,w_k)$ and $(u_1,\dots,u_k)$ be $k$-tuples of generators for $A$ and let $M,N\in\cM_{n\times k}(\bbZ)$ be the matrices whose columns represent $w_1,\dots,w_k$ and $u_1,\dots,u_k$ respectively. If $k=n$ and $w_1,\dots,w_k$ and $u_1,\dots,u_k$ are Nielsen equivalent, then there must be a sequence of columns operations going from $M$ to $N$, and in particular the two matrices must have the same determinant.

We look at the first row of $M$, and using column operations we perform the Euclidean algorithm; we obtain a new matrix $M_1^1\not=0$ and $M_1^j=0$ for $j=2,\dots,k$. But since $w_1,\dots,w_k$ generate the whole group, we must have that $M_1^1$ is coprime with $d_1$. We add $d_1$ to $M_1^2$ (we can, since elements of the first row are elements of $\bbZ/d_1\bbZ$), and then we perform the Euclidean algorithm again, obtaining a new matrix with $M_1^1=1$ and $M_1^j=0$ for $j=2,\dots,k$.

We reiterate the same reasoning by induction. If $k\ge n+1$ we obtain a matrix with $M_i^i=1$ for $i=1,\dots,n$ and $M_i^j=0$ for $j\not=i$. Performing the same procedure on $N$ we obtain the same matrix, and in particular there is a sequence of column operations going from $M$ to $N$, and thus $w_1,\dots,w_k$ is Nielsen equivalent to $u_1,\dots,u_k$. If $k=n$, we obtain a matrix with $M_i^i=1$ for $i=1,\dots,n-1$ and $M_n^n$ invertible modulo $d_n$ and $M_i^j=0$ for $j\not=n$. Notice that the determinant of $M$ modulo $d_n$ is an invariant under column operations, and for the matrix reduced in this form it is equal to $M_n^n$. If the determinant of $M$ was equal to the one of $N$, then performing the same procedure on $N$ we obtain the same matrix, and in particular there is a sequence of column operations going from $M$ to $N$, and thus $(w_1,\dots,w_k)$ is Nielsen equivalent to $(u_1,\dots,u_k)$.
\end{proof}

We will also need the following two lemmas.

\begin{lem}\label{Nielsen-equiv-rel1}
Let $a\in A$ be an element represented by a vector $(a_1,\dots,a_n)$, and suppose that $\mcd{a_1,\dots,a_{n-1},a_n}=1$. Let $(a,w_2,\dots,w_k)$, $(a,u_2,\dots,u_k)$ be Nielsen equivalent $k$-tuples of generators for $A$ with $k\ge n+1$. Then $([w_2],\dots,[w_k])$, $([u_2],\dots,[u_k])$ are Nielsen equivalent $(k-1)$-tuples of generators for $A/\gen{a}$.
\end{lem}
\begin{proof}
Let $R\in\cM_{n\times n}(\bbZ)$ be the matrix given by $R_i^i=d_i$ for $i=1,\dots,n$ and $R_i^j=0$ for $j\not=i$. We have that $A=\bbZ^n/\gencol{R}$. Let $R'$ be the matrix obtained from $R$ by adding an extra column, equal to $(a_1,\dots,a_n)$. Then $A/\gen{a}=\bbZ^n/\gencol{R'}$. Since $\mcd{a_1,\dots,a_{n-1}}=1$, we see that $D_1(R')=1$ and $D_2(R')=d_n\not=1$. In particular, by Lemmas \ref{computing-Smith}, \ref{Smith-to-abelian} and \ref{min-generators}, we have that the minimum number of generators for $A/\gen{a}$ is $n-1$. But then by Proposition \ref{Nielsen-equiv}, the $(k-1)$-tuples of generators $([w_2],\dots,[w_k])$, $([u_2],\dots,[u_k])$ for $A/\gen{a}$ are Nielsen equivalent, since $k-1\ge n$.
\end{proof}

\begin{lem}\label{Nielsen-equiv-rel2}
Let $a\in A$ be an element represented by a vector $(a_1,\dots,a_n)$, and suppose that $\mcd{a_1,\dots,a_{n-1}}=1$. Let $(a,w_2,\dots,w_n)$, $(a,u_2,\dots,u_n)$ be Nielsen equivalent $n$-tuples of generators for $A$. Then $([w_2],\dots,[w_n])$, $([u_2],\dots,[u_n])$ are Nielsen equivalent $(n-1)$-tuples of generators for $A/\gen{a}$.
\end{lem}
\begin{proof}
Let $R\in\cM_{n\times n}(\bbZ)$ be the matrix given by $R_i^i=d_i$ for $i=1,\dots,n$ and $R_i^j=0$ for $j\not=i$. We have that $A=\bbZ^n/\gencol{R}$. Let also $M,N\in\cM_{n\times k}(\bbZ)$ be matrices whose columns represent $a,w_2,\dots,w_n$ and $a,u_2,\dots,u_n$ respectively. By Proposition \ref{Nielsen-equiv}, $M$ and $N$ must have the same determinant modulo $d_n$.

Let $R'$ be the matrix obtained from $R$ by adding an extra column, equal to the first column of $M$ and of $N$. We have that $A/\gen{a}=\bbZ^n/\gencol{R'}$ and $D_1(R')=1$ and $D_2(R')=d_n$ (since $\mcd{a_1,\dots,a_{n-1}}=1$). Thus, by Proposition \ref{Nielsen-equiv}, we only need to check that the matrices representing $([w_2],\dots,[w_n])$ and $([u_2],\dots,[u_n])$ have the same determinant modulo $d_n$.

We perform row operations at the same time on $R,M,N$, running the Euclidean algorithm on the first column of $M$ (which is also the first column on $N$). We obtain matrices $\ol{R},\ol{M},\ol{N}$ such that the first column of $\ol{M}$ (which is also the first column of $\ol{N}$) is equal to $(1,0,\dots,0,0)$. Notice that the determinant of $\ol{M}$ is the same as the one of $\ol{N}$ modulo $d_n$, since row operations do not change the determinant. Since row operations correspond to changing basis of $\bbZ^n$, we have an isomorphism $A\cong\bbZ^n/\gencol{\ol{R}}$, with the two matrices $\ol{M}$ and $\ol{N}$ representing the two $k$-tuples $a,w_2,\dots,w_n$ and $a,u_2,\dots,u_n$. Let $\ol{M}''$ be the minor of $\ol{M}$ obtained by canceling the first row and the first column of $\ol{M}$, and define $\ol{N}''$ analogously; let also $\ol{R}''$ be the matrix obtained from $\ol{R}$ by removing the first row. The isomorphism $A\cong\bbZ^n/\gencol{\ol{R}}$ induces an isomorphism $A/\gen{a}\cong\bbZ^{n-1}/\gencol{\ol{R}''}$ and the matrices $\ol{M}''$ and $\ol{N}''$ represent the $(n-1)$-tuples $([w_2],\dots,[w_n])$ and $([u_2],\dots,[u_n])$. But since the first column of $\ol{M}$ and of $\ol{N}$ is $(1,0,\dots,0)$ and $\ol{M},\ol{N}$ have the same determinant modulo $d_n$, it follows that $\ol{M}''$ and $\ol{N}''$ have the same determinant modulo $d_n$. The conclusion follows.
\end{proof}

\section{Nielsen equivalence for big vectors}\label{sec:Nielsen-equiv-big}

Fix a free abelian group $\bbZ^n$. Let $(w_1,\dots,w_k)$ be an ordered $k$-tuple of vectors of $\bbZ^n$. For an integer $M\ge1$ we say that an ordered $k$-tuple $(w_1,\dots,w_k)$ is \textbf{$M$-big} if each component of each of $w_1,\dots,w_k$ is $\ge M$. Our aim is to classify the $M$-big ordered $k$-tuples of vectors up to Nielsen moves. This means that, given two $M$-big ordered $k$-tuples, we want to know whether there is a sequence of Nielsen moves that goes from one to the other, and such that all the $k$-tuples that we write along the sequence are $M$-big.

We are going to need also a relative version of the same problem, that we now introduce. Let $v_1,\dots,v_h$ be vectors of $\bbZ^n$ and let $(w_1,\dots,w_k)$ be an ordered $k$-tuple of vectors of $\bbZ^n$. A \textbf{Nielsen move relative to $v_1,\dots,v_h$}, also called \textbf{relative Nielsen move} when there is no ambiguity, is any operation that consists of substituting $w_i$ with $w_i+v_j$ or $w_i-v_j$ for some $i\in\{1,\dots,k\}$ and $j\in\{1,\dots,h\}$. Fixed vectors $v_1,\dots,v_h$, we want to classify the $M$-big ordered $k$-tuples up to Nielsen moves and relative Nielsen moves. This means that, given two $M$-big ordered $k$-tuples, we want to know whether there is a sequence of Nielsen moves and relative Nielsen moves that goes from one to the other, and such that all the $k$-tuples that we write along the sequence are $M$-big.

Given vectors $v_1,\dots,v_h\in\bbZ^n$ and an ordered $k$-tuple $(w_1,\dots,w_k)$, we define the finitely generated abelian group $$A(w_1,\dots,w_k;v_1,\dots,v_h)=\gen{v_1,\dots,v_h,w_1,\dots,w_k}/\gen{v_1,\dots,v_h}.$$
It is immediate to see that $A$ is invariant if we change the $k$-tuple by Nielsen moves and relative Nielsen moves. Moreover, we can look at the ordered $k$-tuple of generators $([w_1],\dots,[w_k])$ for the abelian group $A$, and we notice that Nielsen moves on $(w_1,\dots,w_k)$ correspond to Nielsen moves on $([w_1],\dots,[w_k])$, while relative Nielsen moves on $(w_1,\dots,w_k)$ correspond to the identity on $([w_1],\dots,[w_k])$.

The problem of classification of $M$-big ordered $k$-tuples of vectors up to Nielsen moves and relative Nielsen moves, and the problem of classification of ordered $k$-tuples of generators for $A$ up to Nielsen equivalence, are strictly related. In fact, we will prove that they are equivalent for $k\ge 3$, regardless of $M$. The main result of Appendix \ref{sec:Nielsen-equiv-big} is the following theorem:

\begin{thm}\label{Nielsen-equiv-3big}
Let $v_1,\dots,v_h\in\bbZ^n$ and let $M\ge1$. Let $k\ge 3$ and let $(w_1,\dots,w_k),(w_1',\dots,w_k')$ be $M$-big ordered $k$-tuples. Then the following are equivalent:
\begin{enumerate}
\item\label{neb1} There is a sequence of Nielsen moves and Nielsen moves relative to $v_1,\dots,v_h$ going from $(w_1,\dots,w_k)$ to $(w_1',\dots,w_k')$ such that all the $k$-tuples that we write along the sequence are $M$-big.
\item\label{neb2} We have that $A(w_1,\dots,w_k;v_1,\dots,v_h)=A(w_1',\dots,w_k';v_1,\dots,v_h)=A$ and the two ordered $k$-tuples of generators $([w_1],\dots,[w_k]),([w_1'],\dots,[w_k'])$ are Nielsen equivalent in $A$.
\end{enumerate}
\end{thm}

\subsection*{Preliminary results}

We provide a couple of preliminary lemmas, that will be useful in what follows. Fix $n\ge1$ and vectors $v_1,\dots,v_h$ in $\bbZ^n$; fix also an integer $M\ge1$.

\begin{lem}\label{can-reorder}
Let $k\ge1$ and let $(w_1,\dots,w_k)$ be an $M$-big ordered $k$-tuple. Then for every even permutation $\sigma$ of $\{1,\dots,k\}$ there is a sequence of Nielsen moves going from $(w_1,\dots,w_k)$ to $(w_{\sigma(1)},\dots,w_{\sigma(k)})$ such that all the $k$-tuples that we write along the sequence are $M$-big.
\end{lem}
\begin{proof}
Consider the sequence of Nielsen moves
\begin{gather}\notag
\begin{split}
(w_1,w_2,w_3)\rar & (w_1,w_2+w_3,w_3)\rar (w_1,w_2+w_3,w_3+w_1)\rar \\
\rar & (w_1+w_2+w_3,w_2+w_3,w_3+w_1)\rar (w_2,w_2+w_3,w_3+w_1)\rar \\ \rar & (w_2,w_3,w_3+w_1)\rar (w_2,w_3,w_1)
\end{split}
\end{gather}
and notice that in the same way we can obtain any permutation which is a $3$-cycle. But $3$-cycles generate all even permutations; the conclusion follows.
\end{proof}

\begin{lem}\label{can-slide}
Let $k\ge2$ and let $(w_1,w_2,\dots,w_k),(w_1',w_2,\dots,w_k)$ be $M$-big ordered couples. Suppose that $w_1'=w_1+\lambda_2w_2+\dots+\lambda_kw_k+\mu_1v_1+\dots+\mu_hv_h$ for some $\lambda_2,\dots,\lambda_k,\mu_1,\dots,\mu_h\in\bbZ$. Then there is a sequence of Nielsen moves and relative Nielsen moves going from $(w_1,w_2,\dots,w_k)$ to $(w_1',w_2,\dots,w_k)$ such that all the $k$-tuples that we write along the sequence are $M$-big.
\end{lem}
\begin{proof}
We choose an arbitrary positive integer $C$ which is much bigger than all the components of $w_2,\dots,w_k,v_1,\dots,v_h$ and than all of $\lambda_2,\dots,\lambda_k,\mu_1,\dots,\mu_h$. We proceed as follows
$$(w_1,w_2,\dots,w_k)\xrar{*} (w_1+Cw_2,w_2,\dots,w_k)\xrar{*}(w_1'+Cw_2,w_2,\dots,w_k)\xrar{*} (w_1',w_2,\dots,w_k)$$
where $\xrar{*}$ denotes a finite sequence of Nielsen moves and relative Nielsen moves. It is evident that the first and the third $\xrar{*}$ can be performed in such a way that every $k$-tuple that we write along the sequence is $M$-big. All the $k$-tuples that we write during the second $\xrar{*}$ are $M$-big, provided that $C$ had been chosen big enough.
\end{proof}

\begin{prop}\label{positive-vector}
Let $k\ge 1$ and let $(w_1,\dots,w_k),(w_1',\dots,w_k')$ be $M$-big ordered $k$-tuples such that $A(w_1,\dots,w_k;v_1,\dots,v_h)=A(w_1',\dots,w_k';v_1,\dots,v_h)=A$. Suppose there is a vector in $\gen{v_1,\dots,v_h}$ such that all of its components are strictly positive. Then the following are equivalent:
\begin{enumerate}
\item\label{pv1} There is a sequence of Nielsen moves and relative Nielsen moves going from $(w_1,\dots,w_k)$ to $(w_1',\dots,w_k')$ such that all the $k$-tuples that we write along the sequence are $M$-big.
\item\label{pv2} The ordered $k$-uples of generators $([w_1],\dots,[w_k]),([w_1'],\dots,[w_k'])$ are Nielsen equivalent in $A$.
\end{enumerate}
\end{prop}
\begin{proof}
\ref{pv1} $\Rar$ \ref{pv2}. Trivial.

\ref{pv2} $\Rar$ \ref{pv1}. Let $v$ be a vector in $\gen{v_1,\dots,v_h}$ with all components strictly positive. For every element $[w]\in A$ we can construct a lifting $w\in\gen{w_1,\dots,w_k,v_1,\dots,v_h}$ such that all the components of $w$ are $\ge M$; in fact, we can just pick any lifting and add $v$ repeatedly until the components become all $\ge M$.

Suppose there is a sequence of Nielsen moves going from $([w_1],\dots,[w_k])$ to $([w_1'],\dots,[w_k'])$ in $A$. Let $([w_1^r],\dots,[w_k^r])$ be such a sequence, for $r=1,\dots,R$, so that $[w_i^1]=[w_i]$ and $[w_i^R]=[w_i']$ for $i=1,\dots,k$. Let $w_i^r$ be an $M$-big lifting of $[w_i^r]$ for $i=1,\dots,k$ and $r=1,\dots,R$, and we can take $w_i^1=w_i$ and $w_i^R=w_i'$ for $i=1,\dots,k$. By Lemma \ref{can-slide}, it is possible to pass from $(w_1^r,\dots,w_k^r)$ to $(w_1^{r+1},\dots,w_k^{r+1})$ with a sequence of Nielsen moves and relative Nielsen moves. The conclusion follows.
\end{proof}

\subsection*{Nielsen equivalence for three or more big vectors}

Fix $n\ge1$ and vectors $v_1,\dots,v_h$ in $\bbZ^n$; fix also an integer $M\ge1$. This section is dedicated to the proof of Theorem \ref{Nielsen-equiv-3big}. The idea is to try and apply Proposition \ref{positive-vector}. When we do not have a positive vector in $\gen{v_1,\dots,v_h}$, we essentially use one of the $w_i$s instead. This means that, given two $M$-big ordered $k$-tuples $(w_1,\dots,w_k),(w_1',\dots,w_k')$, we want to manipulate them until they end up with a common vector. The following Proposition \ref{finding-common-vector} explains how to manipulate one of the two $k$-tuples in order to get a vector ``almost'' in common with the other, up to an integer multiple; the next Propositions \ref{finding-primitive1} and \ref{finding-primitive2} are aimed at removing the integer multiple, and to make the vector ``really'' in common to the $k$-tuples.

\begin{prop}\label{finding-common-vector}
Let $k\ge 3$ and let $(w_1,\dots,w_k),(w_1',\dots,w_k')$ be $M$-big ordered $k$-tuples such that $A(w_1,\dots,w_k;v_1,\dots,v_h)=A(w_1',\dots,w_k';v_1,\dots,v_h)=A$. Then there is an $M$-big ordered $k$-tuple $(w_1'',\dots,w_k'')$ such that:
\begin{enumerate}
\item There is a sequence of Nielsen moves and relative Nielsen moves going from $(w_1,\dots,w_k)$ to $(w_1'',\dots,w_k'')$ such that all the $k$-tuples that we write along the sequence are $M$-big.
\item We have $w_1'=d(w_1''-w_2'')+\mu_1v_1+\dots+\mu_hv_h$ for some $d,\mu_1,\dots,\mu_h\in\bbZ$.
\end{enumerate}
\end{prop}
\begin{proof}
Let $w_1'=\lambda_1w_1+\dots+\lambda_kw_k+\mu_1v_1+\dots+\mu_hv_h$ for $\lambda_1,\dots,\lambda_k,\mu_1,\dots,\mu_h\in\bbZ$. We notice that for $i\not=j$ we have the identity
$$\lambda_iw_i+\lambda_jw_j=(\lambda_i-\lambda_j)w_i+\lambda_j(w_j+w_i).$$
For an ordered $k$-tuple of integers $(\lambda_1,\dots,\lambda_k)$ we define a \textit{subtraction} to be any operation that consists of substituting $\lambda_i$ with $\lambda_i-\lambda_j$ for some $i\not=j$. We can perform any sequence of subtractions to the $k$-tuple of integers $(\lambda_1,\dots,\lambda_k)$, as these are achieved by Nielsen moves on the $k$-tuple of vectors $(w_1,\dots,w_k)$ that keep it $M$-big.

We now work on the ordered $k$-tuple of integers $(\lambda_1,\dots,\lambda_k)$. As $k\ge 3$, it is easy to see that using subtractions we can assume that $\lambda_k<0$. We now subtract $\lambda_k$ from each of $\lambda_1,\dots,\lambda_{k-1}$ sufficiently many times, in order to get $\lambda_1,\dots,\lambda_{k-1}>0$. By Lemma \ref{mcd-concentrate}, we can subtract $\lambda_k$ from $\lambda_1$ the right amount of times, in order to get $\mcd{\lambda_1,\dots,\lambda_{k-1}}=\mcd{\lambda_1,\dots,\lambda_k}$. We now use subtractions on the non-negative integers $\lambda_1,\dots,\lambda_{k-1}$ in order to perform the Euclidean algorithm (and here we need $k\ge3$): we reduce to a situation of the type $d,0,\dots,0,-cd$ for some integers $c,d\ge1$. From this we obtain $d,-d,0,\dots,0,-cd$ and finally $d,-d,0,\dots,0$. The conclusion follows.
\end{proof}

\begin{lem}\label{mcd-concentrate}
Let $a,b,c\in\bbZ$ with $b\not=0$. Then there is $\lambda\in\bbZ$ such that $$\mcd{a,b,c}=\mcd{a+\lambda c,b}.$$
Moreover, the same holds with $\lambda'$ instead of $\lambda$, where $\lambda'$ is any element of the coset $\lambda+b\bbZ$.
\end{lem}
\begin{proof}
Let $p\divides b$ be a prime. If $v_p(c)>v_p(a)$ then we have $v_p(a+\lambda c)=v_p(a)=\min\{v_p(a),v_p(c)\}$. If $v_p(c)<v_p(a)$ then we impose the condition $\lambda\not\equiv 0$ mod $p$; this implies that $v_p(a+\lambda c)=v_p(c)=\min\{v_p(a),v_p(c)\}$. If $v_p(c)=v_p(a)$ then we write $a=p^{v_p(a)}a'$ and $c=p^{v_p(c)}c'$ with $a',c'$ coprime with $p$, and we impose the condition $\lambda\not\equiv-a'/c'$ mod $p$; this implies that $v_p(a+\lambda c)=v_p(a)=v_p(c)$.

By the Chinese reminder theorem, we can find $\lambda$ satisfying the above conditions for all primes $p\divides b$. Now for every prime $p$ that divides $b$ we have $v_p(\mcd{a+\lambda c,b})=\min\{v_p(a+\lambda c),v_p(b)\}=\min\{v_p(a),v_p(b),v_p(c)\}=v_p(\mcd{a,b,c})$, and for every prime $p$ that does not divide $b$ we have $v_p(\mcd{a+\lambda c,b})=0=v_p(\mcd{a,b,c})$. It follows that $\mcd{a+\lambda c,b}=\mcd{a,b,c}$, as desired.
\end{proof}

\begin{prop}\label{finding-primitive1}
Let $k\ge 2$ and let $(w_1,\dots,w_k)$ be an $M$-big ordered $k$-tuple. Suppose $A=A(w_1,\dots,w_k;v_1,\dots,v_h)\cong\bbZ\oplus\bbZ\oplus A'$ for some finitely generated abelian group $A'$. Then there is an $M$-big ordered $k$-tuple $(w_1',\dots,w_k')$ such that:
\begin{enumerate}
\item There is a sequence of Nielsen moves and relative Nielsen moves going from $(w_1,\dots,w_k)$ to $(w_1',\dots,w_k')$ such that each $k$-tuple that we write along the sequence is $M$-big.
\item The first two components of $[w_1']\in A\cong\bbZ\oplus\bbZ\oplus A'$ are coprime, and $[w_1']$ is not a proper power.
\end{enumerate}
\end{prop}
\begin{proof}
Let $[w_i]$ be equal to the vector $(\alpha_i,\beta_i,\gamma_i)$ through the isomorphism $A\cong\bbZ\oplus\bbZ\oplus A'$, for $i=1,\dots,k$ and $\alpha_i,\beta_i\in\bbZ$ and $\gamma_i\in A'$. Since $[w_1],\dots,[w_k]$ generate $A$, we must have $\alpha_i\beta_j-\alpha_j\beta_i\not=0$ for some $i,j\in\{1,\dots,k\}$. By Lemma \ref{can-reorder}, we can reorder the vectors with an arbitrary even permutation. Thus we can assume that
$$\alpha_1\beta_2-\alpha_2\beta_1\not=0.$$
If $\alpha_2=0$ then $\alpha_1\not=0$ (since $\alpha_1\beta_2-\alpha_2\beta_1\not=0$) and we substitute $w_2$ with $w_2+w_1$. Thus we can assume that
$$\alpha_2\not=0.$$
By Lemma \ref{mcd-concentrate}, we find $\lambda\ge0$ such that $\mcd{\alpha_1+\lambda\alpha_3,\alpha_2}=\mcd{\alpha_1,\alpha_2,\alpha_3}$. We substitute $w_1$ with $w_1+\lambda w_3$ and thus we can assume $\mcd{\alpha_1,\alpha_2}=\mcd{\alpha_1,\alpha_2,\alpha_3}$. Since $\lambda$ can be chosen modulo $\alpha_2$, we can also take $\lambda$ in such a way that $\alpha_1\beta_2-\alpha_2\beta_1$ remains $\not=0$. Reiterating the same reasoning, we obtain $\mcd{\alpha_1,\alpha_2}=\mcd{\alpha_1,\dots,\alpha_k}$, which is $1$ since $[w_1],\dots,[w_k]$ generate $A$. Thus we can assume that
$$\mcd{\alpha_1,\alpha_2}=1.$$
By Lemma \ref{mcd-concentrate}, we find $C\ge0$ such that $\mcd{\alpha_1+C\alpha_2,\alpha_1\beta_2-\alpha_2\beta_1}=\mcd{\alpha_1,\alpha_2,\alpha_1\beta_2-\alpha_2\beta_1}=1$. We now have that
$$
\begin{array}{l}
\mcd{\alpha_1+C\alpha_2,\alpha_1\beta_2-\alpha_2\beta_1-\beta_2(\alpha_1+C\alpha_2)}=1 \\
\mcd{\alpha_1+C\alpha_2,-\alpha_2(\beta_1+C\beta_2)}=1 \\
\mcd{\alpha_1+C\alpha_2,\beta_1+C\beta_2}=1
\end{array}
$$
We substitute $w_1$ with $w_1+Cw_2$ and thus obtain that $(\alpha_1,\beta_1)=1$. This also implies that the element $[w_1]$ can not be a proper power in $A$.
\end{proof}

\begin{prop}\label{finding-primitive2}
Let $k\ge 3$ and let $(w_1,\dots,w_k)$ be an $M$-big ordered $k$-tuple. Suppose that $A=A(w_1,\dots,w_k;v_1,\dots,v_h)\cong\bbZ\oplus\bbZ/d_2\bbZ\oplus A'$ for some $d_2\ge 2$ and $A'$ finitely generated abelian group. Then there is an $M$-big ordered $k$-tuple $(w_1',\dots,w_k')$ such that:
\begin{enumerate}
\item There is a sequence of Nielsen moves and relative Nielsen moves going from $(w_1,\dots,w_k)$ to $(w_1',\dots,w_k')$ such that each $k$-tuple that we write along the sequence is $M$-big.
\item The first two components of $[w_1']\in A\cong\bbZ\oplus\bbZ/d_2\bbZ\oplus A'$ are coprime, and $[w_1']$ is not a proper power.
\end{enumerate}
\end{prop}
\begin{proof}
Let $[w_i]$ be equal to the vector $(\alpha_i,\beta_i,\gamma_i)$ through the isomorphism $A\cong\bbZ\oplus\bbZ/d_2\bbZ\oplus A'$, for $i=1,\dots,k$ and $\alpha_i\in\bbZ$ and $\beta_i\in\bbZ/d_2\bbZ$ and $\gamma_i\in A'$. Let also $p$ be a prime that divides $d_2$.

Since $[w_1],\dots,[w_k]$ generate $A$, we must have $\alpha_i\beta_j-\alpha_j\beta_i\not\equiv 0\mod p$ for some $i,j\in\{1,\dots,k\}$. By Lemma \ref{can-reorder}, we can reorder the vectors with an arbitrary even permutation. Thus we can assume
$$\alpha_1\beta_2-\alpha_2\beta_1\not\equiv 0\mod p.$$
If $\alpha_2\equiv 0\mod p$ then $\alpha_1\not\equiv 0\mod p$ (since $\alpha_1\beta_2-\alpha_2\beta_1\not\equiv 0\mod p$) and we substitute $w_2$ with $w_2+w_1$. Thus we can assume
$$\alpha_2\not\equiv 0\mod p.$$
By Lemma \ref{mcd-concentrate}, we find $\lambda\ge0$ such that $\mcd{\alpha_1+\lambda\alpha_3,\alpha_2}=\mcd{\alpha_1,\alpha_2,\alpha_3}$. We substitute $w_1$ with $w_1+\lambda w_3$ and thus we can assume $\mcd{\alpha_1,\alpha_2}=\mcd{\alpha_1,\alpha_2,\alpha_3}$. Since $\lambda$ can be chosen modulo $\alpha_2$, which is coprime with $p$, we can also take $\lambda$ to be a multiple of $p$, so that $\alpha_1\beta_2-\alpha_2\beta_1 \mod p$ remains the same. Reiterating the same reasoning, we obtain $\mcd{\alpha_1,\alpha_2}=\mcd{\alpha_1,\dots,\alpha_k}$, which is $1$ since $[w_1],\dots,[w_k]$ generate $A$. Thus we can assume
$$\mcd{\alpha_1,\alpha_2}=1.$$
By the Chinese remainder theorem we find $C\ge0$ such that $\alpha_3+C\alpha_1\equiv 1 \mod \alpha_2$ and $(\alpha_3\beta_2-\alpha_2\beta_3)+C(\alpha_1\beta_2-\alpha_2\beta_1)\not\equiv 0\mod p$; here we are using the facts that $\mcd{\alpha_2,p}=1$ and $\mcd{\alpha_1,\alpha_2}=1$ and $\mcd{\alpha_1\beta_2-\alpha_2\beta_1,p}=1$. Since $\mcd{\alpha_2,p}=1$ we can find integers $e\ge0$ as big as we want and such that $p^e\equiv 1\mod \alpha_2$. Using the fact that $\alpha_3+C\alpha_1\equiv 1 \mod \alpha_2$ we are thus able to find integers $D\ge0$ and $e\ge0$ such that $\alpha_3+C\alpha_1+D\alpha_2=p^e$. We now have that
$$
\begin{array}{l}
\mcd{p,(\alpha_3\beta_2-\alpha_2\beta_3)+C(\alpha_1\beta_2-\alpha_2\beta_1)}=1,\\
\mcd{p,(\alpha_3\beta_2-\alpha_2\beta_3)+C(\alpha_1\beta_2-\alpha_2\beta_1)-\beta_2(\alpha_3+C\alpha_1+D\alpha_2)}=1,\\
\mcd{p,-\alpha_2(\beta_3+C\beta_1+D\beta_2)}=1,\\
\mcd{p,\beta_3+C\beta_1+D\beta_2}=1.
\end{array}
$$
We substitute $w_3$ with $w_3+Cw_1+Dw_2$ and thus we can assume that $\alpha_3=p^e$ and $(p,\beta_3)=1$. In particular $(\alpha_3,\beta_3)=1$. The element $[w_3]$ can be a proper power in $A$: if $[w_3]=d[w]$, then $\alpha_3=p^e$ implies that $d\divides p^e$, and $(p,\beta_3)=1$ implies that $(p,d)=1$; it follows that $d=\pm1$. To conclude, we permute $(w_1,w_2,w_3)$ using Lemma \ref{can-reorder} in order to bring $w_3$ in first position.
\end{proof}

\begin{prop}\label{case-Z}
Let $k\ge 3$ and let $(w_1,\dots,w_k),(w_1',\dots,w_k')$ be $M$-big ordered $k$-tuples such that $A(w_1,\dots,w_k;v_1,\dots,v_h)=A(w_1',\dots,w_k';v_1,\dots,v_h)=A\cong\bbZ$. Then there is a sequence of Nielsen moves and relative Nielsen moves from $(w_1,\dots,w_k)$ to $(w_1',\dots,w_k')$ such that every $k$-tuple that we write along the sequence is $M$-big.
\end{prop}
\begin{proof}
Let $[w_i],[w_i']$ be equal to $\alpha_i,\alpha_i'$ through the isomorphism $A\cong\bbZ$, for $i=1,\dots,k$ and $\alpha_i,\alpha_i'\in\bbZ$. If among $\alpha_i,\alpha_i'$ there is at least one integer $>0$ and at least one integer $<0$, then $\gen{v_1,\dots,v_h}$ has to contain a vector whose components are all strictly positive, and thus we are done by Proposition \ref{positive-vector}. Thus we assume $\alpha_i,\alpha_i'>0$ for all $i=1,\dots,k$.

By Lemma \ref{mcd-concentrate}, we find $\lambda\ge0$ such that $\mcd{\alpha_1+\lambda\alpha_3,\alpha_2}=\mcd{\alpha_1,\alpha_2,\alpha_3}$. We substitute $w_1$ with $w_1+\lambda w_3$ and thus we can assume $\mcd{\alpha_1,\alpha_2}=\mcd{\alpha_1,\alpha_2,\alpha_3}$. By reiterating the same reasoning, we can assume $\mcd{\alpha_1,\alpha_2}=\mcd{\alpha_1,\dots,\alpha_k}=1$. Now, by substituting $w_3$ with $w_3+\lambda_1w_1+\lambda_2w_2$ for $\lambda_1,\lambda_2\ge0$, we can set $\alpha_3$ to be equal to any sufficiently large natural number. Similarly, we can assume $\mcd{\alpha_1',\alpha_2'}=1$, and we are able to set $\alpha_3'$ to be any sufficiently large natural number. We choose to set $\alpha_3=\alpha_3'$ and such that $\mcd{\alpha_1,\alpha_3}=\mcd{\alpha_1',\alpha_3}=1$. Now, by substituting $w_2$ with $w_2+\mu_1w_1+\mu_3w_3$ for $\mu_1,\mu_3\ge0$, we can set $\alpha_2$ to be any sufficiently large natural number. Similarly, we are able to set $\alpha_2'$ to be any big enough natural number. We choose to set $\alpha_2=\alpha_2'$ and such that $\mcd{\alpha_2,\alpha_3}=1$. Reasoning in the same way, and keeping  $w_2,w_3,w_2',w_3'$ fixed, we can set $\alpha_1=\alpha_1'$ and $\alpha_i=\alpha_i'$ for $i\ge4$.

To conclude, for $i=1,\dots,k$, we observe that $\alpha_i=\alpha_i'$ implies that $w_i'=w_i+\mu_1v_1+\dots+\mu_hv_h$ for some $\mu_1,\dots,\mu_h\in\bbZ$ and thus by Lemma \ref{can-slide}, we can obtain $w_i=w_i'$ by  Nielsen moves and relative Nielsen moves. The statement follows.
\end{proof}

\begin{proof}[Proof of Theorem \ref{Nielsen-equiv-3big}]
\ref{neb1} $\Rar$ \ref{neb2}. Trivial.

\ref{neb2} $\Rar$ \ref{neb1}. If $A$ is a finite abelian group, then we must have $d[w_1]=0$ in $A$ for some $d\ge 1$, and this implies that $\gen{v_1,\dots,v_h}$ contains the vector $dw_1$, which has all the components strictly positive; the conclusion follows from Proposition \ref{positive-vector}. If $A$ is isomorphic to $\bbZ$, then we are done by Proposition \ref{case-Z}.

Otherwise we write $A\cong\bbZ\oplus\bbZ/d_2\bbZ\oplus\ldots \oplus\bbZ/d_m\bbZ$ for integers $m\ge2$ and $d_2,\dots,d_m\ge0$ with $1\not=d_m\divides d_{m-1}\divides\dots\divides d_2$. By either Proposition \ref{finding-primitive1} or Proposition \ref{finding-primitive2}, we are able to change $(w_1',\dots,w_k')$ by means of Nielsen moves and relative Nielsen moves in such a way that $[w_1']$ is not a proper power in $A$, and that $[w_1']$ can be represented by a vector whose first two components are coprime. We now apply Proposition \ref{finding-common-vector} and we apply a sequence of Nielsen moves and relative Nielsen moves to $(w_1,\dots,w_k)$, in order to obtain that $w_1'=w_1-w_2+\mu_1v_1+\dots+\mu_hv_h$. By Proposition \ref{can-slide} we can easily change $w_1$ in such a way that $w_1=w_1'$.

If $m\ge3$ then we can apply Lemma \ref{Nielsen-equiv-rel2}, since $[w_1']$ can be represented by a vector whose first two components are coprime, and we obtain that $([w_2],\dots,[w_k]),([w_2'],\dots,[w_k'])$ are Nielsen equivalent $(k-1)$-tuples of generators for $A/\gen{[w_1]}$. If $m=2$ then $[w_1']$ can be represented by a vector whose components are coprime, and thus we can apply Lemma \ref{Nielsen-equiv-rel1} to obtain that $([w_2],\dots,[w_k]),([w_2'],\dots,[w_k'])$ are Nielsen equivalent $(k-1)$-tuples of generators for $A/\gen{[w_1]}$. In both cases, the conclusion follows from Proposition \ref{positive-vector}.
\end{proof}

\end{document}